\documentclass{amsart}
\usepackage{amssymb, amsmath, latexsym, amscd, graphicx, color}
\usepackage[all]{xy}
\usepackage[dvipdfm]{hyperref}
\usepackage[all]{hypcap}
\newtheorem{theorem}{Theorem}[section]
\newtheorem{lemma}[theorem]{Lemma}
\newtheorem{cor}[theorem]{Corollary}
\newtheorem{definition}[theorem]{Definition}
\newtheorem{example}[theorem]{Example}
\newtheorem{remark}[theorem]{Remark}
\newtheorem{proposition}[theorem]{Proposition}
\def\pagenumber{1}

\begin{document}
\setcounter{page}{\pagenumber}
\newcommand{\T}{\mathbb{T}}
\newcommand{\R}{\mathbb{R}}
\newcommand{\Q}{\mathbb{Q}}
\newcommand{\N}{\mathbb{N}}
\newcommand{\Z}{\mathbb{Z}}
\newcommand{\tx}[1]{\quad\mbox{#1}\quad}
\parindent=0pt
\def\SRA{\hskip 2pt\hbox{$\joinrel\mathrel\circ\joinrel\to$}}
\def\tbox{\hskip 1pt\frame{\vbox{\vbox{\hbox{\boldmath$\scriptstyle\times$}}}}\hskip 2pt}
\def\circvert{\vbox{\hbox to 8.9pt{$\mid$\hskip -3.6pt $\circ$}}}
\def\IM{\hbox{\rm im}\hskip 2pt}
\def\ES{\vbox{\hbox to 8.9pt{$\big/$\hskip -7.6pt $\bigcirc$\hfil}}}
\def\TR{\hbox{\rm tr}\hskip 2pt}
\def\GRAD{\hbox{\rm grad}\hskip 2pt}
\def\bull{\vrule height .9ex width .8ex depth -.1ex}
\def\VLLA{\hbox to 25pt{\leftarrowfill}}
\def\VLRA{\hbox to 25pt{\rightarrowfill}}
\def\sqcup{\mathop{\bigcup}\limits^{.}}
\setbox2=\hbox to 25pt{\rightarrowfill}
\def\DRA{\vcenter{\copy2\nointerlineskip\copy2}}
\def\RANK{\hbox{\rm rank}\hskip 2pt}
\font\rsmall=cmr7 at 7truept \font\bsmall=cmbx7 at 7truept
\newfam\Slfam
\font\tenSl=cmti10 \font\neinSl=cmti9 \font\eightSl=cmti8
\font\sevenSl=cmti7 \textfont\Slfam=\tenSl
\scriptfont\Slfam=\eightSl \scriptscriptfont\Slfam=\sevenSl
\def\Sl{\fam\Slfam\tenSl}

\def\CODIM{\hbox{\rm codim}\hskip 2pt}
\def\CODIMS{\hbox{\rsmall codim}\hskip 2pt}
\def\SRA{\hskip 2pt\hbox{$\joinrel\mathrel\circ\joinrel\to$}}
\def\tbox{\hskip 1pt\frame{\vbox{\vbox{\hbox{\boldmath$\scriptstyle\times$}}}}\hskip 2pt}
\def\circvert{\vbox{\hbox to 8.9pt{$\mid$\hskip -3.6pt $\circ$}}}
\def\IM{\hbox{\rm im}\hskip 2pt}
\def\ES{\vbox{\hbox to 8.9pt{$\big/$\hskip -7.6pt $\bigcirc$\hfil}}}
\def\TR{\hbox{\rm tr}\hskip 2pt}
\def\TRS{\hbox{\rsmall tr}\hskip 2pt}
\def\DIV{\hbox{\rm div}\hskip 2pt}
\def\DIVS{\hbox{\rsmall div}\hskip 2pt}
\def\GRAD{\hbox{\rm grad}\hskip 2pt}
\def\GRADS{\hbox{\rsmall grad}\hskip 2pt}
\def\bull{\vrule height .9ex width .8ex depth -.1ex}
\def\VLLA{\hbox to 25pt{\leftarrowfill}}
\def\VLRA{\hbox to 25pt{\rightarrowfill}}
\setbox2=\hbox to 25pt{\rightarrowfill}
\def\DRA{\vcenter{\copy2\nointerlineskip\copy2}}
\def\RANK{\hbox{\rm rank}\hskip 2pt}
\def\ROT{\hbox{\rm rot}\hskip 2pt}
\def\ROTS{\hbox{\rsmall rot}\hskip 2pt}

\vskip -1cm

\title[Quantum Extended Crystal PDE's]{\mbox{}\\[1cm] QUANTUM EXTENDED CRYSTAL PDE's}
\author{Agostino Pr\'astaro}
\maketitle
\vspace{-.5cm}

{\footnotesize
\begin{center}
Department SBAI - Mathematics, University of Rome ''La Sapienza'', Via A.Scarpa 16,
00161 Rome, Italy. \\
E-mail: {\tt agostino.prastaro@uniroma1.it; prastaro@dmmm.uniroma1.it}
\end{center}
\vspace{.5cm}

\vspace{.5cm} {\bsmall ABSTRACT.} Our recent results on {\em
extended crystal PDE's} are generalized to PDE's in the category
$\mathfrak{Q}_S$ of quantum supermanifolds. Then obstructions to the
existence of global quantum smooth solutions for such equations are
obtained, by using algebraic topologic techniques. Applications are
considered in details to the quantum super Yang-Mills equations.
Furthermore, our geometric theory of stability of PDE's and their
solutions, is also generalized to quantum extended crystal PDE's. In
this way we are able to identify quantum equations where their
global solutions are stable at finite times. These results, are also
extended to quantum singular (super)PDE's, introducing ({\em quantum
extended crystal singular (super) PDE's}). \footnote{See also companion paper \cite{PRA33}.\\ Work partially
supported by Italian grants MIUR ''PDE's Geometry and
Applications''.}}

{\bsmall AMS (MOS) MS CLASSIFICATION. 57R90, 53C99, 81Q99.}

{\rsmall KEY WORDS AND PHRASES. Integral bordisms in quantum PDE's.
Quantum (super)gravity. Stability. Extended crystal structures.}

\section{\bf Introduction}
In a previous paper \cite{PRA25} we proved that PDE's can be
considered as extended crystals, in the sense that their integral
bordism groups can be seen as crystallographic subgroups extensions.
In this paper we aim generalize that result to quantum super PDE's.
This is possible, since we utilize our geometric theory of PDE's
considered in the category $\mathfrak{Q}$ of quantum manifolds and
in the category $\mathfrak{Q}_S$ of quantum supermanifolds
\cite{PRA15, PRA16, PRA20, PRA21, PRA22, PRA23, PRA30, PRA31,
PRA32}. Then we relate integral bordism groups of quantum super
PDE's to crystallographic groups. The main results are Theorem
\ref{crystal-structure-quantum-super-pdes} and Theorem
\ref{obstruction-smooth-solutions}. The first relates formal
integrability and complete integrability of quantum PDE's to
crystallographic groups. In this way we can consider quantum super
PDE's as {\em quantum extended crystallographic structures}. In the
second theorem, we identify an obstruction characterizing existence
of global quantum smooth solutions. This is called {\em quantum
crystal obstruction} of a quantum super PDE. Applications to quantum
super Yang-Mills PDE's are given too.

Another characterizing aspect of this paper is a new geometric
theory for stability of quantum super PDE's and their solutions.
This is made by extending to the category of quantum super manifolds
$\mathfrak{Q}_S$ our previous geometric approach on commutative
PDE's stability \cite{PRA24, PRA25, PRA26, PRA27, PRA28, PRA29}.
Here a $k$-order quantum (super) PDE is considered as a subset $\hat
E_k\subset J\hat D^k(W)$ of the $k$-jet-derivative space $ J\hat
D^k(W)$, built on some fiber bundle $\pi:W\to M$, in the category of
quantum smooth (super)manifolds. Then, to investigate the stability
of a regular solution $D^ks(M)\equiv V\subset \hat E_k$, of $\hat
E_k\subset J\hat D^k(W)$, one considers the linearization of $\hat
E_k$ at the solution $V$. The integrable solutions of the linearized
equation $(D^ks)^*vT\hat E_k\equiv \hat E_k[s]\subset J\hat
D^k(s^*vTW)\equiv J\hat D^k(\hat E[s])$, represent the infinitesimal
admissible perturbations of the original solution. Then, if to an
initial Cauchy data for such linearized equation there correspond
solutions (perturbations) that or oscillate around the zero solution
(of the linearized equation), or remain limited around such zero
solution, then the solution $s$ is said to be stable, otherwise $s$
is called unstable. Taking into account that the linearized equation
$\hat E_k[s]$ belongs to a vector neighborhood of $\hat E_k$, at
the solution $V$, and that integrable solutions of $\hat E_k[s]$ are
infinitesimal vertical symmetries of $\hat E_k$, it follows that
such perturbations deform the original solution $V\subset \hat E_k$
into solutions $\widetilde{V}\subset \hat E_k$ such that, if $V$ is
stable, remain into suitable neighborhoods of the same $V$. When,
instead the perturbations blow-up, then $V$ is unstable. The
blowing-up of the perturbation corresponds to the fact that such a
solution of the linearized equation $\hat E_k[s]$ is not regular in
all of its points, but there are present singular points. Then in
the cases where $V$ is unstable, between the solutions of above type
$\widetilde{V}$, there are ones that are also singular and this fact
just characterizes unstable solutions of $\hat E_k$. This approach
to the stability can be related to the Ljapunov concept of stability
in functional analysis \cite{LJA}, and it is founded on the
assumption that the possible perturbations can influence only the
given solution, say $V\subset \hat E_k$, but do not have any
influence on the same equation $\hat E_k$.

On the other hand, we can more generally assume that perturbations
can change the same original equation. In such a case we can ask
wether a given solution of the original equation can change for
''little'' perturbations of the same equation. Then we talk about
(un)stable equations. This last approach is, instead, related to the
concept of Ulam (un)stability for functional equations \cite{ULA}.

We prove that all above points of view for stability in quantum
(super) PDE's can be unified in the geometric theory of quantum
(super) PDE's on the ground of integral bordism groups. This extends
to the category of quantum super PDE's, our previous results on the
stability of commutative PDE's \cite{PRA24, PRA25, PRA26, PRA27,
PRA28, PRA29}.\footnote{For basic informations on the geometry of
PDE's see also the following refs.\cite{B-C-G-G-G, GOL, GOL-SPE,
GRO, L-P, PRA2, PRA3, PRA4, PRA5, PRA6, PRA7, PRA8, PRA9, PRA10,
PRA11, PRA12, PRA13, PRA14, PRA15, PRA16, PRA17, PRA18, PRA19}. For
basic informations on some subjects of differential topology and
algebraic topology, related to this paper, see also
refs.\cite{C-R-R, HIR, M-M, M-S, RUDY, STO, SULL, SWI, WAL1, WAL2}.}

 In this paper, the main results on the quantum super PDE's stability are the following. Theorem
 \ref{criteria-fun-stab} that gives some criteria to recognize
 functional stability in any quantum (super) PDE's. Theorem
 \ref{Criterion-fun-stable-sol-conn} that relates functional
 stability with quantum $(k+1)$-connections.
 Theorem \ref{finite-stable-extended-crystal-PDE} proving that to a formally integrable and completely integrable
 quantun (super) PDE, one can canonically associate another quantum (super) PDE ${}^{(S)}\hat E_k$,
 {\em stable quantum extended crystal (super) PDE} of $\hat E_k$, having the same regular smooth solutions
 of $\hat E_k$, but in ${}^{(S)}\hat E_k$ these solution
 are stable.\footnote{This theorem allows to avoid all the problems present in the applications, related to
 finite instability of solutions.} Theorem
 \ref{criterion-average-asymptotic-stability} that gives a criterion
 to recognize the average asymptotic stability with respect to
 quantum frames. Applications to quantum super d'Alembert equation
 and quantum super Navier-Stokes equation are considered too.

Finally we extend above results also to quantum singular (super)
PDE's, and we characterize {\em quantum extended crystal singular
(super) PDE's}. For such equations we identify algebraic-topological
obstructions to the existence of global (smooth) solutions solving
boundary value problems and crossing singular points too.

The paper, after the Introduction, contains two more sections. In
the first section we relate the integral bordism groups of quantum
super PDE's to crystallographic groups and recognize a topologic
algebraic obstruction to the existence of quantum smooth solutions.
Applications of these results to the quantum super Yang-Mills
equation are considered. In Section 3 we formulate a geometric
theory of stability for solutions of quantum super PDE's. We follow
some our previous works devoted to the algebraic topological
characterization of PDE's stability and their solutions satibility
\cite{PRA24, PRA25, PRA26, PRA27, PRA28, PRA29}. Thus, in this paper
the stability of quantum (super) PDE's is studied in the framework
of the geometric theory of quantum (super) PDE's, and in the
framework of the bordism groups of quantum (super) PDE's. In
particular we identify criteria to recognize quantum (super) PDE's
that are stable (in extended Ulam sense) and in their regular smooth
solutions do not occur unstabilities in finite times. We call such
equations {\em stable quantum extended crystal (super) PDE's}.
Applications to the quantum super d'Alembert equation and the
quantum super Navier-Stokes equation respectively are explicitly
considered. Section 4 is devoted to extend above results also to
quantum singular super PDE's. The main results in this section is
Theorem \ref{main-quantum-singular1} that identifies conditions in
order to recognize global (smooth) solutions of quantum singular
super PDE's crossing singular points. There we characterize {\em
quantum $0$-crystal singular super PDE's}, i.e., quantum singular
super PDE's having smooth global solutions crossing singular points,
stable at finite times.

\section{\bf INTEGRAL BORDISM GROUPS OF QUANTUM SUPER PDE's vs CRYSTALLOGRAPHIC GROUPS}
\vskip 0.5cm

In this section we extend to quantum super PDE's our previous
results on the algebraic topological crystal characterization of
commutative PDE's.\footnote{Quantum super PDE's are PDE's in the
category $\mathfrak{Q}_S$ of quantum supermanifolds, in the sense
introduced by A.Pr\'astaro \cite{PRA7, PRA8, PRA11, PRA12, PRA13,
PRA14, PRA18, PRA19, PRA20, PRA21, PRA28, PRA32}.}

\begin{remark}
Here and in the following we shall denote the boundary $\partial V$
of a compact quantum supermanifold $V$, of dimension $m|n$, with
respect to a quantum superalgebra $A$, split in the form $\partial
V=N_0\bigcup P\bigcup N_1$, where $N_0$ and $N_1$ are two disjoint
$(m-1|n-1)$-dimensional quantum sub-supermanifolds of $V$, that are
not necessarily closed, and $P$ is another $(m-1|n-1)$-dimensional
quantum sub-supermanifold of $V$. For example, if $V=\hat
D{}^{m|n}\times \hat D{}^{1|1}$, where $\hat D{}^{r|s}\subset \hat
S{}^{r|s}$ is the $(r|s)$-dimensional quantum superdisk, contained
in the $(r|s)$-dimensional quantum supersphere, one has that $\dim
V=(m+1|n+1)$ and  $N_0=\hat D{}^{m|n}\times\{0\}$, $N_1=\hat
D{}^{m|n}\times\{1\}$, $P=\partial \hat D{}^{m|n}\times \hat
D{}^{1|1}\cong \hat S{}^{m-1|n-1}\times \hat
D{}^{1|1}$.\footnote{Recall that a {\em$(m|n)$-dimensional quantum
supersphere}, $\hat S^{m|n}$, over a quantum superalgebra $A$, is
the Alexandrov compactification of $A^{m|n}$, i.e., $\hat
S^{m|n}=A^{m|n}\bigcup\{\infty\}$. A {\em$(m|n)$-dimensional quantum
superdisk} over a quantum superalgebra $A$, is a connected compact
sub-supermanifold  $\hat D^{m|n}\subset \hat S^{m|n}$, of dimension
$m|n$ over $A$, such that $\partial\hat D^{m|n}\cong \hat
S^{m-1|n-1}$, i.e., with boundary $\partial\hat D^{m|n}$
diffeomorphic to $\hat S^{m-1|n-1}$. For details on such quantum
supermanifolds see \cite{PRA31}.} Therefore $\dim N_0=\dim N_1=\dim
P=(m|n)$. Note that since $\hat D{}^{m|n}\times \hat D{}^{1|1}=\hat
D{}^{m+1|n+1}$, therefore we can also write $\partial V=\partial\hat
D{}^{m+1|n+1}=\hat S{}^{m|n}=\hat S{}^{m-1|n-1}\times\hat
S{}^{1|1}$. Since

\begin{equation}
\begin{array}{ll}
\partial V&=\partial(\hat D{}^{m|n}\times \hat D{}^{1|1})
      =(\partial \hat D{}^{m|n})\times \hat D{}^{1|1}\bigcup \hat
      D{}^{m|n}\times\partial\hat D{}^{1|1}\\
      &=\hat S{}^{m-1|n-1}\times \hat D{}^{1|1}\bigcup \hat
      D{}^{m|n}\times\hat S{}^{0|0}.\\
      \end{array}
    \end{equation}
Therefore, $\partial V$ is obtained by means of the quantum
surgering removing $\hat S{}^{m-1|n-1}\times \hat
D{}^{1|1}\subset\hat S{}^{m|n}$. (For details on quantum surgering
see \cite{PRA27}.) Of course if $V=\hat S{}^{m|n}\times \hat
D{}^{1|1}$, then $P=\varnothing$, hence $\partial V=\hat
S{}^{m|n}\times\{0\}\sqcup S{}^{m|n}\times\{1\}$.\footnote{We denote
disjoint union by the symbol $\sqcup$ or $\sqcup$.}

This example shows that if $V$ is a solution of a quantum super PDE,
then it can be obtained by propagating an initial Cauchy
hypersurface $X\subset V$, $\dim X=(m|n)$, by means of an integrable
full quantum vector field $\zeta:V\to \widehat{TV}\equiv
Hom_Z(A;TV)$, $\partial\phi=\zeta$, where $\phi:A\times V\to V$.
(See Fig.1.)
\begin{figure}[h]
\centerline{\includegraphics[width=5cm]{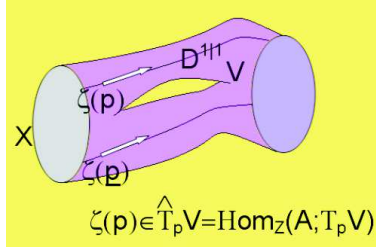}} \caption{Quantum
solution $V$, of dimension $(m+1|n+1)$ over a quantum superalgebra
$A$, propagating $X$, $\dim
X=(m|n)$.\label{quantum-supermanifold-solution}}
\end{figure}

Let us emphasize also, that in some cases solutions can be obtained
also by flows of integrable vector fields $\zeta:V\to TV$, i.e.,
$\zeta=\partial\psi$, with $\psi:\mathbb{R}\times V\to V$. For
example if $V=\hat D{}^{m|n}\times I$, with
$I\equiv[0,1]\subset\mathbb{R}$, then $\dim V=m+1|n$, and $\partial
V=M_0\bigcup P\bigcup M_1$, with $P=\hat S{}^{m-1|n-1}\times I$,
$M_0=\hat D{}^{m|n}\times\{0\}$, $M_1=\hat D{}^{m|n}\times\{1\}$. So
we get $\dim P=m|n-1$, $\dim M_0=\dim M_1=m|n$. Therefore $\partial
V$ has some components ($M_0$ and $M_1$) that have only the even
dimension dropped by $1$, with respect to $V$, and other one ($P$)
where also the odd dimension drops by $1$.

Let us also recall that with the term {\em quantum solutions} we
mean integral bordisms relating Cauchy quantum hypersurfaces of
$\hat E_{k+s}$, contained in $J^{k+s}_{m|n}(W)$, but not necessarily
contained into $\hat E_{k+s}$. (For details see refs.\cite{PRA21,
PRA22, PRA23, PRA31}.)
\end{remark}

\begin{definition}
We say that a quantum super PDE $\hat E_k\subset \hat J^k_{m|n}(W)$
is an {\em quantum extended 0-crystal super PDE}, if its weak
integral bordism group $\Omega^{\hat E_k}_{m-1|n-1,w}$ is zero.
\end{definition}

\begin{theorem}{\em(Criterion to recognize quantum extended 0-crystal super PDE's)}.\label{main4}
Let $\hat E_k\subset \hat J^k_{m|n}(W)$ be a formally quantum
integrable and completely quantum superintegrable quantum super PDE
such that $W$ is contractible. If $m-1\not=0$ and $n-1\not=0$, then
$E_k$ is a quantum extended $0$-crystal super PDE.
\end{theorem}

\begin{proof}
In fact, one has the following ismorphisms, (see \cite{PRA22}):
\begin{equation}
    \begin{array}{ll}
     \Omega^{\hat E_k}_{m-1|n-1,w}& \cong H_{m-1|n-1}(W;A) \\
      & \cong \left(A_0\bigotimes_{\mathbb{K}}H_{m-1}(W;\mathbb{K})\right)\bigoplus
\left(A_1\bigotimes_{\mathbb{K}}H_{n-1}(W;\mathbb{K})\right).
    \end{array}
\end{equation}

Thus, when $W$ is contractible, and $m-1\not=0$, $n-1\not=0$, one
has $H_{m-1}(W;\mathbb{K})=H_{n-1}(W;\mathbb{K})=0$, hence we get
$\Omega^{\hat E_k}_{m-1|n-1,w}=0$.
\end{proof}

\begin{theorem}{\em(Crystal structure of quantum super PDE's).}\label{crystal-structure-quantum-super-pdes}
Let $\hat E_k\subset \hat J^k_{m|n}(W)$ be a formally quantum
superintegrable and completely quantum superintegrable quantum super
PDE. Then its integral bordism group $\Omega_{m-1|n-1}^{\hat E_k}$
is an extension of some crystallographic subgroup $G\triangleleft
G(d)$. We call $d$ the {\em crystal dimension} of $\hat E_k$ and
$G(d)$ its {\em crystal structure} or {\em crystal group}.
\end{theorem}

\begin{proof}
Let us first note that there is a relation between lower dimensions
integral bordisms in a commutative PDE. In fact, one has the
following lemma.
\begin{lemma}{\em(Relations between lower dimensions integral bordisms in commutative PDE's).}\label{commutative-PDE-lower-order-relation-bordism}
Let $E_k\subset J^k_n(W)$ be a PDE on the fiber bundle $\pi:W\to M$,
$\dim W=m+n$, $\dim M=n$. Let ${}^{S}C_p(E_k)$ be the set of all
compact $p$-dimensional admissible integral smooth manifolds of
$E_k$. The disjoint union gives an addition on ${}^{S}C_p(E_k)$ with
$\varnothing$ as the zero element. Let us consider the homomorphisms
$\partial_p:{}^{S}C_p(E_k)\to {}^{S}C_{p-1}(E_k)$ that associates to
any element $a\in{}^{S}C_p(E_k)$ its boundary $\partial
a=\partial_p(a)$. So we obtain the following chain complex of
abelian groups {\em(integral smooth bordisms chain complex)}:
\begin{equation}\label{integral-smooth-bordism-chain-complex}
\xymatrix{{}^{S}C_{n}(E_k)\ar[r]^{\partial_n}&{}^{S}C_{n-1}(E_k)\ar[r]^{\partial_{n-1}}&{}^{S}C_{n-2}(E_k)\ar[r]^(.6){\partial_{n-2}}&
\cdots\ar[r]^(.4){\partial_1}&{}^{S}C_{0}(E_k).\\}
\end{equation}
Then the $p$-bordism groups $\Omega_p^{E_k}$, $0<p<n$, can be
represented by means of the homology of the chain complex
{\em(\ref{integral-smooth-bordism-chain-complex})}.
\end{lemma}
\begin{proof}
Let us denote by
$\left\{{}^{S}C_{\bullet}(E_k),\partial_\bullet\right\}$ the chain
complex in (\ref{integral-smooth-bordism-chain-complex}). Then, we
can build the following exact commutative diagram:
\begin{equation}\label{commutative-diagram-integral-smooth-bordism-chain-complex}
\xymatrix{&&0\ar[d]&0\ar[d]&&\\
&0\ar[r]&{}^{S}B_{\bullet}(E_k)\ar[d]\ar[r]&{}^{S}Z_{\bullet}(E_k)\ar[d]\ar[r]&{}^{S}H_{\bullet}(E_k)\ar[r]&0\\
&&{}^{S}C_{\bullet}(E_k)\ar[d]\ar@{=}[r]&{}^{S}C_{\bullet}(E_k)\ar[d]&&\\
0\ar[r]&\Omega_\bullet^{E_k}\ar[r]&{}^{S}Bor_{\bullet}(E_k)\ar[d]\ar[r]&{}^{S}Cyc_{\bullet}(E_k)\ar[d]\ar[r]&0&\\
&&0&0&&\\}
\end{equation}
where

$$\left\{\begin{array}{l}
{}^{S}B_{\bullet} (E_k)=\ker(\partial|_{\bullet});\hskip
2pt {}^{S}Z_{\bullet} (E_k)=\IM(\partial_{\bullet});\\
{}^{S}H_{\bullet} (E_k)={}^{S}Z_{\bullet} (E_k)/{}^{S}B_{\bullet} (E_k),\\
b\in[a]\in {}^{S}Bor_{\bullet}(E_k)\Rightarrow a-b=\partial c;\hskip
2pt
                 c\in {}^{S}C_{\bullet}(E_k);\hskip 2pt
b\in[a]\in {}^{S}Cyc_{\bullet}(E_k)\Rightarrow \partial(a-b)=0;\\
b\in[a]\in \Omega_{\bullet}^{E_k}\Rightarrow
                 \left\{\begin{array}{l}
                          \partial a=\partial b=0\\
                          a-b=\partial c,\quad c\in {}^{S}C_{\bullet}(E_k)\\
                          \end{array}
                                \right\}.\\
\end{array}\right.$$
Then from
(\ref{commutative-diagram-integral-smooth-bordism-chain-complex}) it
follows directly that $\Omega_{p}^{E_k}\cong {}^{S}H_p(E_k)$,
$0<p<n$.
\end{proof}

\begin{lemma}{\em(Relations between integral bordisms groups in commutative PDEs).}\label{relation-between-integral-bordisms-groups-in-commutative-PDEs}
One has the following canonical isomorphism:
\begin{equation}\label{isomorphism-relation-between-integral-bordisms-groups-in-commutative-PDEs}
    \mathbb{Z}\bigotimes_{\Omega_\bullet^{E_k}}\mathbb{Z}{}^{S}Bor_{\bullet}(E_k)\cong\mathbb{Z}{}^{S}Cyc_{\bullet}(E_k).
\end{equation}
\end{lemma}
\begin{proof}
Follows directly from the extension of groups given at the bottom of
the commutative exact diagram
(\ref{commutative-diagram-integral-smooth-bordism-chain-complex})
and some properties between extension of groups. (See, e.g.,
\cite{PRA9}.)
\end{proof}

\begin{lemma}{\em(Relations between integral bordisms groups in commutative PDEs-2).}\label{relation-between-integral-bordisms-groups-in-commutative-PDEs-2}
If $H^2({}^{S}Cyc_{\bullet}(E_k),\Omega_\bullet^{E_k})=0$ one has
the following canonical isomorphism:
\begin{equation}\label{isomorphism-relation-between-integral-bordisms-groups-in-commutative-PDEs-2}
    {}^{S}Bor_{\bullet}(E_k)\cong\Omega_\bullet^{E_k}\times{}^{S}Cyc_{\bullet}(E_k).
\end{equation}
\end{lemma}

\begin{proof}
Follows directly from the extension of groups given at the bottom of
the commutative exact diagram
(\ref{commutative-diagram-integral-smooth-bordism-chain-complex})
and some properties between extension of groups. (See, e.g.,
\cite{PRA9}.)
\end{proof}
\begin{example}
In particular if $E_k\subset J^k_n(W)$ is a $0$-crystal PDE with
$\Omega^{E_k}_{n-1}=0$, one has:
${}^{S}Bor_{n-1}(E_k)\cong{}^{S}Cyc_{n-1}(E_k)$. Such an example is,
e.g., the d'Alembert equation $uu_{xy}-u_xu_y=0$ on the trivial
fiber bundle $\pi:W\equiv\mathbb{R}^3\to M\equiv\mathbb{R}^2$,
$(x,y,u)\mapsto(x,y)$. In fact, in such a case one has
$\Omega_1^{d'A}=0$ and also
${}^{S}Bor_{1}(d'A)\cong{}^{S}Cyc_{1}(d'A)\cong 0$. (For the
integral bordism group of this equation, and its generalizations,
see Refs.\cite{PRA9, PRA17, PRA25, PRA28, PRA-RAS1}.) \end{example}

\begin{lemma}{\em(Integral ringoid of PDE).}\label{integral-ringoid-pde}
A {\em ringoid} is a structure $(A,+,\cdot)$, where $A$ is a set and
$+$ is a binary operation such that $(A,+)$ is an abelian additive
group with zero $0\in A$; $\cdot$ is a partially binary operation,
i.e., it is defined only for some couples $(a,b)\in A\times A$, such
that it is associative, and distributive with respect to $+$, i.e.,
if $a\cdot b$ and $a\cdot c$ are defined, then it is defined also
$a\cdot(b+c)=a\cdot b+a\cdot c$. A {\em graded ringoid} is a set
$A=\bigoplus_nA_n$, where each $A_n$ is an abelian additive group
and there is a partial binary operation $\cdot$, associative, and
distributive with respect to $+$, such that if $a\in A_n$, $b\in
A_m$, then $a\cdot b\in A_{m+n}$, whenever it is defined.

Let $E_k\subset J^k_n(W)$ be a PDE, with $\pi:W\to M$ a fiber
bundle, $\dim W=m+n$, $\dim M=n$. Then the integral bordism groups
$\Omega_p^{E_k}$, $0\le p\le n-1$, identify a graded ringoid
$\Omega_\bullet^{E_k}$, that we call {\em integral ringoid} of
$E_k$, that is an extension of a graded ringoid contained in the
nonoriented bordism ring $\Omega_\bullet$.

One has the following commutative diagram
\begin{equation}\label{ringoid-hmomorphisms}
\xymatrix{&&0\ar[d]&&\\
0&{\scriptstyle
Hom_{ringoid}(\overline{K}^{E_k}_\bullet;\mathbb{R})}\ar[l]&{\scriptstyle
Hom_{ringoid}(\Omega^{E_k}_\bullet;\mathbb{R})}\ar[l]\ar[d]&
{\scriptstyle Hom_{ringoid}({}^{(n-1)}\Omega_\bullet;\mathbb{R})}\ar[l]\ar[ld]&0\ar[l]\\
&&{\scriptstyle\mathbf{H}_\bullet(E_k)}&&\\}
\end{equation}
where $\mathbf{H}_\bullet(E_k)\equiv\bigoplus_{0\le p\le
n-1}\mathbf{H}_p(E_k)$, that allows us to represent differential
$p$-conservation laws of order $k$ by means of ringoid homomorphisms
$\Omega^{E_k}_\bullet\to\mathbb{R}$, and as extensions of ringoid
homorphisms $\overline{K}^{E_k}_\bullet\to\mathbb{R}$.
\end{lemma}

\begin{proof}
Set
\begin{equation}\label{integral-ringoid}
    \Omega_\bullet^{E_k}\equiv\bigoplus_{0\le p\le
    n-1}\Omega_p^{E_k}.
\end{equation}
Each $\Omega_p^{E_k}$ are additive abelian groups, with addition
induced by disjoint union, $\sqcup$. Furthermore, there is a natural
product induced by the cartesian product, i.e.,
$[X_1]\cdot[X_2]=[X_1\times X_2]\in\Omega_{p_1+p_2}^{E_k}$, for
$[X_i]\in\Omega_{p_i}^{E_k}$, $0\le p_i\le n-1$, $i=1,2$, $0\le
p_1+p_2\le n-1$. This product it is not always defined for any
closed admissible integral manifolds $X_i$,  $i=1,2$, but only for
ones such that $X_1\times X_2$ is a closed integral admissible
manifold. Therefore $\Omega_\bullet^{E_k}$ is a graded ringoid. Set
${}^{(n-1)}\Omega_\bullet\equiv\bigoplus_{0\le p\le
    n-1}\Omega_p$. It has in a natural way a graded ringoid structure, with
    respect the same operations with respect to which
    $\Omega_\bullet$ is a graded ring. Furthermore, for any $0\le p\le
    n-1$, one has the following exact sequence, (see proof of
    Theorem 3.16 in \cite{PRA23}),
\begin{equation}
\xymatrix{0\ar[r]&\overline{K}_p^{E_k}\ar[r]&\Omega_p^{E_k}\ar[r]&\Omega_p\ar[r]&0\\}
\end{equation}
As a by-product one has also the following exact commutative
diagram:
\begin{equation}
\xymatrix{&&&0\ar[d]&\\
0\ar[r]&\overline{K}_\bullet^{E_k}\ar[r]&\Omega_\bullet^{E_k}\ar[r]&{}^{(n-1)}\Omega_\bullet\ar[r]\ar[d]&0\\
&&&\Omega_\bullet\\}
\end{equation}
where $\overline{K}_\bullet^{E_k}\equiv\bigoplus_{0\le p\le
    n-1}\overline{K}_p^{E_k}$.

A full $p$-conservation law is any function
$f:\Omega_p^{E_k}\to\mathbb{R}$, $0\le p\le n-1$. These, identify
elements of $\mathbf{H}_\bullet(E_k)\equiv\bigoplus_{0\le p\le
n-1}\mathbf{H}_p(E_k)$ in a natural way. In
$\mathbf{H}_\bullet(E_k)$ are contained also ones identified by
means of differential conservation laws of order $k$, identified
with $\mathfrak{I}(E_k)^\bullet\equiv\bigoplus_{0\le q\le
n-1}\mathfrak{I}(E_k)^q$, with
\begin{equation}\label{p-conservation-laws-order-k}
    {\frak I}(E_k)^{q}
\equiv\frac{\Omega^q(E_k)\cap
d^{-1}(C\Omega^{q+1}(E_k))}{d\Omega^{q-1}(E_k)\oplus\{C\Omega^q(E_k)\cap
d^{-1}(C\Omega^{q+1}(E_k))\}}.
\end{equation}
Here, $\Omega^q(E_k)$ is the space of smooth $q$-differential forms
on $E_k$ and $C\Omega^q(E_k)$ is the space of Cartan $q$-forms on
$E_k$, that are zero on the Cartan distribution $\mathbf{E}_k$ of
$E_k$. Therefore, $\beta\in C\Omega^q(E_k)$ iff
$\beta(\zeta_1,\cdots,\zeta_q)=0$, for all $\zeta_i\in
C^\infty(\mathbf{E}_k)$.\footnote{${\frak C}ons(E_k)$ can be
identified with the spectral term $E_1^{0,n-1}$ of the spectral
sequence associated to the filtration induced in the graded algebra
$\Omega^\bullet(E_\infty)\equiv\oplus_{q\ge 0}\Omega^q(E_\infty)$,
by the subspaces $C\Omega^q(E_\infty)\subset\Omega^q(E_\infty)$.
(For abuse of language we shall call ''conservation laws of
$k$-order'', characteristic integral $(n-1)$-forms too. Note that
$C\Omega^0(E_k)=0$. See also Refs.\cite{PRA11, PRA13, PRA20}.)} Any
$\alpha\in\mathfrak{I}(E_k)^\bullet$ identifies a ringoid
homomorphism $f[\alpha]:\Omega_\bullet^{E_k}\to\mathbb{R}$. More
precisely one has
$f[\alpha]([X_1]+[X_2])=<\alpha,X_1>+<\alpha,X_2>$, for $[\alpha]\in
{\frak I}(E_k)^{p}$, $[X_1],[X_2]\in \Omega_p^{E_k}$, and
$f[\alpha]([X_1]\cdot[X_2])=<\alpha_1,X_1><\alpha_2,X_2>$, for
$[\alpha]=[\alpha_1]+[\alpha_2]\in {\frak I}(E_k)^{p_1}\oplus {\frak
I}(E_k)^{p_2}$, $[X_1]\in \Omega_{p_1}^{E_k}$, $[X_2]\in
\Omega_{p_2}^{E_k}$.
\end{proof}

In the following we extend above results to PDE's in the category
$\mathfrak{Q}_S$.

\begin{lemma}{\em(Relations between lower order integral bigraded-bordisms in quantum super PDE's).}\label{quantum-PDE-lower-order-relation-bigraded-bordism}
Let $\hat E_k\subset \hat J^k_{m|n}(W)$ be a quantum super PDE on
the fiber bundle $\pi:W\to M$, $\dim_B W=(m|n,r|s)$, $\dim_A M=m|n$,
$B=A\times E$, $E$ a quantum superalgebra that is also a $Z$-module,
with $Z=Z(A)$ the centre of $A$. Let ${}^{S}C_{p|q}(\hat E_k)$ be
the set of all compact $p|q$-dimensional, (with respect to $A$),
admissible integral smooth manifolds of $\hat E_k$, $0\le p\le m$,
$0\le q\le n$. The disjoint union gives an addition on
${}^{S}C_{p|q}(\hat E_k)$ with $\varnothing$ as the zero element. Let
us consider the homomorphisms $\partial_{p|q}:{}^{S}C_{p|q}(\hat
E_k)\to {}^{S}C_{p-1|q-1}(\hat E_k)$ that associates to any element
$a\in{}^{S}C_{p|q}(\hat E_k)$ its boundary $\partial
a=\partial_{p|q}(a)$. So we obtain the following chain complex of
abelian groups {\em(integral smooth bigraded-bordisms chain
complex)} of $\hat E_k\subset \hat J^k_{m|n}(W)$:
\begin{equation}\label{integral-smooth-bigraded-bordism-chain-complex}
\xymatrix{{\scriptstyle {}^{S}C_{m|n}(\hat E_k)}\ar[r]^(0.5){\partial_{m|n}}&{\scriptstyle {}^{S}C_{m-1|n-1}(\hat
E_k)}\ar[r]^(0.5){\partial_{m-1|n-1}}&{\scriptstyle {}^{S}C_{m-2|n-2}(\hat
E_k)}\ar[r]^(0.6){\partial_{m-2|n-2}}&
\cdots\ar[r]^(.3){\partial_{r|r}}&{\scriptstyle {}^{S}C_{m-r|n-r}(\hat
E_k)}\\}
\end{equation}
where $r=\min\{m,n\}$. Then the $p|q$-integral bordism groups $\Omega_{p|q}^{\hat
E_k}$, $(m-r)<p<m$, $(n-r)<q<n$, can be represented by means of the homology
of the chain complex
{\em(\ref{integral-smooth-bigraded-bordism-chain-complex})}.

One has the following canonical isomorphism:
\begin{equation}\label{isomorphism-relation-between-integral-bigraded-bordisms-groups-in-quantum-PDEs}
    \mathbb{Z}\bigotimes_{\Omega_{\bullet|\bullet}^{\hat E_k}}\mathbb{Z}{}^{S}Bor_{\bullet|\bullet}(\hat E_k)
    \cong\mathbb{Z}{}^{S}Cyc_{\bullet|\bullet}(\hat E_k).
\end{equation}
Furthermore, if $H^2({}^{S}Cyc_{\bullet|\bullet}(\hat
E_k),\Omega_{\bullet|\bullet}^{\hat E_k})=0$ one has the following
canonical isomorphism:
\begin{equation}\label{isomorphism-relation-between-integral-bigraded-bordisms-groups-in-quantum-PDEs-2}
    {}^{S}Bor_{\bullet|\bullet}(\hat E_k)\cong\Omega_{\bullet|\bullet}^{\hat E_k}\times{}^{S}Cyc_{\bullet|\bullet}(\hat E_k).
\end{equation}
\end{lemma}

\begin{proof}
The proof can be conduced similarly to the ones for Lemma
\ref{commutative-PDE-lower-order-relation-bordism}, Lemma
\ref{relation-between-integral-bordisms-groups-in-commutative-PDEs}
and Lemma
\ref{relation-between-integral-bordisms-groups-in-commutative-PDEs-2}.
\end{proof}

Similarly we can prove the following lemma concerning the total
analogous of the complex
(\ref{integral-smooth-bigraded-bordism-chain-complex}) too.

\begin{lemma}{\em(Relations between lower order integral total-bordisms in quantum super PDE's).}\label{quantum-PDE-lower-order-relation-total-bordism}
Let $\hat E_k\subset \hat J^k_{m|n}(W)$ be a quantum super PDE on
the fiber bundle $\pi:W\to M$, $\dim_B W=(m|n,r|s)$, $\dim_A M=m|n$,
$B=A\times E$, $E$ a quantum superalgebra that is also a $Z$-module,
with $Z=Z(A)$ the centre of $A$. Let ${}^{S}C_{p}(\hat E_k)$, $0\le
p\le m+n$, be the set of all compact $u|v$-dimensional, (with
respect to $A$), admissible integral smooth manifolds of $\hat E_k$,
such that $u+v=p$. The disjoint union gives an addition on
${}^{S}C_{p}(\hat E_k)$ with $\varnothing$ as the zero element. Thus
we can write
\begin{equation}\label{total-smooth-integral-bordisms-quantum-pde}
{}^{S}C_{p}(\hat E_k)=\bigoplus_{u,v; u+v=p}{}^{S}C_{u|v}(\hat
E_k)={}^{Tot,S}C_{p}(\hat E_k).
\end{equation}

Let us consider the homomorphisms $\partial_{p}:{}^{S}C_{p}(\hat
E_k)\to {}^{S}C_{p-1}(\hat E_k)$ that associates to any element
$a\in{}^{S}C_{p}(\hat E_k)$ its boundary $\partial
a=\partial_{p}(a)$, i.e., one has:
\begin{equation}\label{total-smooth-integral-bordisms-quantum-pde-morphisms}
\begin{array}{ll}
  \partial_pa&=\partial_p(a_{p|0},a_{p-1|1},a_{p-2|2},\cdots,a_{0|p})\\
  &=(\partial_{p|0}a_{p|0},\partial_{p-1|1}a_{p-1|1},\partial_{p-2|2}a_{p-2|2},\cdots,\partial_{0|p}a_{0|p})\\
  &\in\bigoplus_{u,v; u+v=p-1}{}^{S}C_{u|v}(\hat
E_k)={}^{S}C_{p-1}(\hat E_k).
\end{array}
\end{equation}
One has $\partial_{p-1}\circ\partial_p=0$. So we get the following
chain complex of abelian groups {\em(integral smooth
bigraded-bordisms chain complex)} of $\hat E_k\subset \hat
J^k_{m|n}(W)$:
\begin{equation}\label{integral-smooth-total-bordism-chain-complex}
\xymatrix{{}^{S}C_{n}(\hat E_k)\ar[r]^{\partial_{n}}&{}^{S}C_{n-1}(\hat
E_k)\ar[r]^{\partial_{n-1}}&{}^{S}C_{n-2}(\hat
E_k)\ar[r]^{\partial_{n-2}}&
\cdots\ar[r]^(.3){\partial_{1}}&{}^{S}C_{0}(\hat
E_k).\\}
\end{equation}
Then the $p$-integral total bordism groups $\Omega_{p}^{\hat
E_k}$, $0<p<m+n$, can be represented by means of the homology of the
chain complex
{\em(\ref{integral-smooth-total-bordism-chain-complex})}.

One has the following canonical isomorphism:
\begin{equation}\label{isomorphism-relation-between-integral-total-bordisms-groups-in-quantum-PDEs}
    \mathbb{Z}\bigotimes_{\Omega_{\bullet}^{\hat E_k}}\mathbb{Z}{}^{S}Bor_{\bullet}(\hat E_k)
    \cong\mathbb{Z}{}^{S}Cyc_{\bullet}(\hat E_k).
\end{equation}
Furthermore, if $H^2({}^{S}Cyc_{\bullet}(\hat
E_k),\Omega_{\bullet}^{\hat E_k})=0$ one has the following canonical
isomorphism:
\begin{equation}\label{isomorphism-relation-between-integral-total-bordisms-groups-in-quantum-PDEs-2}
    {}^{S}Bor_{\bullet}(\hat E_k)\cong\Omega_{\bullet}^{\hat E_k}\times{}^{S}Cyc_{\bullet}(\hat E_k).
\end{equation}
\end{lemma}
\begin{proof}
The proof is similar to the one of Lemma \ref{quantum-PDE-lower-order-relation-bigraded-bordism}.
\end{proof}

\begin{lemma}{\em(Integral ringoid of PDE's in $\mathfrak{Q}_S$ and quantum conservation laws).}
Let $\hat E_k\subset \hat J^k_{m|n}(W)$ be a PDE in the category
$\mathfrak{Q}_S$ as defined in Lemma
\ref{quantum-PDE-lower-order-relation-total-bordism}. Then,
$\Omega^{\hat E_k}_\bullet\equiv\bigoplus_{0\le p\le
m+n}\Omega^{\hat E_k}_p$, has a natural structure of graded ringoid,
with respect to the (partial) binary operations similar to the
commutative case. We call $\Omega^{\hat E_k}_\bullet$ the {\em
integral ringoid} of $\hat E_k$. Furthermore, quantum conservation
laws of order $k$, $\hat f\in Map(\Omega^{\hat E_k}_{p|q},B_k)\equiv
\mathbf{H}_{p|q}(\hat E_k)$, can be projected on their classic
limits $\hat f\mapsto\hat f_C\equiv c\circ\hat f\in Map(\Omega^{\hat
E_k}_{p|q},\mathbb{K})\equiv \mathbf{H}_{p|q}(\hat E_k)_C$. By
passing to the corresponding total spaces, we get the following
exact commutative diagram:
\begin{equation}\label{quantum-conservation-laws-classic-limit}
\xymatrix{0\ar[d]&0\ar[d]&\\
\mathbf{H}_{p|q}(\hat E_k)\ar[d]\ar[r]&\mathbf{H}_{p|q}(\hat E_k)_C\ar[d]\ar[r]&0\\
\mathbf{H}_\bullet(\hat E_k)\ar[r]&\mathbf{H}_\bullet(\hat E_k)_C\ar[r]&0\\}
\end{equation}
Moreover, graded ringoid hmomorphisms $\hat h\in
Hom_{ringoid}(\Omega_\bullet^{\hat E_k},\mathbb{K})$, can be
identified by means of classic limit quantum conservation laws of
$\hat E_k$. One has the following exact commutative diagram:
\begin{equation}\label{ringoid-homomorphisms-quantum-conservation-laws}
\xymatrix{0\ar[r]&{}^{R}\mathbf{H}_\bullet(\hat E_k)\ar[r]\ar[d]&\mathbf{H}_\bullet(\hat E_k)\ar[d]\\
0\ar[r]&Hom_{ringoid}(\Omega_\bullet^{\hat
E_k},\mathbb{K})\ar[r]\ar[d]&\mathbf{H}_\bullet(\hat E_k)_C\ar[d]\\
&0&0\\}
\end{equation}
that defines a subalgebra ${}^{R}\mathbf{H}_\bullet(\hat E_k)$ of
$\mathbf{H}_\bullet(\hat E_k)$, whose elements we call {\em rigid quantum
conservation laws}, and whose classic limit can be identified with
ringoid homomorphisms $\Omega_\bullet^{\hat E_k}\to\mathbb{K}$. In
particular, quantum conservation laws arising by full quantum
differential form classes
\begin{equation}\label{quantum-differential-conservation-laws}
\left\{\begin{array}{ll}
         [\alpha]& \in \bigoplus_{p,q\ge 0 }\hat{\frak I}(\hat
E_k)^{p|q} \\
         & \hat{\frak I}(\hat
E_k)^{p|q}\equiv{{\widehat{\Omega}^{p|q}(\hat E_k)\cap
d^{-1}(C\widehat{\Omega}^{p+1|q+1}(\hat
E_k))}\over{d\widehat{\Omega}^{p-1|q-1}(E_k)\oplus\{C\widehat{\Omega}^{p|q}(\hat
E_k)\cap d^{-1}(C\widehat{\Omega}^{p+1|q+1}(\hat E_k)))\}}}
       \end{array}
\right.
\end{equation}
belong to ${}^{R}\mathbf{H}_\bullet(\hat E_k)$.
\end{lemma}
\begin{proof}
The proof follows directly from above lemmas. (For details on spaces $\hat{\frak I}(\hat
E_k)^{p|q}$ see Refs.\cite{PRA15, PRA20, PRA22}.)
\end{proof}

Let us, now, denote $\mathop{\Omega}\limits_c{}_{p|q}^{\hat E_k}$
(or $\mathop{\Omega}\limits_c{}_{p+q}^{\hat E_k}$), the classic
limit of integral $(p|q)$-bordism group of $\hat E_k$, i.e., the
$(p+q)$-bordism group of classic limits of integral supermanifolds
$N\subset\hat E_k$, such that $\dim_AN=p|q$. Furthermore, let us
denote by $\mathop{\Omega}\limits_c{}_{\widehat{p+q}}^{\hat E_k}$
the classic limit of total integral $(p+q)$-bordism group of $\hat
E_k$, i.e., the $(p+q)$-bordism group of classic limits of integral
supermanifolds $N\subset \hat E_k$, such that $\dim_AN=u|v$, with
$u+v=p+q$. One has the following exact commutative diagram:

\begin{equation}\label{commutative-exact-diagram-relation-integral-bord-classic-limit-total-integral-bord}
\xymatrix{0\ar[r]&\Omega_{p|q}^{\hat
E_k}\ar[d]\ar[r]&\Omega_{p+q}^{\hat E_k}\ar[d]\\
0\ar[r]&\mathop{\Omega}\limits_c{}_{p+q}^{\hat E_k}\ar[d]\ar[r]&\mathop{\Omega}\limits_c{}_{\widehat{p+q}}^{\hat E_k}\ar[d]\\
&0&0\\}
\end{equation}

Taking into account Theorem 3.6 in \cite{PRA22} we get a relation
between $\Omega_{m-1|n-1}^{\hat E_k}$,
$\mathop{\Omega}\limits_c{}_{p|q}^{\hat E_k}$ and the bordism group
$\Omega_{m+n-2}$. In fact, we can see that there is a relation
between integral bordism groups in quantum super PDEs and Reinhart
integral bordism groups of commutative manifolds. More precisely,
let $N_0, N_1\subset\hat E_k\subset\hat J^k_{m|n}(W)$ be closed
admissible integral quantum supermanifolds of a quantum super PDE
$\hat E_k$, of dimension $(m-1|n-1)$ over $A$, such that $N_0\sqcup
N_1=\partial V$, for some admissible integral quantum supermanifold
$V\subset \hat E_k$, of dimension $(m|n)$ over $A$. Then $(N_0)_C\sqcup
(N_1)_C=\partial V_C$ iff $(N_0)_C$ and $(N_1)_C$ have the same
Stiefel-Whitney and Euler characteristic numbers. In fact, by
denoting $ \Omega_p^\uparrow$ the Reinhart $p$-bordism groups and
$\Omega_p$ the $p$-bordism group for closed smooth finite
dimensional manifolds respectively, one has the following exact
commutative diagram
\begin{equation}\label{Reinhart-bordism-groups-relation}
\xymatrix{0\ar[r]&K^{\hat
E_k}_{m-1|n-1;m+n-2}\ar[r]&\Omega_{m-1|n-1}^{\hat E_k}\ar[r]&
\mathop{\Omega}\limits_c{}_{m+n-2}^{\hat E_k}\ar[d]\ar[r]\ar[dr]&0&\\
&0\ar[r]&K^\uparrow_{m+n-2}\ar[r]&\Omega^\uparrow_{m+n-2}\ar[r]&
\Omega_{m+n-2}\ar[r]& 0\\}
\end{equation}

This has as a consequence that if $N_0\sqcup N_1=\partial V$, then
$(N_0)_C\sqcup (N_1)_C=\partial V_C$ iff $(N_0)_C$ and $(N_1)_C$ have
the same Stiefel-Whitney and Euler characteristic
numbers.\footnote{Note that for $p+q=3$ one has $K_3^\uparrow=0$,
hence one has $\Omega_3^\uparrow=\Omega_3$.}

From above exact commutative diagram one has that
$\Omega_{m-1|n-1}^{\hat E_k}$ is an extension of a subgroup of
$\Omega_{m+n-2}$.

Let us consider, now, the following lemmas.

\begin{lemma}{\em\cite{PRA25}}\label{bordism-groups-crystallography}
Bordism groups, $\Omega_p$, relative to smooth manifolds can be
considered as extensions of some crystallographic subgroup
$G\triangleleft G(d)$.
\end{lemma}

\begin{lemma}
If the group $G$ is an extension of $H$, any subgroup
$\widetilde{G}\vartriangleleft G$ is an extension of a subgroup
$\widetilde{H}\vartriangleleft H$.
\end{lemma}

\begin{proof}
In fact $\widetilde{G}$ is an extension of
$p(\widetilde{G})\vartriangleleft H$, with respect to the following
short exact sequence:
$\xymatrix{0\ar[r]&K\ar[r]&G\ar[r]^{p}&H\ar[r]&0\\}$.
\end{proof}

Therefore by using above two lemmas, we get also that
$\Omega_{m-1|n-1}^{\hat E_k}$ is an extension of some
crystallographic subgroup $G\triangleleft G(d)$.
\end{proof}

The theorem below relates the integrability properties of a quantum
super PDE to crystallographic groups. Let us first give the
following definition.

\begin{definition}
We say that a quantum super PDE $\hat E_k\subset \hat J^k_{m|n}(W)$
is an {\em extended crystal quantum super PDE}, if conditions of
Theorem \ref{crystal-structure-quantum-super-pdes} are verified.
Then, for such a PDE $\hat E_k$ are defined its crystal group $G(d)$
and crystal dimension $d$.
\end{definition}

In the following we relate crystal structure of quantum super PDE's
to the existence of global smooth solutions for smooth boundary
value problems, by identifying an algebraic-topological obstruction.

\begin{theorem}\label{obstruction-smooth-solutions}
Let $B_k$ be the model quantum superalgebra of $\hat J^k_{m|n}(W)$,
$k\ge 0$. (See \cite{PRA21, PRA22}.) We denote also by $
B_\infty=\lim_k B_k$.\footnote{We also adopt the notation $B_k(A)$
and $B_\infty(A)$, whether it is necessary to specify the starting
original quantum super algebra $A$.} Let $\hat E_k\subset \hat
J^k_{m|n}(W)$ be a quantum formally integrable and completely
qunatum superintegrable quantum super PDE. Then, in the algebra
$\mathbf{H}_{m-1|n-1}(\hat E_k)\equiv Map(\Omega_{m-1|n-1}^{\hat
E_k};B_k)$, {\em Hopf quantum superalgebra} of $\hat E_k$, there is
a quantum sub-superalgebra, {\em (crystal Hopf quntum superalgebra)}
of $\hat E_k$.\footnote{Recall that with the term {\em quantum Hopf superalgebra} we mean an extension
$\xymatrix{A\ar[r]&C\equiv A\otimes_{\mathbb{K}}H\ar[r]&D\ar[r]&D/C\ar[r]&0\\}$, where $H$ is an Hopf $\mathbb{K}$-algebra and $A$ is a quantum superalgebra. (For more details on generalized Hopf algebras, associated to PDE's, see Refs.\cite{PRA10, PRA11, PRA22}.)} On such an algebra we can represent the quantum
superalgebra $B^{G(d)}$ associated to the quantum crystal supergroup
$G(d)$ of $\hat E_k$. (This justifies the name.) We call {\em
quantum crystal conservation superlaws} of $\hat E_k$ the elements
of its quantum Hopf crystal superalgebra. Then, the obstruction to
find global smooth solutions of $\hat E_k$, for integral boundaries
with orientable classic limit, can be identified with the quotient
$\mathbf{H}_{m-1|n-1}(\hat E_\infty)/B_\infty^{\Omega_{m+n-2}}$.
\end{theorem}

\begin{proof}
Let $N_0, N_1\subset \hat E_k$ be two respectively initial and
final, closed compact Cauchy data of $\hat E_k$. Then there exists a
weak, (resp. singular, resp. smooth) solution $V\subset \hat E_k$,
such that $\partial V=N_0\sqcup N_1$, iff $X\equiv N_0\sqcup
N_1\in[0]\in\Omega_{m-1|n-1,w}^{\hat E_k}$, (resp.
$X\in[0]\in\Omega_{m-1|n-1,s}^{\hat E_k}$, resp.
$X\in[0]\in\Omega_{m-1|n-1}^{\hat E_k}$). Let $X_C$ be orientable,
then $X$ is the boundary of a smooth solution, iff $X$ has zero all
the integral characteristic quantum supernumbers, i.e.,
$<\alpha,X>=0$, $\forall\alpha\in\mathbf{H}_{m-1|n-1}(\hat
E_\infty)=Map(\Omega_{m-1|n-1}^{\hat E_\infty},B_\infty)$. Taking
into account the following short exact sequence: $0\to
\Omega_{m-1|n-1,w}^{\hat E_k}\to\widetilde{\Omega}_{m+n-2}$, where
$\widetilde{\Omega}_{m+n-2}\vartriangleleft G(d)$, for some
crystallographic group $G(d)$, we get also the following short exact
sequence: $\mathbf{H}_{m-1|n-1}(\hat E_\infty)\leftarrow
B_\infty^{\widetilde{\Omega}_{m+n-2}}\leftarrow 0$. So
$B_\infty^{\widetilde{\Omega}_{m+n-2}}$ can be identified with a
subalgebra of $\mathbf{H}_{m-1|n-1}(\hat E_\infty)$. Then the
obstruction to find smooth solutions can be identified with the
quotient $\mathbf{H}_{m-1|n-1}(\hat E_\infty)/
B_\infty^{\widetilde{\Omega}_{m+n-2}}$. Taking into account that
$\widetilde{\Omega}_{m+n-2}\vartriangleleft\Omega_{m+n-2}\vartriangleleft
G(d)$, we can also represent $B^{\widetilde{\Omega}_{m+n-2}}$ with
$B_\infty^{\Omega_{m+n-2}}$, or with $B_\infty^{G(d)}$. Thus, it is
justified also call $B_\infty^{\widetilde{\Omega}_{m+n-2}}$ as
crystal quantum superlaws algebra of $\hat E_k$.
\end{proof}

\begin{definition}
We define {\em crystal obstruction} of $\hat E_k$ the above quotient
of algebras, and put: $ cry(\hat E_k)\equiv
\mathbf{H}_{m-1|n-1}(\hat E_\infty)/B_\infty^{\Omega_{m+n-2}}$. We
call {\em quantum $0$-crystal super PDE} a quantum super PDE $\hat
E_k\subset \hat J^k_{m|n}(W)$ such that $cry(\hat E_k)=0$.
\end{definition}

\begin{remark}
A quantum extended $0$-crystal super PDE $\hat E_k\subset \hat
J^k_{m|n}(W)$ does not necessitate to be a quantum $0$-crystal super
PDE. In fact $\hat E_k$ is an extended $0$-crsytal quantum super PDE
if $\Omega_{m-1|n-1,w}^{\hat E_k}=0$. This does not necessarily
implies that $\Omega_{m-1|n-1}^{\hat E_k}=0$. In fact, the different
types of integral bordism groups of PDE's in the category
$\mathfrak{Q}_S$, are related by the following proposition.
\end{remark}

\begin{proposition}{\em(Relations between integral bordism groups).}{\em\cite{PRA22}}
The different types of integral bordism groups for a quantum super
PDE, are related by the exact commutative diagram reported in
{\em(\ref{relations-between-integral-bordism-groups-commutative-diagram})}.

\begin{equation}\label{relations-between-integral-bordism-groups-commutative-diagram}
\xymatrix{
&0\ar[d]&0\ar[d]&0\ar[d]&\\
0\ar[r]&K^{\hat E_k}_{m-1|n-1,w/(s,w)}\ar[d]\ar[r]&
K^{\hat E_k}_{m-1|n-1,w}\ar[d]\ar[r]&K^{\hat E_k}_{m-1|n-1,s,w}\ar[d]\ar[r]&0\\
0\ar[r]&K^{\hat E_k}_{m-1|n-1,s}\ar[d]\ar[r]&
\Omega^{\hat E_k}_{m-1|n-1}\ar[d]\ar[r]&\Omega^{\hat E_k}_{m-1|n-1,s}\ar[d]\ar[r]&0\\
&0\ar[r]&\Omega^{\hat E_k}_{m-1|n-1,w}\ar[d]\ar[r]&\Omega^{\hat E_k}_{m-1|n-1,w}\ar[d]\ar[r]&0\\
&&0&0&}
\end{equation}

One has the canonical isomorphisms:
\begin{equation}\label{canonical-isomorphisms-integral-bordism-groups-comm-pde}
\left\{   \begin{array}{l}
     K^{\hat E_k}_{m-1|n-1,w/(s,w)}\cong K^{\hat E_k}_{m-1|n-1,s}\\
\Omega^{\hat E_k}_{m-1|n-1}/K^{\hat E_k}_{m-1|n-1,s}\cong
\Omega^{\hat E_k}_{m-1|n-1,s}\\
\Omega^{\hat E_k}_{m-1|n-1,s}/K^{\hat
E_k}_{m-1|n-1,s,w}\cong\Omega^{\hat E_k}_{m-1|n-1,w}\\
\Omega^{\hat E_k}_{m-1|n-1}/K^{\hat
E_k}_{m-1|n-1,w}\cong\Omega^{\hat E_k}_{m-1|n-1,w}.\\
   \end{array}\right.
\end{equation}

\end{proposition}

\begin{cor}\label{main7}
Let $\hat E_k\subset \hat J^k_{m|n}(W)$ be a quantum $0$-crystal
super PDE. Let $N_0, N_1\subset \hat E_k$ be two closed initial and
final Cauchy data of $\hat E_k$ such that $X\equiv N_0\sqcup
N_1\in[0]\in\Omega_{m-1|n-1}$, and such that $X_C$ is orientable.
Then there exists a smooth solution $V\subset \hat E_k$ such that
$\partial V=X$.
\end{cor}

\begin{example}{\em(Quantum extended crystal SG-Yang-Mills PDE's.)}\label{SG-Yang-Mills}
Let us introduce some fundamental geometric objects to encode
quantum supergravity. (See also our previous works on this subjects
that formulate quantum supergravity in the framework of our
geometric theory of quantum super PDE's \cite{PRA13, PRA18, PRA21,
PRA28, PRA30}.) The first geometric object to consider is a {\em
quantum Riemannian (super)manifold}, $(M,\widehat{g})$, of dimension
$m$, with respect to a quantum algebra $A$, where $\widehat{g}:M\to
Hom_Z(\dot T^2_0M;A)$ is a quantum metric. We shall assume that $M$
is locally {\em quantum (super) Minkowskian}, i.e., there is a
$Z$-isomorphism, {\em(quantum vierbein)}:
\begin{equation}\label{quantum-vierbein}
    \hat\theta(p):T_pM\cong A\otimes_{R}\mathbf{M}_C, \forall p\in M,
\end{equation}
where $T_pM$ is the tangent space, at $p\in M$, to $M$, and
$\mathbf{M}_C$ is the $m$-dimensional vector space of free vectors
of a $m$-dimensional affine Minkowsky space-time. Equivalently a
quantum verbein is a section $\hat\theta:M\to
Hom_Z(TM;E)\cong\widehat{E}\otimes_{\widehat{A}}(TM)^+$, where $E$
is the trivial fiber bundle $\bar\pi:E\equiv M\times
A\otimes_{R}\mathbf{M}_C$. Let us denote by $\underline{g}\in
S^0_2(\mathbf{M}_C)$ the hyperbolic scalar product with signature
$(+,---\cdots)$. $\underline{g}$ induces a $A$-valued scalar
product, $\underline{\hat g}$, on $A\otimes_{R}\mathbf{M}_C$, given
by $\underline{\hat g}(a\otimes u,b\otimes v)=ab\hskip
  3pt\underline{g}(u,v)\in A $. By using the canonical splitting $Hom_Z(\dot T^2_0(A\otimes_{\mathbb{R}}\mathbf{M}_C;A)\cong Hom_Z(\dot S^2_0(A\otimes_{\mathbb{R}}\mathbf{M}_C;A)\oplus
  Hom_Z(\dot \Lambda^2_0(A\otimes_{\mathbb{R}}\mathbf{M}_C;A)$, we get also the split representation $\underline{\hat g}=\underline{\hat g}_{(s)}+\underline{\hat g}_{(a)}$. More precisely one has
 \begin{equation}
 \underline{\hat g}_{(s)}(a\otimes u,b\otimes v)=[a,b]_+\underline{g}(u,v),\quad
 \underline{\hat g}_{(a)}(a\otimes u,b\otimes v)=[a,b]_-\underline{g}(u,v).
  \end{equation}
  Furthermore, if $(e_\alpha)$ is a basis in $\mathbf{M}_C$, and $(e^\beta)$ is its dual, characterized by the conditions $e_\alpha e^\beta=\delta^\beta_\alpha$, let us denote respectively by $(\widehat{1\otimes e_\alpha})$ and
  $((1\otimes e^\beta)^+)$ the induced dual bases on the spaces $\widehat{A\otimes_{\mathbb{R}}\mathbf{M}_C}$ and $(A\otimes_{\mathbb{R}}\mathbf{M}_C)^+$ respectively. Then one has the following representations
 \begin{equation}\label{basis-representation-quantum-extended-minkowsky-product}
\left\{   \begin{array}{l}
     \underline{\hat g}=\underline{\hat g}_{\alpha\beta}(1\otimes e^\alpha)^+\otimes(1\otimes e^\beta)^+,\quad
     \underline{\hat g}_{\alpha\beta}\in\mathop{\widehat{A}}\limits^{2}\\
     \underline{\hat g}_{(s)}=\underline{\hat g}_{(s)}{}_{\alpha\beta}(1\otimes e^\alpha)^+\bullet(1\otimes e^\beta)^+,\quad
     \underline{\hat g}_{(s)}{}_{\alpha\beta}\in\mathop{\widehat{A}}\limits^{2}\\
     \underline{\hat g}_{(a)}=\underline{\hat g}_{(a)}{}_{\alpha\beta}(1\otimes e^\alpha)^+\triangle(1\otimes e^\beta)^+,\quad
     \underline{\hat g}_{(a)}{}_{\alpha\beta}\in\mathop{\widehat{A}}\limits^{2}.\\
     \end{array}\right.
 \end{equation}

By means of the isomorphism $\hat\theta^{\otimes}$, we can induce on $M$ a quantum
  metric, i.e., the {\em quantum Minkowskian metric} of $M$,
$\widehat{g}=\underline{\hat g}\circ\hat\theta^{\otimes}$.
Conversely any quantum metric $\widehat{g}$ on $M$, induces on the
space $A\otimes_{R}\mathbf{M}_C$, scalar products, for any $p\in M$:
$\hat g(p)=\widehat{g}(p)\circ(\hat\theta^{\otimes}(p))^{-1}$. As a
by-product, we get that any quantum metric $\widehat{g}$ on $M$,
induces a quantum metric on the fiber bundle $\bar\pi:E\to M$, that
we call the {\em deformed quantum metrics} of $\bar\pi:E\to M$. Therefore, when we talk about locally Minkowskian quantum manifold $M$, we mean that on $M$ is defined a Minkowskian quantum metric.
Since
$Hom_Z(TM;A\otimes_{\mathbb{R}}\mathbf{M}_C)\cong\widehat{A\otimes_{\mathbb{R}}\mathbf{M}_C}\otimes_{\widehat{A}}(TM)^+$,
we can locally represent a quantum vierbein in the following form:
\begin{equation}\label{local-quantum-vierbein}
    \hat\theta=\widehat{1\otimes e_\beta}\bigotimes
    \hat\theta^\beta_\alpha\hskip
  3pt dx^\alpha,
\end{equation}
where $\widehat{1\otimes e_\beta}\in
Hom_Z(A;A\otimes_{\mathbb{R}}\mathbf{M}_C)$, is the full quantum
extension of a basis $(e_\alpha)_{0\le\alpha\le m-1}$ of
$\mathbf{M}_C$, i.e., $\widehat{1\otimes e_\beta}(a)=a\otimes
e_\beta$. Furthermore, $\hat\theta^\beta_\alpha(p)\in \widehat{A}$.
Then, if $\zeta:M\to \widehat{TM}\equiv Hom_Z(A;TM)$ is a full
quantum vector field on $M$, locally represented by $\zeta=\partial
x_\alpha \zeta^\alpha$, we get that its local representation by
means of quantum vierbein, is given by the following formula:
\begin{equation}\label{quantum-vierbein-full-quantum-vector-field}
    \hat\theta(\zeta)=\widehat{1\otimes e_\beta}\hat\theta^\beta_\alpha\zeta^\alpha,
\end{equation}
where the product is given by composition:
\begin{equation}
\xymatrix@1@C=50pt{A\ar[r]^{\zeta^\alpha}\ar@/_1pc/[rrr]_{\sum_{\alpha,\beta}=\hat\theta(\zeta)}&
A\ar[r]^{\hat\theta_\alpha^\beta}&A\ar[r]^{\widehat{1\otimes e_\beta}}&A\otimes_Z\mathbf{M}_C\\}
\end{equation}
(For abuse of notation we
can also denote $\hat\theta(\zeta)$ by $\zeta$ yet.)  Whether
$\widehat{g}=\widehat{g}_{\alpha\beta}dx^\alpha\otimes dx^\beta$, is
the quantum Minkowskian metric of $M$, then its local representation
by means of the quantum vierbein is the following:
\begin{equation}\label{quantum-vierbein-full-quantum-metric}
\left\{\begin{array}{l}
         \widehat{g}=\widehat{g}_{\alpha\omega}dx^\alpha\otimes dx^\omega\\
       \widehat{g}_{\alpha\omega}=\hat\theta^\beta_\alpha\otimes\hat\theta^\gamma_\omega\hskip 3pt\underline{g}_{\beta\gamma}=
   \underline{g}_{\beta\gamma}\hat\theta^\beta_\alpha\otimes\hat\theta^\gamma_\omega,\quad
   \widehat{g}_{\alpha\omega}(p)\in\mathop{\widehat{A}}\limits^{2},\quad \forall p\in M.\\
   \end{array}\right.
   \end{equation}
where $\hat\theta^\beta_\alpha\otimes\hat\theta^\gamma_\omega(p)$,
can be identified with
$\hat\theta^\beta_\alpha\otimes\hat\theta^\gamma_\omega(p)\in
Hom_Z(\dot T^2_0(A);\dot T^2_0(A))$. In fact, one has the following
extension $\hat\theta^\otimes$ of $\hat\theta$:
\begin{equation}\label{tensor-extension-quantum-vierbein}
  \left\{
  \begin{array}{ll}
     \hat\theta^{\otimes}& \in
   Hom_Z(T\otimes_ZTM;(A\otimes_{\mathbb{R}}\mathbf{M}_C)\otimes_Z(A\otimes_{\mathbb{R}}\mathbf{M}_C)) \\
    & \cong\widehat{(A\otimes_{\mathbb{R}}\mathbf{M}_C)\otimes_Z(A\otimes_{\mathbb{R}}\mathbf{M}_C)}\bigotimes_{\widehat{A}}
   (TM\otimes_ZTM)^+.\\
  \end{array}
\right.
\end{equation}
Locally one can write
\begin{equation}\label{local-tensor-extension-quantum-vierbein}
   \hat\theta^\otimes=\widehat{(1\otimes e_\gamma)\otimes(1\otimes
   e_\omega)}\otimes\hat\theta^\gamma_\alpha\otimes\hat\theta^\omega_\beta\hskip
  3pt
   dx^\alpha\otimes dx^\beta.
\end{equation}
In fact, we have
\begin{equation}\label{calculation-local-tensor-extension-quantum-vierbein}
\widehat{g}(\zeta,\xi) =  \widehat{g}(\widehat{1\otimes e_\beta}\hat\theta^\beta_\alpha\zeta^\alpha,
  \widehat{1\otimes e_\gamma}\hat\theta^\gamma_\omega\xi^\omega)
  =\hat\theta^\beta_\alpha\zeta^\alpha\hat\theta^\gamma_\omega\xi^\omega\hskip
  3pt \underline{g}(e_\beta,e_\gamma)
  =\hat\theta^\beta_\alpha\zeta^\alpha\hat\theta^\gamma_\omega\xi^\omega\hskip
  3pt
  \underline{g}_{\beta\gamma}.
\end{equation}
In the particular case that $(e_\beta)$ is an orthonormal basis,
then we get the following quantum Minkowskian representation for
$\widehat{g}$
\begin{equation}\label{minkowskian-representation-full-quantum-metric}
(\widehat{g}_{\alpha\omega})=\hat\theta^\beta_\alpha\otimes\hat\theta^\gamma_\omega\hskip
  3pt
  \eta_{\beta\gamma},\quad(\eta_{\beta\gamma}) = \left(
                                   \begin{array}{cccc}
                                     1& 0&\cdots&0 \\
                                     0& -1&\cdots&0 \\
                                     \cdots&\cdots &\cdots&\cdots \\
                                     0& 0&\cdots&-1\\
                                   \end{array}
                                 \right).
\end{equation}

The splitting in symmetric and skewsymmetric part of $\widehat{g}$,
i.e.,
\begin{equation}\label{splitting-full-quantum-metric}
\widehat{g}=\widehat{g}_{(s)}+\widehat{g}_{(a)}=\widehat{g}_{\alpha\beta}\hskip
  3ptdx^\alpha\bullet
dx^\beta+\widehat{g}_{\alpha\beta}\hskip
  3ptdx^\alpha\triangle dx^\beta
\end{equation}
can be written in term of quantum vierbein in the following way:
\begin{equation}\label{symmetric-skewsymmetric-vierbein}
\left\{
\begin{array}{l}
  \hat\theta^\otimes=\hat\theta^{\odot}+ \hat\theta^{\wedge}\\
  \hat\theta^{\odot}=\widehat{(1\otimes e_\gamma)\otimes(1\otimes
   e_\omega)}\otimes\hat\theta^\gamma_\alpha\otimes\hat\theta^\omega_\beta\hskip
  3pt
   dx^\alpha\bullet dx^\beta\\
\hat\theta^{\wedge}=\widehat{(1\otimes e_\gamma)\otimes(1\otimes
   e_\omega)}\otimes\hat\theta^\gamma_\alpha\otimes\hat\theta^\omega_\beta\hskip
  3pt
   dx^\alpha\triangle dx^\beta\\
\widehat{g}_{(s)}(\zeta,\xi)=[\hat\theta^\beta_\alpha\zeta^\alpha,\hat\theta^\gamma_\omega\xi^\omega]_+\hskip
  3pt
  \underline{g}_{\beta\gamma},\Rightarrow \widehat{g}_{(s)}{}_{\alpha\omega}=
  \hat\theta^\beta_\alpha\bullet\hat\theta^\gamma_\omega\hskip 3pt\underline{g}_{\beta\gamma}\\
  \widehat{g}_{(a)}(\zeta,\xi)=[\hat\theta^\beta_\alpha\zeta^\alpha,\hat\theta^\gamma_\omega\xi^\omega]_-\hskip
  3pt
  \underline{g}_{\beta\gamma},\Rightarrow
  \widehat{g}_{(a)}{}_{\alpha\omega}=\hat\theta^\beta_\alpha\triangle\hat\theta^\gamma_\omega\hskip
  3pt\underline{g}_{\beta\gamma}.\\
\end{array}
\right.
\end{equation}

Conversely, the local expression of the quantum deformed metrics on
$\bar\pi:E\to M$, induced by a quantum metrics $\widehat{g}$ on $M$,
is given by the following formulas:

\begin{equation}\label{symmetric-skewsymmetric-vierbein}
\left\{
\begin{array}{l}
  (\hat\theta^\otimes)^{-1}=(\hat\theta^{\odot})^{-1}+ (\hat\theta^{\wedge})^{-1}\\
  (\hat\theta^{\odot})^{-1}=\partial x_\alpha\bullet\partial x_\beta\otimes\hat\theta_\gamma^\alpha\otimes
  \hat\theta_\omega^\beta\hskip 3pt
   (1\otimes e^\gamma)^+\bullet(1\otimes e^\omega)^+\\
(\hat\theta^{\wedge})^{-1}=\partial x_\alpha\triangle\partial
x_\beta\otimes\hat\theta_\gamma^\alpha\otimes
  \hat\theta_\omega^\beta\hskip 3pt
   (1\otimes e^\gamma)^+\triangle(1\otimes e^\omega)^+\\
\hat{g}(\zeta^\alpha\otimes e_\alpha,\xi^\beta\otimes
e_\beta)=\widehat{g}_{\gamma\omega}\hat\theta^\gamma_\alpha\zeta^\alpha\otimes\hat\theta^\omega_\beta\xi^\beta,\Rightarrow
\hat{g}_{\alpha\beta}=\widehat{g}_{\gamma\omega}\hat\theta^\gamma_\alpha\otimes\hat\theta^\omega_\beta\\
\hat{g}_{(s)}(\zeta^\alpha\otimes e_\alpha,\xi^\beta\otimes
e_\beta)=\widehat{g}_{(s)}{}_{\gamma\omega}[\hat\theta^\gamma_\alpha\zeta^\alpha,\hat\theta^\omega_\beta\xi^\beta]_+,\Rightarrow
\hat{g}_{(s)}{}_{\alpha\beta}=\widehat{g}_{(s)}{}_{\gamma\omega}\hat\theta^\gamma_\alpha\bullet\hat\theta^\omega_\beta\\
\hat{g}_{(a)}(\zeta^\alpha\otimes e_\alpha,\xi^\beta\otimes
e_\beta)=\widehat{g}_{(s)}{}_{\gamma\omega}[\hat\theta^\gamma_\alpha\zeta^\alpha,\hat\theta^\omega_\beta\xi^\beta]_-,\Rightarrow.
\hat{g}_{(a)}{}_{\alpha\beta}=\widehat{g}_{(a)}{}_{\gamma\omega}\hat\theta^\gamma_\alpha\triangle\hat\theta^\omega_\beta.\\
\end{array}
\right.
\end{equation}
In the particular case that $\widehat{g}$ is Minkowskian, then we
can use for $\widehat{g}_{\gamma\omega}$,
$\widehat{g}_{(s)}{}_{\gamma\omega}$ and
$\widehat{g}_{(a)}{}_{\gamma\omega}$ the corresponding expressions
in {\em(\ref{quantum-vierbein-full-quantum-metric})}, and by using
the property that
$\hat\theta^\delta_\gamma\hat\theta^\gamma_\alpha=\delta^\delta_\alpha$,
we get $\hat{g}_{\gamma\omega}=\underline{\hat{g}}_{\gamma\omega}$,
$\hat{g}_{(s)}{}_{\gamma\omega}=\underline{\hat{g}}_{(s)}{}_{\gamma\omega}$
and
$\hat{g}_{(a)}{}_{\gamma\omega}=\underline{\hat{g}}_{(a)}{}_{\gamma\omega}$.
The {\em controvariant full quantum metric} $\overline{\widehat{g}}$
of $\widehat{g}:M\to Hom_Z(\dot T^2_0M;A)$ is a section
$\overline{\widehat{g}}:M\to Hom_Z(A;\dot T^2_0M)$ such that the
following conditions are satisfied:
\begin{equation}\label{local-coordinates-representation-full-quantum-metric}
\left\{
\begin{array}{l}
 \overline{\widehat{g}}=\partial x_\alpha\otimes\partial
x_\beta\hat{g}^{\alpha\beta},\quad
\widehat{g}=\hat{g}_{\gamma\omega}dx^\gamma\otimes dx^\omega,\quad
 \hat{g}^{\alpha\beta}(p)\in Hom_Z(A\otimes_Z A;A),\quad  \hat{g}_{\alpha\beta}(p)\in \mathop{\widehat{A}}\limits^{2}\\
\hat{g}_{\gamma\omega}(p)\hat{g}^{\gamma\beta}(p)=\delta^\beta_\omega \in\mathbb{R}\subset\widehat{A},\quad
\hat{g}^{\gamma\beta}(p)\hat{g}_{\gamma\omega}(p)=\delta^\beta_\omega \in\mathbb{R}\subset Hom_Z(A\otimes_Z A;A\otimes_Z A).\\
\end{array}
\right.
\end{equation}
The products in {\em(\ref{local-coordinates-representation-full-quantum-metric})} are meant by composition:
\begin{equation}\label{composition-local-coordinates-representation-full-quantum-metric}
\xymatrix@1@C=50pt{A\ar[r]^{\hat g{}^{\alpha\beta}(p)}\ar@/_1pc/[rr]_{\delta^\beta_\gamma}&A\otimes_ZA\ar[r]^{\hat g{}_{\alpha\gamma}(p)}&A\\}
\quad
\xymatrix@1@C=50pt{A\otimes_ZA\ar[r]^{\hat g{}_{\alpha\gamma}(p)}\ar@/_1pc/[rr]_{\delta^\beta_\gamma}&A\ar[r]^{\hat g{}^{\alpha\beta}(p)}&A\otimes_ZA.\\}
\end{equation}
In the following commutative diagram it is shown the pairing working between the fiber bundles $(\dot T^2_0M)^+$ and $\widehat{\dot T^2_0M}$ over $M$.
\begin{equation}\label{commutative-diagram-pairing}
\xymatrix{\widehat{\dot T^2_0M}\times_M(\dot T^2_0M)^+\ar[d]\ar[r]^{<,>}&\widehat{A}\\
Hom_Z(\dot T^2_0M;\dot T^2_0M)\cong \widehat{\dot T^2_0M}\bigotimes_{\widehat{A}}(\dot T^2_0M)^+\ar[r]_(0.77){\TR}&M\times\widehat{A}\ar[u]^{pr_2}\\}
\end{equation}
In particular, one has:
\begin{equation}\label{trace-pairing}
\frac{1}{s}<\overline{\widehat{g}},\widehat{g}>=\frac{1}{s}\hat g{}^{\alpha\beta}\hat g{}_{\alpha\beta}
=\frac{1}{s}\delta^\beta_\beta 1_{\widehat{A}}=1_{\widehat{A}},
\quad s=\left\{\begin{array}{l}
                 m,\hskip 3pt  \dim_AM=m\\
                 m+n,\hskip 3pt  \dim_AM=m|n.\\
               \end{array}\right.
\end{equation}

It is direct to verify that
$\widehat{g}^{\alpha\beta}=\hat\theta^\alpha_\omega\otimes\hat\theta^\beta_\epsilon\underline{\hat{g}}^{\omega\epsilon}$
is the controvariant expression of the full quantum metric
$\widehat{g}_{\alpha\beta}=\hat\theta_\alpha^\gamma\otimes\hat\theta^\delta_\beta\underline{\hat{g}}_{\gamma\delta}$,
when $\underline{\hat{g}}^{\omega\epsilon}$ is the controvariant one
of $\underline{\hat{g}}_{\gamma\delta}$. In other words if
$\underline{\hat{g}}_{\omega\delta}\underline{\hat{g}}^{\omega\epsilon}=\delta^\epsilon_\delta$,
then
$\widehat{g}_{\alpha\gamma}\widehat{g}^{\alpha\beta}=\delta^\beta_\gamma$.
This means that the full quantum metric $\underline{\hat{g}}$,
induced on $A\otimes_{\mathbb{R}}\mathbf{M}_C$ by $\underline{g}$,
is not degenerate, i.e. one has the following short exact sequence:
\begin{equation}\label{short-sequence-non-degeneration}
\xymatrix{0\ar[r]&A\otimes_{\mathbb{R}}\mathbf{M}_C\ar[r]^{{}'\underline{\hat
g}}&(A\otimes_{\mathbb{R}})^+.\\}
\end{equation}
In fact, one can see that $\ker({}'\underline{\hat{g}})=\{0\}$.
Really, ${}'\underline{\hat{g}}(a\otimes v)(b\otimes
u)=ab\underline{g}(v,u)=0$, for all $b\in A$ and $u\in \mathbf{M}_C$
iff $a=0$ or $v=0$. In fact we can take $b=1$ and $u$ any vector of
$\mathbf{M}_C$. So, since $\underline{g}$ is not degenerate, it
follows that cannot be  $\underline{g}(v,u)=0$, for a non zero $v$,
and $\forall u\in \mathbf{M}_C$. The nondegeneration of
$\underline{\hat{g}}$ induces also the following isomorphism
$\widehat{A\otimes_{\mathbb{R}}\mathbf{M}_C}\cong(A\otimes_{\mathbb{R}}\mathbf{M}_C)^+$.

We define {\em quantum supergravity Yang-Mills PDE}, {\em(quantum SG-Yang-Mills PDE)}, a quantum super
Yang-Mills PDE where the quantum super Lie algebra $\mathfrak{g}$ in
the configuration bundle $\pi:W\equiv Hom_Z(TM;\mathfrak{g})\to M$
is a quantum superextension of the Poincar\'e Lie algebra and admits
the following splitting:
\begin{equation}\label{split-quantum-algebra}
\mathfrak{g}=\mathfrak{g}_{\circledR}+\mathfrak{g}_{\copyright}+\mathfrak{g}_{\maltese}
\end{equation}
where $\mathfrak{g}_{\circledR}=A\otimes_{\mathbb{R}}\mathbf{M}_C$, (resp. $\mathfrak{g}_{\copyright}$ is the quantum superextension of the Lorentz part of the Poincar\'e algebra). Here $A$ is a quantum (super)algebra on which is modeled the quantum (super)manifold $M$, and $\mathbf{M}_C$ is the $4$-dimensional Minkowsky vector space. Taking
into account the canonical splitting:
\begin{equation}
Hom_Z(TM;\mathfrak{g})\cong Hom_Z(TM;\mathfrak{g}_{\circledR})\times
Hom_Z(TM;\mathfrak{g}_{\copyright})\times
Hom_Z(TM;\mathfrak{g}_{\maltese})
\end{equation}
we get that the fundamental field $\hat\mu:M\to W$, in a quantum
supergravity Yang-Mills PDE, admits the following canonical splitting:
\begin{equation}
\hat \mu=\hat\mu_{\circledR}+\hat\mu_{\copyright}+\hat\mu_{\maltese}.
\end{equation}
We say that $\hat\mu$ is {\em nondegenerate} if
$\hat\mu_{\circledR}$ identifies, for any $p\in M$, an isomorphism
$\hat\mu_{\circledR}(p):T_pM\cong
A\otimes_{\mathbb{R}}\mathbf{M}_C$, hence $\hat\mu_{\circledR}$ can
be identified with a quantum verbein on $M$:
$\hat\mu_{\circledR}\equiv\hat\theta$. Then we define
$\hat\mu_{\circledR}$, (resp. $\hat\mu_{\copyright}$, resp.
$\hat\mu_{\maltese}$), the {\em vierbein-component}, (resp. {\em
Lorentz-component}, resp. {\em deviatory-component}), of $\hat\mu$.
This property is represented, in local quantum coordinates, by the
fact that in the following formula
\begin{equation}
    \hat\mu=(\hat\mu^K_A dx^A)=(\mu_{\circledR}{}^K_Adx^A+\hat\mu_{\copyright}{}^K_Adx^A+\hat\mu_{\maltese}{}^K_Adx^A)
\end{equation}
one has $(\mu_{\circledR}{}^K_A(p))\in GL(\widehat{A};4)$.

An example that we have just considered in some previous works
\cite{PRA15, PRA16, PRA23, PRA30, PRA32}, is when the quantum Lie
superalgebra $\widehat{\frak g}$ is identified by means of the
following infinitesimal generators: $\{Z_K\}_{1\le K\le
19}\equiv\{J_{\alpha\beta},P_\alpha,\overline{Z},Q_{\beta
i}\}_{0\le\alpha,\beta\le 3;1\le a\le 2}$, such that
$J_{\alpha\beta}=-J_{\beta\alpha},P_\alpha,\overline{Z}\in
Hom_Z(A_0;{\frak g})$, $Q_{\beta i}\in Hom_Z(A_1;{\frak g})$, and
such that nonzero ${\mathbb Z}_2$-graded brackets are reported in
Tab.1.
$$\begin{tabular}{|l|} \hline \multicolumn {1}{|c|}{\bsmall Tab.1 - {\boldmath$\scriptstyle N=2$} Super Poincar\'e algebra.}\\
\hline\hline $\scriptstyle [J_{\alpha\beta},J_{\gamma\delta}]=\eta
_{\beta\gamma} J_{\alpha\delta}+\eta
_{\alpha\delta}J_{\beta\gamma}-\eta
_{\alpha\gamma}J_{\beta\delta}-\eta_{\beta\delta}J_{\alpha\gamma}$\\
 $\scriptstyle [P_\alpha,P_\beta]=0,\quad
[J_{\alpha\beta},P_\gamma]=\eta _{\beta\gamma}P_\beta-\eta
_{\alpha\gamma}P_\beta$\\  $\scriptstyle [J_{\alpha\beta},Q_{\gamma
i}]=(\sigma_{\alpha\beta})^{\mu j}_\gamma Q_{\mu j},\quad [Q_{\beta
i},Q_{\mu j}]=(C\gamma^\alpha)_{\beta\mu}\delta_{ij}P_\alpha+C
_{\beta\mu}\epsilon_{ij}\overline{Z}$\\ \hline\end{tabular}$$

Here $ C_{\alpha\beta}$ is the antisymmetric charge conjugation
matrix, $ \sigma_{\beta\mu}={1\over 4}[\gamma_\beta,\gamma_\mu]$,
with $\gamma^\mu$ the Dirac matrices. $\overline{Z}$ commutes with
all the other ones. Then,
$\hat\mu_{\circledR}{}^{\alpha\beta}_\gamma=\frac{1}{2}\hat\omega^{\alpha\beta}_\gamma$,
$\hat\mu_{\copyright}{}^\mu_\gamma=\hat\theta^\mu_\gamma$,
$\hat\mu_{\maltese}=(\hat A_\gamma,\hat\psi^{\alpha j}_\gamma )$,
where $\hat\omega^{\alpha\beta}_\gamma$ are called {quantum
Levi-Civita connection coefficients}, $\hat\theta^\mu_\gamma$ are
the {\em quantum vierbein components}, $\hat A_\gamma$ are the {\em
quantum electromagnetic field components} and $\hat\psi^{aj}_\gamma$
are the {\em quantum spin ${3\over 2}$ field components}. The
curvature, corresponding to $\hat\mu$, can be locally written in the
form: $ \hat R=Z_K\otimes \hat R^K_{\alpha\beta}dx^\alpha\triangle
dx^\beta$, with $\hat R^K_{\beta\alpha}=(\partial x_{\beta}
\mu^K_{\alpha})+C^K_{IJ}[\mu^I_{\beta},\mu^J_{\alpha}]_+$. The local
expression of the {\em dynamic equation} $ \hat E_{2k}\subset J\hat
D^{2k}(W)$ is resumed in Tab.2, for some quantum Lagrangian $L:J\hat
D^{k}(W)\to\widehat{A}$ of order $k$.
$$\begin{tabular}{|c|l|} \hline \multicolumn {2}{|c|}{\bsmall Tab.2 - Dynamic Equation {\boldmath$\scriptstyle
\hat E_{2k}\subset J\hat D^{2k}(W)$} and Bianchi identity.}\\
\hline\hline {\sevenSl Fields equations}&
                   $\scriptstyle  (\partial\omega^\gamma_{ab}.L)-
                  \partial_\mu(\partial
                  \omega^{\gamma\mu}_{ab}.L)=0
                  \hskip 3pt\hbox{  ({\sevenSl curvature equation})}$
                  \\  &$\scriptstyle  (\partial\theta^\gamma_\alpha.L)
                  -\partial_\mu(\partial\theta^{\gamma\mu}_\alpha.L)=0
                  \hskip 3pt\hbox{  ({\sevenSl torsion equation})}$
                  \\  $\scriptstyle (\hat E_2)$&
                  $\scriptstyle  (\partial\psi^\gamma_{\beta i}.L)
                  -\partial_\mu(\partial\psi^{\gamma\mu}_{\beta i}.L)=0
                  \hskip 3pt\hbox{  ({\sevenSl gravitino equation})}$
                  \\  &$\scriptstyle  (\partial A^\gamma.L)
                   -\partial_\mu(\partial A^{\gamma\mu}.L)=0
                  \hskip 3pt\hbox{  ({\sevenSl Maxwell's equation})}$
                  \\ \hline{\sevenSl Bianchi identity}&
                            $\scriptstyle (\partial x_{\gamma }.R^{ab}_{\beta\alpha })
                            +2[\omega ^a_{e\gamma},R^{eb}_{\beta\alpha}]_+=0$
                            \\  &$\scriptstyle (\partial x_{\gamma }.
                            R^{\alpha }_{\beta \omega })+
                            [\omega ^{\alpha b}_{\gamma},R_{\beta\omega b}]_++
                            (C\gamma ^{\alpha})_{\delta\mu}[\psi^\delta_{j\gamma},
                            \rho _{\beta\omega}^{\mu j}]_+=0$
                            \\  $\scriptstyle(B)$&$\scriptstyle (\partial
x_{\gamma } .\rho ^{\beta i}_{\omega \alpha })
                           +(\sigma_{ab})^{\beta i}_{\delta j}
                            [\omega ^{ab}_{\gamma },
                            \rho _{\omega\alpha}^{\delta j}]_+=0$
                            \\  &$\scriptstyle (\partial x_{\gamma }F_{\beta \alpha
})+
                            C_{\delta\mu}
                            \epsilon_{ij}[\psi _{\gamma}^{\delta i},
                            \rho ^{\mu j}_{\beta\alpha }]_+=0$
                            \\ \hline{\sevenSl Fields}&
                   $\scriptstyle  R^{ab}_{\mu \nu }=(\partial x_{\mu} .\omega ^{ab}_{\nu} )
                   +2[\omega ^a _{e\mu},
                   \omega ^{eb}_{\nu}]_+\hskip 3pt\hbox{({\sevenSl curvature})}$
                   \\  &$\scriptstyle  R^{\alpha }_{\mu \nu }=
                   (\partial x_{\mu }.{\theta} ^{\alpha }_{\nu} )+
                   [{\omega} ^{\alpha }_{\beta\mu},{\theta} ^{\beta}_{\nu}]_+
                   +(C\gamma^\alpha)_{\beta\delta}
                  [ \psi^\beta_{j\mu},\psi^{\delta j}_{\nu}]_+
                   \hskip 3pt\hbox{ ({\sevenSl torsion})}$
                   \\  &$\scriptstyle  {\rho} ^{\beta i}_{\mu \nu }
                   =(\partial x_{\mu} .{\psi }^{\beta i}_{\nu })
                   +(\sigma_{ab})^{\beta i}_{\gamma j}[\omega^{ab}_{\mu},
                   \psi^{\gamma j}_{\nu}]_+\hskip 3pt\hbox{({\sevenSl gravitino})}$
                   \\  &$\scriptstyle  F_{\mu \nu }=(\partial x_{\mu}
.A_{\nu} )+
                   C_{\beta\gamma}\epsilon_{ij}[\psi^{\beta i}_{\mu},\psi^{\gamma j}_{\nu}]_+
                   \hskip 3pt\hbox{({\sevenSl electromagnetic field})}$
                   \\ \hline\end{tabular}$$

So in a quantum SG-Yang-Mills PDE, the quantum Riemannian
metric $\widehat{g}$ is not a fundamental field, but a secondary
field, obtained by means of the quantum verbein
$\hat\theta=\hat\mu_{\circledR}$, that, instead is a fundamental
dynamic field. Of course since there is a relation one-to-one
between quantum verbein and quantum metric, on a locally Minkowskian
quantum (super)manifold, one can choice also quantum metric as a
fundamental field, instead of the quantum verbein. However, in a
quantum SG-Yang-Mills PDE it is more natural to adopt
quantum verbein as indipendent field, since it is just enclosed in
the fundamental field $\hat\mu$. Furthermore, may be useful to
emphasize that the so-called quantum Levi-Civita connection
coefficients $\hat\omega^{\alpha\beta}_\gamma$ are not, in general,
metric coefficients, i.e., do not necessitate to be uniquely
expressed by means of the quantum metric $\widehat{g}$. The name
''Levi-Civita connection coefficients'' is reserved since under
suitable dynamic conditions they can be uniquely identified by the
quantum metric, similarly to what happens in the commutative
differential geometry. However, in general, such property is
dynamically relaxed.

By assuming the following first order Lagrangian function: $L:J\hat
D(W)\to\widehat{A}$, $L\circ Ds=\frac{1}{2}\bar
R^K_{\alpha\beta}\bar R_K^{\alpha\beta}$, $\forall s\in
Q^\infty_w(W)$,\footnote{The rising and lowering of indexes is
obtained by means of the fullquantum metrics $\widehat{g}$ on $M$
and $\underline{g}$ on ${\frak g}$ respectively.} the local
expression of $\widehat{(YM)}$ results given by the equations
reported in Tab.3. Note that the quantum super Yang-Mills equation
is now
$(\partial\bar\mu^A_K.L)-(\partial_B(\partial\bar\mu^{AB}_K.L))=0$.
Furthermore, it results
${(\partial\bar\mu^A_K.L)=[\widehat{C}^H_{KR}\bar\mu^R_C,\bar
R^{[AC]}_H]_+}$ and $(\partial\bar\mu^{AB}_K.L)=\bar R^{BA}_K$.

$$\begin{tabular}{|l|c|} \hline
\multicolumn {2}{|c|}{\bsmall Tab.3 - Local expression of
{\boldmath$\scriptstyle \widehat{(YM)}\subset J\hat
D^2(W)$} and Bianchi identity {\boldmath$\scriptstyle(B)\subset J\hat D^2(W)$}.}\\
\hline \hline $\scriptstyle\hbox{\rsmall(Field equations)}\hskip 2pt
E^A_{K}\equiv-(\partial_{B}.\bar R^{BA}_K)+[\widehat
C^H_{KR}\bar\mu^R_{C},\bar
R^{[AC]}_H]_+=0$&$\scriptstyle \widehat{(YM)}$\\
\hline $\scriptstyle\hbox{\rsmall(Fields)}\hskip 2pt \bar
F^K_{A_1A_2}\equiv \bar R^K_{A_1A_2}-\left[(\partial X_{
A_1}.\bar\mu^K_{A_2})+\frac{1}{2}\widehat{C}{}^K_{IJ}[\bar\mu^I_{A_1},\bar\mu^J_{A_2}]_+\right]=0$&\\
 $\scriptstyle\hbox{\rsmall(Bianchi
identities)}\hskip 2pt B^K_{HA_1A_2}\equiv(\partial X_{H}.\bar
R^K_{A_1A_2})+\frac{1}{2} \widehat{C}{}^K_{IJ}[\bar\mu^I_{H},\bar
R^J_{A_1A_2}]_+=0$&$\scriptstyle (B)$\\
\hline \multicolumn {2}{l}{$\scriptstyle
F^K_{A_1A_2}:\Omega_1\subset J\hat
D(W)\to\mathop{\widehat{A}}\limits^2;\quad
B^K_{HA_1A_2}:\Omega_2\subset J\hat
D^2(W)\to\mathop{\widehat{A}}\limits^3;\quad E^A_{K}:\Omega_2\subset
J\hat D^2(W)\to\mathop{\widehat{A}}\limits^{3}.$}\\ \end{tabular}$$

In Refs.\cite{PRA13, PRA21} it is proved that $\widehat{(YM)}\subset
J\hat D^2(W)$ is formally quantum superintegrable and also
completely quantum superintegrable. That proof works well also in
this situation, since it is of local nature, and remains valid also
for quantum supermanifolds that are only locally quantum
super-Minkowskian ones. Then, by using Theorem \ref{main4} we get

$\Omega_{3|3,w}^{\widehat{(YM)}}\cong\Omega_{3|3,s}^{\widehat{(YM)}}=A_0\otimes_{\mathbb{K}}H_3(W;\mathbb{K})\bigoplus
A_1\otimes_{\mathbb{K}}H_3(W;\mathbb{K})$. Since the fiber of $W$ is
contractible, we have
$\Omega_{3|3,w}^{\widehat{(YM)}}\cong\Omega_{3|3,s}^{\widehat{(YM)}}=A_0\otimes_{\mathbb{K}}H_3(M;\mathbb{K})\bigoplus
A_1\otimes_{\mathbb{K}}H_3(M;\mathbb{K})$. Thus, under the condition
that $H_3(M;\mathbb{K})=0$, one has
$\Omega_{3|3,w}^{\widehat{(YM)}}\cong\Omega_{3|3,s}^{\widehat{(YM)}}=0$,
hence $\widehat{(YM)}$ becomes a quantum extended crystal super PDE.
This is surely the case when $M$ is globally quantum super
Minkowskian. (See Refs.\cite{PRA14, PRA18, PRA21, PRA28, PRA30}.) In
such a case one has
$\Omega_{3|3}^{\widehat{(YM)}}=K_{3|3}\widehat{(YM)}$, where
\begin{equation}
K_{3|3}\widehat{(YM)}\equiv\left\{[N]_{\overline{\widehat{(YM)}}}\in\Omega_{3|3}\widehat{(YM)}
\left|\begin{array}{l}
{N=\partial V,\hskip 2pt\hbox{\rm for some (singular)}}\\
\hbox{\rm $(4|4)$-dimensional quantum}\\
\hbox{\rm supermanifold $V\subset W$ }\\
\end{array}\right.\right\}.
\end{equation}

So $\widehat{(YM)}$ is not a quantum $0$-crystal super PDE. However,
if we consider admissible only integral boundary manifolds, with
orientable classic limit, and with zero characteristic quantum
supernumbers, ({\em full admissibility hypothesis}), one has:
$\Omega_{3|3}^{\widehat{(YM)}}=0$, and $\widehat{(YM)}$ becomes a
quantum $0$-crystal super PDE. Hence we get existence of global
$Q^\infty_w$ solutions for any boundary condition of class
$Q^\infty_w$.

With respect to the commutative exact diagram in
{\em(\ref{Reinhart-bordism-groups-relation})} we get the following
exact commutative diagram

\begin{equation}\label{Yang-Mills-Reinhart-bordism-groups-relation}
\xymatrix{0\ar[r]&K^{\widehat{(YM)}}_{3|3;2}\ar[r]&\Omega_{3|3}^{\widehat{(YM)}}\ar[r]&
\mathop{\Omega}\limits_c{}_{6}^{\widehat{(YM)}}\ar[d]\ar[r]\ar[dr]&0&\\
&0\ar[r]&K^\uparrow_{6}\ar[r]&\Omega^\uparrow_{6}\ar[r]&
\Omega_{6}\ar[r]& 0\\}
\end{equation}
Taking into account the result by Thom on the unoriented cobordism
groups \cite{THO1}, we can calculate
$\Omega_6\cong\mathbb{Z}_2\bigoplus\mathbb{Z}_2\bigoplus\mathbb{Z}_2$.
Then, we can represent $\Omega_6$ as a subgroup of a $3$-dimensional
crystallographic group type $[G(3)]$. In fact, we can consider the
amalgamated subgroup $D_2\times\mathbb{Z}_2\star_{D_2}D_4$, and
monomorphism $\Omega_6\to D_2\times\mathbb{Z}_2\star_{D_2}D_4$,
given by $(a,b,c)\mapsto(a,b,b,c)$. Alternatively we can consider
also $\Omega_6\to D_4\star_{D_2}D_4$.  (See Appendix C in
\cite{PRA25} for amalgamated subgroups of $[G(3)]$.) In any case the
crystallographic dimension of $\widehat{(YM)}$ is $3$ and the
crystallographic space group type are $D_{2d}$ or $D_{4h}$ belonging
to the tetragonal syngony. (See Tab.4 in \cite{PRA25} and, for
further informations, \cite{HAH}.)

Finally, the evaluation of $\widehat{(YM)}$ on a macroscopic shell
$i(M_C)\subset M$ is given by the equations reported in Tab.2.

$$\begin{tabular}{|l|c|} \hline
\multicolumn {2}{|c|}{\bsmall Tab.4 - Local expression of
{\boldmath$\scriptstyle \widehat{(YM)}[i]\subset J\hat D^2(i^*W)$} and Bianchi idenity {\boldmath$\scriptstyle  (B)[i]\subset
J\hat D^2(i^*W)$}.}\\
\hline\hline $\scriptstyle \hbox{\rsmall(Field equations)}\hskip 2pt
(\partial_{\alpha}.\tilde R^{K\alpha\beta})+[\widehat
C^K_{IJ}\tilde\mu^I_{\alpha},\tilde R^{J\alpha\beta}]_+
=0$&$\scriptstyle \widehat{(YM)}[i]$ \\
\hline $\scriptstyle \hbox{\rsmall(Fields)}\hskip 2pt \bar
R^K_{\alpha_1\alpha_2}=(\partial
\xi_{[\alpha_1}.\tilde\mu^K_{\alpha_2]})+\frac{1}{2}\widehat{C}{}^K_{IJ}\tilde\mu^I_{[\alpha_2}\tilde\mu^J_{\alpha_1]}
$&{}\\
$\scriptstyle \hbox{\rsmall(Bianchi identities)}\hskip 2pt (\partial
\xi_{[\gamma}.\tilde R^K_{\alpha_1\alpha_2]})+\frac{1}{2}
\widehat{C}{}^K_{IJ}\tilde\mu^I_{[\gamma}\tilde
R^J_{\alpha_1\alpha_2]}=0$&$\scriptstyle (B)[i]$\\
\hline
\end{tabular}$$

\vskip 0.5cm This equation is also formally quantum superintegrable
and completely quantum superintegrable. Furthermore, the
$3$-dimensional integral bordism group of $\widehat{(YM)}[i]$ and
its infinity prolongation $\widehat{(YM)}[i]_ +\infty$ are trivial,
under the full admissibility hypothesis:
$\Omega_3^{\widehat{(YM)}[i]}\cong\Omega_3^{\widehat{(YM)}[i]_
+\infty}\cong 0$. So equation $\widehat{(YM)}[i]\subset J\hat
D^2(i^*W)$ becomes a quantum $0$-crystal super PDE and it admits global
(smooth) solutions for any fixed time-like $3$-dimensional (smooth)
boundary conditions.
\end{example}

\section{\bf STABILITY IN QUANTUM SUPER PDE's}
\vskip 0.5cm

In this section we shall consider the stability of quantum super
PDE's in the framework of the geometric theory of quantum super
PDE's. We will follow the line just drawn in some our previous
papers on this subject for commutative PDE's, where we have
interpreted stability of PDE's on the ground of their integral
bordism groups and related the quantum bordism of PDE's to Ulam
stability too.

Let us first revise some definitions and results about stability of
mappings and their relations with singularities of mappings,
adapting them to this new category of more complex mathematical noncommutative objects.

\begin{definition}
Let $X$, (resp. $Y$), be a quantum supermanifold of dimension $m|n$,
(resp. $r|s$), with respect to a quantum superalgebra $A=A_0\oplus
A_1$, (resp. $B=B_0\oplus B_1$). We shall assume that the centre
$Z=Z(A)$ of $A$, acts on $B$ that becomes a $Z$-module.\footnote{In
the following, whether it is not differently specified, $X$ and $Y$
are such quantum supermanifolds.} Let $f\in Q_w^\infty(X,Y)$. Then
$f$ is {\em stable} if there is a neighborhood $W_f\subset
Q_w^\infty(X,Y)$ of $f$, in the natural Whitney-type topology of
$Q_w^\infty(X,Y)$, such that every $W_f$ is contained in the orbit
of $f$, via the action of the group $\hat Diff(X)\times \hat
Diff(Y)$.\footnote{Here $\hat Diff(X)$ denotes the group of quantum
diffeomorphisms of a quantum super manifold $X$.} This is equivalent
to say that for any $f'\in W_f$ there exist quantum diffeomorphisms
$g:X\to X$ and $h:Y\to Y$ such that $h\circ f= f'\circ g$.
Furthermore, $f$ is called {\em infinitesimally stable} if there
exist a map $\zeta:X\to TY$, such that $\pi_Y\circ\zeta=f$, where
$\pi_Y:TY\to Y$ is the canonical map, and integrable vector fields
$\nu:Y\to TY$, $\xi:X\to TX$, such that $\zeta=T(f)\circ\xi+\nu\circ
f$. Thus the following diagram is commutative.
\begin{equation}\label{infinitesimal-stability}
    \xymatrix{TX\ar@/^1pc/[d]^{\pi_X}\ar[r]^(.4){T(f)}&TY\bigoplus TY\ar@/^1pc/[d]^{\pi_Y}\ar[r]^{+}&TY\ar[d]^{\pi_Y}\\
    X\ar[u]^{\xi}\ar[urr]^(.3){\zeta}\ar[r]_{f}&Y\ar[u]^(.3){\nu}\ar@{=}[r]&Y\\}
\end{equation}
\end{definition}

\begin{theorem}
Let $X$ be a compact quantum supermanifold and $f:X\to Y$ be quantum smooth. Then $f$ is
stable iff $f$ is infinitesimally stable. Furthermore, if $f$ is a
proper mapping, then does not necessitate assume that $X$ is
compact.\footnote{Recall that a map $f:X\to Y$ between topological
spaces is a {\em proper map} if for every compact subset $K\subset
Y$, $f^{-1}(K)$ is a compact subset of $X$.}
\end{theorem}

\begin{proof}
Note that the infinitesimal stability, requires existence of flows $g_t:X\to X$, $\partial g=\xi$, $h_t:Y\to Y$, $\partial h=\nu$, such that for the infinitesimal variation $\zeta$ of $f_t=h_t\circ f\circ g_t$ one has $\zeta=T(f)\circ\xi+\nu\circ f$. In fact, one has the following lemma.

\begin{lemma}\label{Lie-derivative}
Let $(W,V,\pi_W;\mathbb{B})$ be a bundle of geometric objects in the
category $\mathfrak{Q}_S$ and in the intrinsic sense \cite{PRA1}
\footnote{See also Refs.\cite{PRA2, PRA3, PRA9} for related
subjects.}. Let $\phi:\mathbb{R}\times V\to V$ be a one-parameter
group of $Q^\infty_w$ transformations of $V$, $\xi=\partial\phi$ its
infinitesimal generator and $s:V\to W$ a field of geometric objects,
i.e. a section of $\pi_W$. Then, $\phi$ induces a deformation
$\widetilde{s}$ of $s$ defined by means of the following commutative
diagram
\begin{equation}
\xymatrix{\mathbb{R}\times\mathbb{R}\times V\ar[d]_{(id_{\mathbb{R}},\phi)}\ar[r]^{\widetilde{s}}&W\\
\mathbb{R}\times V\ar[r]^{(id_{\mathbb{R}},s)}&\mathbb{R}\times W\ar[u]^{\mathop{\phi}\limits^{\overline{\circ}}{}_\lambda}\\}
\end{equation}
where $\mathop{\phi}\limits^{\overline{\circ}}{}_\lambda\equiv\mathbb{B}(\phi^{-1}_\lambda)$, $\forall\lambda\in\mathbb{R}$. One has $\widetilde{s}_{(0,0)}=s$. Then, for the infinitesimal variation of $\widetilde{s}$ ({\em Lie derivative} of $s$ with respect to the integrable field $\xi$),  $\partial(\widetilde{s}\circ d):V\to s^*vTW$, one has:
\begin{equation}\label{infinitesimal variation}
    \begin{array}{ll}
      \partial(\widetilde{s}\circ d)& =\partial(s\circ\phi)+\partial(\mathop{\phi}\limits^{\overline{\circ}})\circ s \\
      & =T(s)\circ\xi+\nu\circ s.
    \end{array}
\end{equation}
\end{lemma}

\begin{proof}
This lemma can be proved by copying the intrinsic proof for the
commutative case given in \cite{PRA1}.
\end{proof}
In our case we can consider the following situation, with respect to Lemma \ref{Lie-derivative}, $W\equiv X\times Y$, $V\equiv X$, $\mathbb{B}(g_\lambda)=h_\lambda$ and $s=(id_X,f)$.

Furthermore, in the case that $X$ is compact, the proof follows the
same lines of the proof given by Mather for commutative manifolds
\cite{MATH2}.
\end{proof}

\begin{theorem}
Stable maps $f:X\to Y$ do not necessiate to be dense in
$Q_w^\infty(X,Y)$.
\end{theorem}

\begin{proof}
This is just a corollary of the corresponding theorem for commutative manifolds given by Thom-Levine \cite{LEV1, LEV2}.
\end{proof}

\begin{example}{\em(Submersions and stability).}
Let $X$ be a compact quantum supermanifold. Let $f:X\to Y$ be a
quantum differentiable mapping of maximum possible super-rank. If
$m\ge r>1$,  $n\ge s>1$, $f$ is a {\em quantum submersion} and it is
(infinitesimally) stable.
\end{example}

\begin{example}{\em(Immersions and stability).}
Let $X$ a compact quantum supermanifold. Let $f:X\to Y$ be a quantum
differentiable mapping of maximum possible super-rank. If $m\le r$,
$n\le s$, $f$ is an {\em immersion} and if it is $1:1$ then it is
also stable. (Not all immersions are stable.)
\end{example}

\begin{definition}{\em(Singular solutions of quantum super PDE's).}
Let $\pi:W\to M$ be a fiber bundle, where $M$ is a quantum supermanifold of
dimension $(m|n)$ on the quantum superalgebra $A$ and $W$ is a
quantum supermanifold of dimension $(m|n,r|s)$ on the quantum
superalgebra $B\equiv A\times E$, where $E$ is also a $Z$-module,
with $Z=Z(A)$ the centre of $A$.

Let $E_k\subset JD^k(W)$ be a quantum super PDE. By using the
natural embedding $J\hat D^k(W)\subset \hat J^k_{m|n}(W) $, we can
consider quantum super PDEs $\hat E_k\subset J\hat D^k(W)$ like
quantum super PDEs $\hat E_k\subset \hat J^k_{m|n}(W) $, hence we
can consider solutions of $\hat E_k$ as $(m|n)$-dimensional, (over
$A$), quantum supermanifolds $V\subset\hat E_k$ such that $V$ can be
represented in the neighborhood of any of its points $q'\in V$,
except for a nowhere dense subset $\Sigma(V)\subset V$, of
dimension $\le (m-1|n-1)$, as $N^{(k)}$, where $N^{(k)}$ is the
$k$-quantum prolongation of a $(m|n)$-dimensional (over $A$) quantum
supermanifold $N\subset W$. In the case that $\Sigma(V)=\varnothing$,
we say that $V$ is a {\em regular solution} of $\hat E_k\subset \hat
J^k_{m|n}(W)$. Solutions $V$ of $\hat E_k\subset \hat J^k_{m|n}(W)$,
even if regular ones, are not, in general diffeomorphic to their
projections $\pi_k(V)\subset M$, hence are not representable by
means of sections of $\pi:W\to M$. $\Sigma(V)\subset V$ is the {\em
singular points set} of $V$. Then $V\setminus\Sigma(V)=\bigcup_rV_r$
is the disjoint union of connected components $V_r$. For every of
such components $\pi_{k,0}:V_r\to W$ is an immersion and can be
represented by means of $k$-prolongation of some quantun
supermanifold of dimension $m|n$ over $A$, contained in $W$. Whether
we consider $\hat E_k$ as contained in $J\hat D^k(W)$ then {\em
regular solutions} are locally obtained as image of $k$-derivative
of sections of $\pi:W\to M$. So we can (locally) represent such
solutions by means of mapping $f:M\to E_k$, such that $f=D^ks$, for
some section $s:M\to W$.

We shall also consider
solutions of $\hat E_k\subset \hat J^k_{m|n}(W)$, any subset $V\subset\hat E_k$,
that can be obtained as projections of ones of the previous type,
but contained in some $s$-prolongation $\hat E_{k+s}\subset \hat
J^{k+s}_{m|n}(W)$, $s>0$.

We define {\em weak solutions}, solutions $V\subset \hat E_k$, such that
the set $\Sigma(V)$ of singular points of $V$, contains also
discontinuity points, $q,q'\in V$, with
$\pi_{k,0}(q)=\pi_{k,0}(q')=a\in W$, or $\pi_{k}(q)=\pi_{k}(q')=p\in
M$. We denote such a set by $\Sigma(V)_S\subset\Sigma(V)$, and, in
such cases we shall talk more precisely of {\em singular boundary}
of $V$, like $(\partial V)_S=\partial V\setminus\Sigma(V)_S$.
However for abuse of notation we shall denote $(\partial V)_S$,
(resp. $\Sigma(V)_S$), simply by $(\partial V)$, (resp.
$\Sigma(V)$), also if no confusion can arise.
\end{definition}

\begin{definition}{\em(Stable solutions of quantum super PDE's).}
Let us consider a quantum super PDE $\hat E_k\subset J\hat D^k(W)$, and let us denote $\underline{Sol}(\hat E_k)$ the
set of regular solutions of $E_k$. This has a natural structure of locally convex manifold.
Let $f:X\to E_k$ be a
regular solution, where $X\subset M$ is a smooth $(m|n)$-dimensional
compact manifold with boundary $\partial X$. Then $f$ is {\em
stable} if there is a neighborhood $W_f$ of $f$ in
$\underline{Sol}(\hat E_k)$,
such that each $f'\in W_f$ is equivalent to $f$, i.e., $f$ is
transformed in $f'$ by some integrable vertical symmetries of $\hat E_k$.
\end{definition}

\begin{theorem}\label{deformation}
Let $\hat E_k\subset J\hat D^k(W)$ be a $k$-order quantum super PDE
on the fiber bundle $\pi:W\to M$ in the category of quantum smooth
supermanifolds. Let $s:M\to W$ be a section, solution of $\hat E_k$,
and let $\nu:M\to s^*vTW\equiv \hat E[s]$ be an integrable solution
of the linearized equation $\hat E_k[s]\subset  J\hat D^k(\hat
E[s])$. Then to $\nu$ it is associated a flow
$\{\phi_\lambda\}_{\lambda\in J}$, where $J\subset \mathbb{R}$ is a
neighborhood of $0\in\mathbb{R}$, that transforms $V$ into a new
solution $\widetilde{V}\subset \hat E_k$.
\end{theorem}

\begin{proof}
Let $(x^\alpha,y^j)$ be fibered coordinates on $W$. Let
$\nu=\partial y_j(\nu^j):M\to s^*vTW$ a vertical vector field on $W$
along the section $s:M\to W$. Then $\nu$ is a solution of $\hat
E_k[s]$ iff the following diagram is commutative:
\begin{equation}
\xymatrix{\hat E_k\ar@{^{(}->}[d]&vT\hat
E_k\ar@{^{(}->}[d]\ar[l]&(D^ks)^*vT\hat
E_k\ar@{^{(}->}[l]\ar@{^{(}->}[d]\ar@{=}[r]^(.6){\sim}&
\hat E_k[s]\ar@{^{(}->}[d]\\
J\hat D^k(W)\ar@/_2pc/[dd]_{\pi_k}\ar[d]^{\pi_{k,0}}&vTJ\hat
D^k(W)\ar[l]\ar[d]&(D^ks)^*vTJ\hat D^k(W)\ar@{^{(}->}[l]
\ar@{=}[r]^(.6){\sim}&J\hat D^k(\hat E[s])\ar[d]_{\bar\pi_{k,0}}\ar@/^2pc/[dd]^{\bar\pi_k}\\
W\ar[d]_{\pi}&vTW\ar[l]_{\pi'}&&\hat E[s]\ar@{^{(}->}[ll]\ar[d]^{\bar\pi}\\
M\ar@/^4pc/[uuu]^(.8){D^ks}\ar@/_1pc/[u]^{s}\ar@{=}[rrr]&&&M\ar@/^1pc/[u]^{\nu}\ar@/_4pc/[uuu]_(.8){D^k\nu}\\}
\end{equation}
Then $D^k\nu(p)$ identifies, for any $p\in M$, a vertical vector on
$\hat E_k$ in the point $q=D^ks(p)\in V=D^ks(M)\subset \hat E_k$. On
the other hand infinitesimal vertical symmetries on $E_k$ are
locally written in the form
\begin{equation}
\left\{\begin{array}{l} \zeta=\sum_{0\le|\alpha|\le k}\partial
y_j^\alpha(Y^j_\alpha),\quad 0=\zeta.F^I=<dF^I,\left(
\sum_{0\le r\le k}\partial y_j^{\alpha_1\cdots\alpha_r}(Y^j_{\alpha_1\cdots\alpha_r})\right)>\\
{}\\
Y^j_\alpha=Z^{(0)}_\alpha( Y^j),\quad
Z^{(0)}_\alpha=\partial x_\alpha+\partial y_sy^s_\alpha\\
{}\\
Y^j_{\alpha_1\cdots \alpha_ri}=Z^{(r)}_i(Y^j_{\alpha_1\cdots
\alpha_r}),\quad Z^{(r)}_i=Z_i^{(r-1)}+\partial y_s^{\gamma_1\cdots
\gamma_r}y^s_{\gamma_1\cdots
\gamma_ri}\\
\end{array}\right.
\end{equation}
where $Y^j_\alpha\in Q^\infty_w(U\subset J\hat
D^k(W);\mathop{\widehat{A}}\limits^{|\alpha|}(E))$,
$\mathop{\widehat{A}}\limits^{|\alpha|}(E)\equiv
Hom_Z(\mathop{\overbrace{A\otimes_Z\cdots\otimes_ZA}}\limits^{|\alpha|};E)$,
$0\le|\alpha|\le k$. $\partial y_j^\alpha(q)\in
Hom_Z(\mathop{\widehat{A}}\limits^{|\alpha|}(E);T_qJ\hat D^k(W))$,
$y^j_{\alpha_1\cdots\alpha_r}\in
Q^\infty_w(U;\mathop{\widehat{A}}\limits^{r}(E))$. Then we can see
that solutions of $\hat E_k[s]$ are vertical vector fields $\nu:M\to
s^*vTW\equiv \hat E[s]$, such that their prolongations
$D^k\nu=\zeta\circ D^ks$, for some vertical symmetry $\zeta$ of
$\hat E_k$.

Therefore, the flows of above integrable vertical vector fields, transform
regular solutions $V$ of $\hat E_k$ into new solutions of $\hat
E_k$. Solutions of the linearized equation $\hat E_k[s]$ give
initial conditions for the determination of such vertical flows.
\end{proof}
The following lemmas are also important to understand how the
structure of solutions of $\hat E_k[s]$ are related to the vertical
symmetries of $\hat E_k$. (For complementary informations on the contact structure of $\hat J^k_{m|n}(W)$, see \cite{PRA21}.)

\begin{lemma}{\em(Symmetries of horizontal $k$-order contact ideals).}\label{Horiz-symm}
Let  $\rceil:\hat J^k_{m|n}(W)\to \hat J^{k+1}_{m|n}(W)$  be a {\em
quantum $(k+1)$-connection} on $W$, i.e.,  a $Q^\infty_w$-section of
$\pi_{k+1,k}$. (The restriction of $\rceil$ to $\hat J^k(W)\subset
\hat J^k_{m|n}(W)$ is also called quantum $(k+1)$-connection). Let
$\widehat{\mathfrak{H}}_k(\rceil)$ be the {\em quantum horizontal
$k$-order contact ideal} of $\widehat{\Omega}^\bullet(\hat
J^k_{m|n}(W))$ given by
$\widehat{\mathfrak{H}}_k(\rceil)\equiv\rceil^*\widehat{\mathfrak{C}}_{k+1}(W)$,
where $\widehat{\mathfrak{C}}_{k+1}(W)$ is the contact ideal of
$\hat J^{k+1}_{m|n}(W)$. Locally one can write
$\widehat{\mathfrak{H}}_k(\rceil)=
<\omega^j,\dots,\omega^j_{\alpha_1\dots\alpha_{k-1}},\hat
H^j_{\alpha_1\dots\alpha_k}>$, where
$$\left\{\begin{array}{ll}
           \hat H^j_{\alpha_1\dots\alpha_k}& \equiv
            \rceil^*\omega^j_{\alpha_1\dots\alpha_k}
            =\rceil^*(dy^j_{\alpha_1\dots\alpha_k}-
            y^j_{\alpha_1\dots\alpha_k\beta}dx^\beta)\\
            &\\
            &=dy^j_{\alpha_1\dots\alpha_k}-
            \rceil^j_{\alpha_1\dots\alpha_k\beta}dx^\beta\in\widehat{\Omega}^1(\hat J^k_{m|n}(W)),\\
         \end{array}\right.$$
with $\rceil^j_{\alpha_1\dots\alpha_k\beta}\equiv
y^j_{\alpha_1\dots\alpha_k\beta}\circ\rceil\in
\widehat{\Omega}^0(J^k_{m|n}(W))$. $\widehat{\mathfrak{C}}_k(W)$ is
a subideal of $\widehat{\mathfrak{H}}$. Then the {\em quantum
horizontal $k$-order Cartan distribution}
$\mathbf{H}_k(\rceil)\subset TJ^k_n(W)$ (identified by a
$(k+1)$-connection $\rceil$) is the Cauchy characteristic
distribution associated to $\mathfrak{H}_k(\rceil)$.
$\mathfrak{d}(\mathbf{H}_k(\rceil))$ admits the following local
{\em(canonical basis):}\footnote{For a distribution
$\mathbf{E}\subset TX$ on a manifold $X$, we denote by
$\mathfrak{d}(\mathbf{E})$ the vector space of vector fields on $X$
belonging to $\mathbf{E}$.}

$$\left\{\begin{array}{ll}
\zeta_\alpha=&\partial x_\alpha+\partial y_jy^j_\alpha\\
&\\
&+\cdots+\partial
y_j^{\alpha_1\dots\alpha_{k-1}}y^j_{\alpha_1\dots\alpha_{k-1}\alpha}
+\partial
y_j^{\alpha_1\dots\alpha_k}\rceil^j_{\alpha_1\dots\alpha_k\alpha}.\\
\end{array}\right.$$
For any quantum $(k+1)$-connection $ \rceil$ on $W$, one has the
following direct sum decompositions:
\begin{equation}\label{splitting-quantum-Cartan-distribution}
   \left\{
\begin{array}{l}
 \mathbf{E}^k_{m|n}(W)_q
\cong\mathbf{H}_k(\rceil)_q \bigoplus Hom_Z(S^k(T_aN);\nu_a) \\
 \widehat{\Omega}^1(\hat J^k_{m|n}(W))\cong\widehat{\Omega}^1(\hat
J^k_{m|n}(W))_v \bigoplus\widehat{\mathfrak{H}}_k(\rceil)^1
\end{array}
\right.
\end{equation}
with $a\equiv\pi_{k,0(q)}\in W$, $\rceil(q)=[N]^{k+1}_a$, and
$\mathbf{H}_k(\rceil)_q\equiv T_qN^{(k)}$,
$\widehat{\mathfrak{H}}_k(\rceil)^1\equiv\widehat{\mathfrak{H}}_k(\rceil)\cap
\widehat{\Omega}^1(\hat J^k_{m|n}(W))$. The connection $\rceil$ is
{\em flat}, i.e., with zero curvature, iff the differential ideal
$\widehat{\mathfrak{H}}_k(\rceil)$ is closed, or equivalently, iff
$\mathbf{H}_k(\rceil)$ is involutive. If $\rceil$ is a flat quantum
$(k+1)$-connection on $W$, then one has the
following:\footnote{$\mathbf{C}_{har}(\widehat{\mathfrak{H}}_k(\rceil))$
denotes the characteristic distribution of
$\widehat{\mathfrak{H}}_k(\rceil)$, and
$\mathfrak{char}(\widehat{\mathfrak{H}}_k(\rceil))$ the
corresponding vector space of its vector fields. Furthermore,
$\mathfrak{s}(\widehat{\mathfrak{H}}_k(\rceil))$ denotes the vector
space of infinitesimal symmetries of the ideal
$\widehat{\mathfrak{H}}_k(\rceil)$. $\mathfrak{v}_k(\rceil)$ is the
vector space of vertical infinitesimal symmetries of the ideal
$\widehat{\mathfrak{H}}_k(\rceil)$.}

\begin{equation}\label{quantum-contact-structure-connection}
   \left\{
\begin{array}{l}
\overline{\widehat{\mathfrak{C}}}_k(W)\subset
\widehat{\mathfrak{H}}_k(\rceil) \quad\hbox{\rm as a closed subideal}\\
\mathbf{H}_k(\rceil)\cong\mathbf{C}_{har}(\widehat{\mathfrak{H}}_k(\rceil)); \quad
\mathfrak{char}(\widehat{\mathfrak{H}}_k(\rceil))
\subset\mathfrak{s}(\widehat{\mathfrak{H}}_k(\rceil)). \\
\end{array}
\right.
\end{equation}

$\hat J^k_{m|n}(W)$ is foliated by regular solutions $Z$ such that
$\widehat{\mathfrak{H}}_k(\rceil)|_Z=0$. The leaves of the foliation
are given in implicit form by the following equations:
$f^I(x^\alpha,y^j,\dots,y^j_{\alpha_1\dots\alpha_k})=\kappa^I\in
B_k$, $ 1\le I\le p+q$, $\dim J^k_n(W)-(p|q)=m|n$, where $f^I$
represent a complete independent system of primitive integrals of
the linear system of PDEs $(\zeta_\alpha.f)=0$, $1\le\alpha\le m+n$,
where $\zeta_\alpha$ is a basis (e.g., the canonical basis) of the
horizontal distribution $\mathbf{H}_k(\rceil)$.\footnote{A (local)
section $s$ of $\pi$ identifies a flat (local) $ (k+1)$-connection
$\rceil^j_{\alpha_1\dots\alpha_{k+1}}\equiv (\partial
x_{\alpha_1}\dots\partial x_{\alpha_{k+1}}.s^j)$.} Any
$\zeta\in\mathfrak{s}((\mathbf{H}_k(\rceil))$ has the following
local representation:
\begin{equation}\label{symmetries-horizontal-k-contact-ideal}
    \left\{
\begin{array}{ll}
  \zeta=& \zeta_\alpha(X^\alpha)+\partial y_j(Y^j)+
\partial y_j^\alpha(\zeta_\alpha.Y^j)+
\partial y_j^{\alpha_1\alpha_2}(\zeta_{\alpha_1}\zeta_{\alpha_2}.Y^j)\\
 &\\
 &+\cdots+
\partial y_j^{\alpha_1\dots\alpha_k}(\zeta_{\alpha_1}\dots\zeta_{\alpha_k}.Y^j),\\
\end{array}\right.
\end{equation}

for any choice of $x^\alpha\in Q^\infty_w(U\subset \hat
J^k_{m|n}(W),A)$, ${1\le \alpha\le m+n}$, and $Y^j\in
Q^\infty_w(U\subset \hat J^k_{m|n}(W),E)$, ${1\le j\le r+s}$, such
that
\begin{equation}\label{cond-symmetries-horizontal-k-contact-ideal}
\left\{
\begin{array}{ll}
(\zeta_{\alpha_1}\dots\zeta_{\alpha_k}.Y^j)=& (\partial
y_i.\rceil^j_{\alpha_1\dots\alpha_k})Y^i+ (\partial
y_i^\gamma.\rceil^j_{\alpha_1\dots\alpha_k})(\zeta_\gamma.Y^i)\\
&\\
&+\cdots+ (\partial y_i^{\gamma_1\dots\gamma_k}.
\rceil^j_{\alpha_1\dots\alpha_k})
(\zeta_{\gamma_k}\dots\zeta_{\gamma_1}.Y^i).\\
\end{array}\right.
\end{equation}

The space $\mathfrak{s}(\mathfrak{H}_k(\rceil))$ admits the following
direct sum decomposition:
$$\mathfrak{s}(\mathfrak{H}_k(\rceil))\cong\mathfrak{d}(\mathbf{H}_k(\rceil))
\bigoplus\mathfrak{v} _k(\rceil),$$ where $\mathfrak{v} _k(\rceil)$
is the collection of all vectors of the form
\begin{equation}\label{}
\left\{
\begin{array}{ll}
\xi=&\zeta-\zeta_\alpha(X^\alpha) =\partial y_j(Y^j)+
\partial y_j^\alpha(\zeta_\alpha.Y^j)+
\partial y_j^{\alpha\beta}(\zeta_\alpha\zeta_\beta.Y^j)\\
&\\
&+\dots+
\partial y_j^{\alpha_1\dots\alpha_k}(\zeta_{\alpha_1}\dots\zeta_{\alpha_k}.Y^j),\\
\end{array}\right.
\end{equation}

for any choice of $Y^j\in Q^\infty_w(U\subset \hat J^k_{m|n}(W),E)$,
${1\le j\le r+s}$, such that conditions
$(\ref{cond-symmetries-horizontal-k-contact-ideal})$ are satisfied.
${\frak s}(\mathfrak{H}_k(\rceil))$ is a Lie algebra that admits the
subalgebra $\mathfrak{d}(\mathbf{H}_k(\rceil))$ as an ideal.
\end{lemma}

The general local expression for the symmetries of the
$(m|n)$-dimensional involutive Cartan distribution
$\mathbf{E}_{\infty}(W)\subset T \hat J^\infty_{m|n}(W)$, can be
also obtained by equations
{\em(\ref{symmetries-horizontal-k-contact-ideal})} with all $k>0$,
and forgetting conditions
(\ref{cond-symmetries-horizontal-k-contact-ideal}).\footnote{In fact
the Cartan distribution on $\hat J^\infty_n(W)$ can be considered an
horizontal distribution induced by the canonical connection
identified by the local canonical basis $\zeta_\alpha=\partial
x_\alpha+\sum_{|\beta|\ge 0}y^j_{\alpha\beta}\partial y_j^\beta$
just generating $\mathbf{E}^{\infty}_n(W)$.} So we get the following
expression for $\zeta\in\mathfrak{s}(\mathbf{E}^\infty_{m|n}(W))$:
\begin{equation}\label{infty-contact-symmetries}
\left\{
\begin{array}{ll}
  \zeta&=\partial_\alpha(X^\alpha)+\sum_{r\ge 0}\partial y_j^{\alpha_1\cdots\alpha_r}(Y^j_{\alpha_1\cdots\alpha_r})\\
  \partial_\alpha&=\partial x_\alpha+\sum_{r\ge 0}\partial y_j^{\alpha_1\cdots\alpha_r}(y^j_{\alpha\alpha_1\cdots\alpha_r})\\
  Y^j_{\alpha_1\cdots\alpha_r}&=(\partial_{\alpha_1}\cdots\partial_{\alpha_r}.Y^j),
  \quad Y^j\in Q^\infty_w(U\subset \hat J^\infty_{m|n}(W),E),
{1\le j\le r+s}.\\
\end{array}\right.
    \end{equation}
Then the canonical splitting
$T_q\hat J^\infty_{m|n}(W)\cong(\mathbf{E}^\infty_{m|n}(W))_q\bigoplus
vT_q\hat J^\infty_{m|n}(W)$, $q\in \hat J^\infty_{m|n}(W)$, gives the following
splitting in
$\mathfrak{s}(\mathbf{E}^\infty_{m|n}(W))=\mathfrak{d}(\mathbf{E}^\infty_{m|n}(W))\bigoplus
\mathfrak{v}_\infty$, $\zeta=\zeta_o+\zeta_v$, with
$\zeta_o=\partial_\alpha(X^\alpha)$ and $\zeta_v=\sum_{r\ge
0}\partial
y_j^{\alpha_1\cdots\alpha_r}(Y^j_{\alpha_1\cdots\alpha_r})$, where $Y^j_{\alpha_1\cdots\alpha_r}$
are given in (\ref{infty-contact-symmetries}).

\begin{definition}\label{fun-stable-PDE}
Let $\hat E_k\subset \hat J^k_{m|n}(W)$, where $\pi:W\to M$ is a
fiber bundle, in the category of quantum smooth supermanifolds. We
say that $\hat E_k$ is {\em functionally stable} if for any compact
regular solution $V\subset \hat E_k$, such that $\partial
V=N_0\bigcup P \bigcup N_1$ one has quantum solutions
$\widetilde{V}\subset \hat J^{k+s}_{m|n}(W)$, $s\ge 0$, such that
$\pi_{k+s,0}(\widetilde{N}_0\sqcup \widetilde{N}_1)=\pi_{k,0}(N_0\sqcup
N_1)\equiv X\subset W$, where $\partial
\widetilde{V}=\widetilde{N}_0\bigcup
\widetilde{P}\bigcup\widetilde{N}_1$.

We call the set $\Omega[V]$ of such solutions $\widetilde{V}$ the
{\em full quantum situs} of $V$. We call also each element
$\widetilde{V}\in \Omega[V]$ a {\em quantum fluctuation} of
$V$.\footnote{Let us emphasize that to $\Omega[V]$ belong also (non
necessarily regular) solutions $V'\subset E_k$ such that $N_0'\sqcup
N_1'=N_0\sqcup N_1$, where $\partial V'=N_0'\bigcup P'\bigcup N_1'$.}
\end{definition}

\begin{definition}\label{infinitesimal}
We call {\em infinitesimal bordism} of a regular solution $V\subset
\hat E_k\subset J\hat D^k(W)$ an element $\widetilde{V}\in\Omega[V]$,
defined in the proof of Theorem \ref{deformation}. We denote by
$\Omega_0[V]\subset \Omega[V]$ the set of infinitesimal bordisms of
$V$. We call $\Omega_0[V]$ the {\em infinitesimal situs} of $V$.
\end{definition}

\begin{definition}\label{fun-stable}
Let $\hat E_k\subset \hat J^k_{m|n}(W)$, where $\pi:W\to M$ is a fiber bundle, in
the category of quantum smooth supermanifolds. We say that a regular solution
$V\subset \hat E_k$, $\partial V=N_0\bigcup P \bigcup N_1$, is {\em
functionally stable} if the infinitesimal situs $\Omega_0[
V]\subset\Omega[V]$ of $V$ does not contain singular infinitesimal
bordisms.
\end{definition}

\begin{theorem}\label{main}
Let $\hat E_k\subset \hat J^k_{m|n}(W)$, where $\pi:W\to M$ is a fiber bundle, in
the category of quantum smooth supermanifolds. If $\hat E_k$ is quantum formally integrable
and completely quantum superintegrable, then it is functionally stable as well as
Ulam-extended superstable.

A regular solution $V\subset \hat E_k$ is stable iff it is functionally
stable.
\end{theorem}

\begin{proof}
In fact, if $\hat E_k$ is quantum formally integrable and completely quantum superintegrable,
we can consider, for any compact regular solution $V\subset \hat E_k$,
its $s$-th prolongation $V^{(s)}\subset (\hat E_k)_{+s}\subset
\hat J^{k+s}_{m|n}(W)$. Since one has the following short exact sequence

\begin{equation}
\xymatrix{\Omega_{m-1|n-1}^{(\hat E_k)_{+s}}\ar[r]&\Omega_{m-1|n-1}((\hat E_k)_{+s})\ar[r]&0\\}
\end{equation}

where $\Omega_{m-1|n-1}^{(\hat E_k)_{+s}}$, (resp.
$\Omega_{m-1|n-1}((\hat E_k)_{+s})$), is the integral bordism group,
(resp. quantum bordism group),\footnote{Here the considered bordism
groups are for admissible non-necessarily closed Cauchy
hypersurfaces.} we get that there exists a solution
$\widetilde{V}\subset \hat J^{k+s}_{m|n}(W)$ such that
\begin{equation}
\left\{\begin{array}{l}
\partial \widetilde{V}=\widetilde{N}_0\bigcup\widetilde{P}\bigcup\widetilde{N}_1;\quad
\partial V^{(s)}=N_0^{(s)}\bigcup P^{(s)}\bigcup N_1^{(s)}\\
\widetilde{N}_0=N_0^{(s)};\quad
\widetilde{N}_1=N_1^{(s)}.\\
\end{array}\right.
\end{equation}
Then, as a by-product we get also:
$\pi_{k+s,0}(\widetilde{N}_0\sqcup\widetilde{N}_1)=\pi_{k,0}(N_0\sqcup
N_1)\subset W$. Therefore, $\hat E_k$ is functionally stable.
Furthermore, $\hat E_k$ is also Ulam-extended superstable, since the
integral bordism group $\Omega_{m-1|n-1}^{\hat E_k}$ for smooth solutions and
the integral bordism group $\Omega_{m-1|n-1,s}^{\hat E_k}$ for singular
solutions, are related by the following short exact sequence:
\begin{equation}
\xymatrix{0\ar[r]&\hat K_{m-1|n-1,s}^{\hat
E_k}\ar[r]&\Omega_{m-1|n-1}^{\hat E_k}\ar[r]&
\Omega_{m-1|n-1,s}^{\hat E_k}\ar[r]&0.\\}
\end{equation}
This implies that in the neighborhood of each smooth solution there
are singular solutions.

Finally a regular solution $V\subset \hat E_k$ is stable iff the set
of solutions of the corresponding linearized equation $\hat E_k[V]$
does not contains singular solutions. But this is just the
requirement that $\Omega_0[V]$ does not contains singular solutions.
Therefore, $V$ is stable if it is functionally stable and vice
versa. More precisely if $f=D^ks:X\to \hat E_k$ is a stable solution
of $\hat E_k$, then there exists an open set $W_s\subset
\underline{Sol}(\hat E_k)$ such that for any $s'\in W_s$, $s'$ is
equivalent to $s$.\footnote{Recall that $Q_w^\infty(W)$ has a
natural structure of quantum smooth supermanifold modeled on locally
convex topological vector fields. $\underline{Sol}(\hat E_k)$ is a
closed submanifold of $\underline{Sol}(\hat E_k)\subset
Q_w^\infty(W)$. (For details see ref.\cite{PRA9}.)} Let us consider
the tangent space $T_s\underline{Sol}(\hat E_k)$. One has the
following isomorphism
\begin{equation}
T_s\underline{Sol}(\hat E_k)\cong\left\{\zeta\in
(Q_w^\infty)_0((D^ks)^*vT\hat E_k)\hskip 2pt|\hskip 2pt \exists \xi\in
T_sQ_w^\infty(W), \zeta=|_k\circ D^k\xi\right\}\cong\Omega_0[V]
\end{equation}
where $|_k$ is the canonical isomorphism $J\hat D^k(s^*vTW)\cong
(D^ks)^*vTJ\hat D^k(W)$, and $V=D^ks(X)\subset \hat E_k$. Since $W_s$ is
open in $\underline{Sol}(\hat E_k)$, one has also the following
isomorphism $T_{s'}W_s\cong T_s\underline{Sol}(\hat E_k)$. Thus also to
$s'$ there correspond vector fields $\zeta\in T_{s'}W_s$ that must
be regular ones, i.e., without singular points. Therefore
$\Omega_0[V]$ cannot contain singular solutions, hence $V$ is
functionally stable. Vice versa, if $V$ is functionally stable, then
we can find an open neighborhood $W_s\subset \underline{Sol}(\hat E_k)$
built by perturbing $V$ with all the flows induced by the regular
vector fields belonging to $\Omega_0[V]$. This set is an open set of
$W_s\subset\underline{Sol}(\hat E_k)$ since its tangent space at any of
its point $s'$ is isomorphic to $T_{s'}\underline{Sol}(\hat E_k)$, since
this last is isomorphic to $\Omega_0[V]$. Furthermore, any two of
such points of such an open set are equivalent since they can be
related both to $s$ by local diffeomorphisms. Therefore $V$ that is
functionally stable is also stable.
\end{proof}

\begin{remark}
Let us emphasize that  the definition of functionally stable quantum super PDE
interprets in pure geometric way the definition of Ulam superstable
functional equation just adapted to PDE's.\footnote{Let us recall the concept of Ulam stability.
Let $F$ be a functional
space, i.e., a space of suitable applications $f:X\to Y$ between
finite dimensional Riemannian manifolds $X$ and $Y$. Let $E$ be a
Banach space and $S$ a subset of $X^n$. Let us consider a functional
equation:
\begin{equation}\label{fueq}
G(f,q^1,\cdots,q^n)=0,\quad \forall (q^1,\cdots,q^n)\in S\subset X^n
\end{equation}
defined by means of a mapping $G:F\times X^n\to E$,
$(f,(q^1,\cdots,q^n))\mapsto G(f,(q^1,\cdots,q^n))$. We say that
such a functional equation is {\em Ulam-extended stable} if for any
function $\widetilde{f}\in F$, satisfying the inequality
\begin{equation}\label{fueq-ex}
\|G(\widetilde{f},(q^1,\cdots,q^n))\|\le\varphi(q^1,\cdots,q^n),\quad
\forall (q^1,\cdots,q^n)\in X^n\end{equation} with
$\varphi:X^n\to[0,\infty)$ fixed, there exists a solution $f$ of
{\em \ref{fueq}} such that $d_Y(\widetilde{f}(q),f(q))\le\Phi(q)$,
$\forall q\in X$, for suitable $\Phi:X\to[0,\infty)$. Here
$d_Y(.,.)$ is the metric induced by the Riemannian structure of
            $Y$. More precisely,
$d_Y(a,b)=inf\int_{[0,1]}\sqrt{g_Y(\dot\gamma(t),\dot\gamma(t))}dt$,
            $\forall a,b\in Y$, where $\gamma\in C^1([0,1],Y)$,
            with $\gamma(0)=a$, $\scriptstyle \gamma(1)=b$ and $g_Y$
            the Riemannian metric on $Y$. One has the following important propositions:
            (i) If $\gamma\in C^1([0,1],Y)$, such that $\gamma(0)=a$,
            $\gamma(1)=b$, and
            $\int_{[0,1]}\sqrt{g_Y(\dot\gamma(t),\dot\gamma(t))}dt=d_Y(a,b)$,
            then $\gamma([0,1])\subset Y$ is a geodesic.
            If in addition, $\gamma$ has constant speed, then $\gamma$
            is a $C^\infty$ geodesic.
            (ii)(de Rham) When $Y$ is geodesically complete, then given $a,b\in Y$,
            there exists at least one geodesic $\gamma$ in $Y$ connecting
            $a$ to $b$ with
            $\int_{[0,1]}\sqrt{g_Y(\dot\gamma(t),\dot\gamma(t))}dt=d_Y(a,b)$.
            (iii)(Hopf, Rinov) The geodesic completeness of $Y$ is equivalent to the completeness
            of $Y$ as a metric space, which is equivalent to the statement that a subset of
            $Y$ is compact iff it is closed and bounded.
            (iv) $Y$ is complete whenever it is compact. If each solution $\widetilde{f}\in F$ of the inequality
            (\ref{fueq-ex})
is either a solution of the functional equation (\ref{fueq}) or
satisfies some stronger conditions, then we say that equation (
\ref{fueq}) is {\em Ulam-extended-superstable}.}
\end{remark}

\begin{definition}
We say that $\hat E_k\subset J\hat D^k(W)$ is a {\em stable extended
crystal quantum super PDE} if it is an extended crystal quantum super PDE that is functionally
stable and all its regular quantum smooth solutions are (functionally)
stable.
\end{definition}

\begin{definition}
We say that $\hat E_k\subset J\hat D^k(W)$ is a {\em stabilizable
extended crystal quantum super PDE} if it is an extended crystal quantum super PDE and to $\hat E_k$
can be canonically associated a stable extended crystal quantum super PDE
${}^{(S)}\hat E_k\subset J\hat D^{k+s}(W)$. We call ${}^{(S)}\hat E_k$ just
the {\em stable extended crystal quantum super PDE of} $\hat E_k$.
\end{definition}

We have the following criteria for functional stability of solutions
of qunatum super PDE's and to identify stable extended crystal quantum super PDE's.
\begin{theorem}{\em(Functional stability criteria).}\label{criteria-fun-stab}
Let $\hat E_k\subset J\hat D^k(W)$ be a $k$-order quantum formally integrable
and completely quantum superintegrable quantum super PDE on the fiber bundle $\pi:W\to M$.

{\em 1)} If the symbol $\hat g_k=0$, then all the quantum smooth regular
solutions $V\subset \hat E_k\subset J\hat D^k(W)$ are functionally
stable, with respect to any non-weak perturbation. So $\hat E_k$ is a
stable extended crystal.

{\em 2)} If $\hat E_k$ is of finite type, i.e., $\hat g_{k+r}=0$,
for $r>0$, then all the quantum smooth regular solutions $V\subset
\hat E_{k+r}\subset J\hat D^{k+r}(W)$ are functionally stable, with
respect to any non-weak perturbation. So $\hat E_k$ is a
stabilizable extended crystal with stable extended crystal
${}^{(S)}\hat E_k=\hat E_{k+r}$.

{\em 3)} If $V\subset(\hat E_k)_{+\infty}\subset J\hat D^\infty(W)$ is a
smooth regular solution, then $V$ is functionally stable, with
respect to any non-weak perturbation. So any quantum formally integrable end
completely quantum superintegrable quantum super PDE $\hat E_k\subset J\hat D^k(W)$, is a
stabilizable extended crystal, with stable extended crystal
${}^{(S)}\hat E_k=(\hat E_k)_{+\infty}$.
\end{theorem}

\begin{proof}
We shall use the following lemmas.

\begin{lemma}\label{eq-prolongations}
Let $\hat E_k\subset J\hat D^k(W)$ be a quantum formally integrable and
completely quantum superintegrable quantum superPDE the fiber bundle $\pi:W\to M$. Then for
any quantum smooth regular solution $s:M\to W$, one has the following
canonical isomorphism: $(\hat E_k[s])_{+h}\cong((\hat E_k)_{+h})[s]$, $\forall
h\ge 1, \infty$.
\end{lemma}

\begin{proof}
In fact one has the following commutative diagram.

\begin{equation}\label{prolongations}
\left\{\begin{array}{ll}
  (\hat E_k[s])_{+h}& =J\hat D^h((D^ks)^*vT\hat E_k)\bigcap J\hat D^{k+h}(s^*vTW) \\
  & \cong (D^{k+h}s)^*vTJ\hat D^{h}(\hat E_k)\bigcap(D^{k+h}s)^*vTJ\hat D^{k+h}(W)\\
& \cong (D^{k+h}s)^*vT\left(J\hat D^{h}(\hat E_k)\bigcap J\hat D^{k+h}(W)\right)\\
&\cong(D^{k+h}s)^*vT((\hat E_k)_{+h})=((\hat E_{k})_{+h})[s].\\
\end{array}
\right.
\end{equation}

\end{proof}

\begin{lemma}\label{symbol-prolongations}
Let $\hat E_k\subset J\hat D^k(W)$ be a formally integrable and
completely integrable PDE the fiber bundle $\pi:W\to M$. Let $\hat
g_k=0$. Then also the prolonged equations  $(\hat E_k)_{+r}$,
$\forall r\ge 1, \infty$, have their symbols zero: $(\hat
g_k)_{+r}=0$, $\forall r\ge 1, \infty$.
\end{lemma}

\begin{proof}
In fact, from the definition of symbol and prolonged symbols, it
follows that the prolonged symbols coincide with the symbols of the
corresponding prolonged equations.
\end{proof}

1) This follows from Lemma \ref{eq-prolongations} and from the fact
that if $\hat g_k=0$ is also $\hat g_k[s]=0$. This excludes that $\hat E_k[s]$
could have singular solutions. Furthermore, Lemma
\ref{symbol-prolongations} excludes also that there are singular
(nonweak) solutions in the prolonged equations $\hat E_k[s]_{+r}$,
$\forall r\ge 1, \infty$.

2) If $\hat E_k$ is of finite type, with $\hat g_{k+r}=0$, then it is also
$\hat g_{k+r}[s]=0$. Then $\hat E_{k+r}[s]$ cannot have singular (nonweak)
solutions.

3) $\hat E_\infty$ has zero symbol, hence also $\hat E_\infty[s]$ has zero
symbol and cannot have singular (nonweak) solutions.

(So the proof follows the same lines drawn for commutative PDE's.)
\end{proof}

\begin{theorem}{\em(Functional stable solutions and $(k+1)$-connections).}\label{Criterion-fun-stable-sol-conn}
Let $\hat E_k\subset \hat J^k_{m|n}(W)$ be a quantum formally
integrable and completely quantum superintegrable quantum super PDE.
Let $\rceil$ be a quantum flat $(k+1)$-connection, such that
$\rceil|_{\hat E_k}$ is a $Q^\infty_w$-section of the affine fiber
bundle $\pi_{k+1,k}:(\hat E_k)_{+1}\to \hat E_k$ . Then, the
sub-equation ${}^{\rceil}E_k\subset \hat E_k$ identified, by means
of the ideal $\mathfrak{H}(\rceil)|_{\hat E_k}$, is formally
integrable and completely quantum superintegrable sub-equation with
zero symbol ${}^{\rceil}\hat g_k$. Then ${}^{\rceil}\hat E_k\subset
\hat E_k$ is functionally stable and Ulam-extended superstable.
Furthermore any regular quantum smooth solution $V\subset
{}^{\rceil}\hat E_k$ is also functionally stable in ${}^{\rceil}\hat
E_k$, with respect to any non weak perturbation.
\end{theorem}

\begin{proof}
In fact, one has the following commutative diagram of exact lines.

\begin{equation}
\xymatrix{\Omega_{m-1|n-1}^{(\hat E_k)_{+s}}\ar[d]
\ar[r]&\Omega_{m-1|n-1}((\hat E_k)_{+s})\ar[d]
\ar[r]&0\\
\Omega_{m-1|n-1}^{({}^{\rceil}\hat E_k)_{+s}}\ar[d]
\ar[r]&\Omega_{m-1|n-1}(({}^{\rceil}\hat E_k)_{+s})\ar[d]
\ar[r]&0\\
0&0&\\}
\end{equation}

Furthermore, since ${}^{\rceil}\hat g_k=0$, ${}^{\rceil}\hat E_k$ is of finite
type, hence its smooth regular solutions are functionally stable.
\end{proof}

Taking into account the meaning that connections assume in any
physical theory, we can give the following definition.

\begin{definition}
Let $\hat E_k\subset \hat J^k_{m|n}(W)$ be a quantum formally integrable and completely
quantum superintegrable quantum super PDE. Let $\rceil$ be a flat quantum $(k+1)$-connection, such that
$\rceil|_{\hat E_k}$ is a $Q^\infty_w$-section of the affine fiber bundle
$\pi_{k+1,k}:(\hat E_k)_{+1}\to \hat E_k$ . We call the couple $(\hat E_k,\rceil)$
a {\em polarized quantum super PDE}. We call also {\em polarized quantum super PDE}, a couple
$(\hat E_k,{}^{\rceil}\hat E_k)$, where ${}^{\rceil}\hat E_k\subset \hat E_k$, is
defined in Theorem \ref{Criterion-fun-stable-sol-conn}. We call
${}^{\rceil}\hat E_k$ a {\em polarization} of $\hat E_k$.
\end{definition}

\begin{cor}
Any quantum smooth regular solutions of a polarization of a polarized couple
$(\hat E_k,{}^{\rceil}\hat E_k)$, is functionally stable, with respect to any
non-weak perturbation.
\end{cor}

\begin{theorem}{\em(Finite stable extended crystal PDE)}\label{finite-stable-extended-crystal-PDE}
Let $\hat E_k\subset J\hat D^k(W)$ be a quantum formally integrable and
completely quantum superintegrable quantum super PDE, such that the centre $Z(A)$ of the quantum superalgebra $A$, model for $M$, is Noetherian. Then, under suitable {\em finite
ellipticity conditions}, there exists a stable extended crystal quantum super PDE
${}^{(S)}\hat E_k$ canonically associated to $\hat E_k$, i.e., $\hat E_k$ is a
stabilizable extended crystal.
\end{theorem}

\begin{proof}
In fact, we can use the following lemma.

\begin{lemma}{\em(Finite stability criterion).}\label{finite-stability-criterion}
Let $\hat E_k\subset J\hat D^k(W)$ be a quantum formally integrable and
completely quantum superintegrable quantum super PDE, such that the centre $Z(A)$ of the quantum superalgebra $A$, model for $M$, is Noetherian. Then there exists an integer $s_0$ such
that, under suitable {\em finite ellipticity conditions}, any
regular quantum smooth solution $V\subset (\hat E_k)_{+s_0}$ is functionally
stable.
\end{lemma}

\begin{proof}
Under the hypotheses that $Z(A)$ is Noetherian, the proof follows the same line of the commutative case.
\end{proof}

Let us, now, use the hypothesis that $\hat E_k$ is quantum formally integrable
and completely quantum superintegrable. Then all its regular quantum smooth solutions are
all that of $(\hat E_k)_{+s_0}$. In fact, these are all the solutions of
$(\hat E_k)_{+\infty}\subset J\hat D^\infty(W)$. However, even if a
smooth regular solution $V\subset \hat E_k$, and their
$s_0$-prolongations, $V^{(s_0)}\subset (\hat E_k)_{+s_0}$, are equivalent
as solutions, they cannot be considered equivalent from the
stability point of view !!! In fact, $\hat E_k$ can admit singular
solutions, instead for $(\hat E_k)_{+s_0}$ these are forbidden.
Therefore, for $\hat E_k[s]$ singular perturbations are possible, i.e.
are possible infinitesimal vertical symmetries of $\hat E_k$, in a
neighborhood of the solution $s$, having singular points. Instead
for $(\hat E_k)_{+s_0}[s]$ all solutions are without singular points,
hence $s$ considered as solution of  $(\hat E_k)_{+s_0}$ necessitates to
be functionally stable.

By conclusions, $\hat E_k$, under the finite ellipticity conditions is a
stabilizable extended crystal quantum super PDE, and its stable extended crystal
quantum super PDE is ${}^{(S)}\hat E_k=(\hat E_k)_{+s_0}$, for a suitable finite number
$s_0$.
\end{proof}

\begin{remark}
With respect to a quantum frame \cite{PRA15, PRA21, PRA22}, we can
consider the perturbation behaviours of global solutions for
$t\to\infty$, where $t$ is the proper time of the quantum frame.
Then, we can talk about asymptotic stability by reproducing similar
situations for commutative PDE's. (See Refs.\cite{PRA24, PRA29}.) In
particular we can consider the concept of ''averaged stability''
also for solutions of quantum (super) PDE's. With this respect, let
us recall the following definition and properties of quantum
(pseudo)Riemannian supermanifold given in \cite{PRA15, PRA23}.
\end{remark}

\begin{definition}\cite{PRA15, PRA23}
A {\em quantum (pseudo)Riemannian supermanifold} $(M.\widehat{A})$
is a quantum supermanifold $M$ of dimension $(m|n)$ over a quantum
superalgebra $A$, endowed with a $Q^\infty_w$ section
$\widehat{g}:M\to Hom_Z(TM\otimes_ZTM;A)$ such that the induced
homomorphisms $T_pM\to(T_pM)^+$, $\forall p\in M$, are
injective.\end{definition}

\begin{proposition}\cite{PRA15, PRA23}
In quantum coordinates $\widehat{g}(p)$ is represented by a matrix
$\widehat{g}_{\alpha\beta}(p)\in\mathop{\widehat{A}}\limits^2{}_{00}(A)\times\mathop{\widehat{A}}
\limits^2{}_{10}(A)\times\mathop{\widehat{A}}\limits^2{}_{01}(A)\times\mathop{\widehat{A}}\limits^2{}_{11}(A)$.
The corresponding dual quantum metric gives
$\widehat{g}^{\alpha\beta}(p)\in\mathop{\widehat{A}}\limits^2{}^{00}(A)\times\mathop{\widehat{A}}\limits^2{}^{10}(A)\times\mathop{\widehat{A}}\limits^2{}^{01}(A)\times\mathop{\widehat{A}}\limits^2{}^{11}(A)$,
with $\mathop{\widehat{A}}\limits^2{}^{ij}(A)\equiv
Hom_Z(A;A_i\otimes_ZA_j)$, $i,j\in{\mathbb Z}_2$, such that
$\widehat{g}_{\gamma\beta}(p)\widehat{g}^{\alpha\beta}(p)=\delta^\alpha_\gamma\in\widehat{A}$,
$\widehat{g}^{\alpha\beta}(p)\widehat{g}_{\gamma\beta}(p)=\delta^\alpha_\gamma\in
Hom_Z(A\otimes_ZA;A\otimes_ZA)$.\end{proposition}

In fact we have the following definition.

\begin{definition}
Let $E_k\subset JD^k(W)$ be a formally integrable and completely
integrable PDE the fiber bundle $\pi:W\to M$, and let
$V=D^ks(M)\subset E_k$ be a regular smooth solution of $E_k$. Let
$\xi:M\to E_k[s]$ be the general solution of $E_k[s]$. Let us assume
that there is an Euclidean structure on the fiber of $E[s]\to M$.
Let $(\psi:\mathbb{R}\times N\to N; i:N\to M)$ be a quantum frame
\cite{PRA15, PRA21, PRA22}. Then, we say that $V$ is {\em average
asymptotic stable}, with respect to the quantum frame, if the
function of time $\mathfrak{p}[i](t)$ defined by the formula:
\begin{equation}\label{average-square-perturbation}
    \mathfrak{p}[i](t)=\frac{1}{2 vol(B_t)}\int_{B_t}i^*\xi^2\hskip 3pt \eta
\end{equation}
has the following behaviour:
$\mathfrak{p}[i](t)=\mathfrak{p}[i](0)e^{-ct}$ for some real number
$c>0$. Here $B_t\equiv N_t\bigcap supp(i^*\xi^2)$, where
$N=\bigcup_{t\in T}N_t$, is the fiber structure of $N$, over the
proper-time of the quantum frame. We call $\tau_0=1/c_0$ the {\em
characteristic stability time} of the solution $V$. If
$\tau_0=\infty$ it means that $V$ is average instable.\footnote{In
the following, if there are not reasons of confusion, we shall call
also stable solution a smooth regular solution of a PDE $E_k\subset
JD^k(W)$ that is average asymptotic stable.}
\end{definition}

We have the following criterion of average asymptotic stability.
\begin{theorem}{\em(Criterion of average asymptotic stability).}\label{criterion-average-asymptotic-stability}
A regular global smooth solution $s$ of $E_k$ is average stable,
with respect to the quantum frame $(\psi:\mathbb{R}\times N\to N;
i:N\to M)$, if the following conditions are satisfied:\footnote{The large cuspidated brackets $<,>$ denote expectation value.}
\begin{equation}\label{stability-inequality}
   <\mathop{\mathfrak{p}}\limits^{\bullet}[i](t)>\le c\hskip 3pt<\mathfrak{p}[i](t)>,\quad c\in\mathbb{R}^+, \forall t.
\end{equation}
where

\begin{equation}\label{average-square-perturbation}
    \mathfrak{p}[i](t)=\frac{1}{2\hskip 2pt vol(B_t)}\int_{B_t}i^*\xi^2\eta
\end{equation}
and
\begin{equation}\label{average-square-perturbation-rate}
  \mathop{\mathfrak{p}}\limits^{\bullet}[i](t)=\frac{1}{2\hskip 2pt vol(B_t)}
  \int_{B_t}\left(\frac{\delta i^*\xi^2}{\delta t}\right)\eta
 =\frac{1}{vol(B_t)}\int_{B_t}\left(\frac{\delta i^*\xi}{\delta t}.i^*\xi\right)\hskip 3pt \eta.
\end{equation}
Here $i^*\xi$ represents the integrable general solution of the
linearized equation $E_k[s|i]$ of $E_k$ at the solution $s$, and
with respect to the quantum frame. Let us denote by $c_0$ the
infimum of the positive constants $c$ such that inequality
{\em(\ref{stability-inequality})} is satisfied. Then we call
$\tau_0=1/c_0$ the {\em characteristic stability time} of the
solution $V$. If $\tau_0=\infty$ means that $V$ is
unstable.\footnote{$\tau_0$ has just the physical dimension of a
time.}

Furthermore, Let $s$ be a smooth regular solution of a formally
quantum integrable and completely quantum superintegrable quantum
super PDE $\hat E_k\subset J\hat{\it D}^k(W)$, where $\pi:W\to M$.
There exists a differential operator $\mathcal{P}[s|i](\xi)$, on
$\bar\pi:\hat E[s|i]\equiv i^*(s^*vTW)\to N$, canonically associated
to the solution $s$, and with respect to the quantum frame, such
that $s$ is average stable in $\hat E_k$, or in some suitable
prolongation $(\hat E_k)_{+h}$, $k+h=2s\ge k$, if the following
conditions are verified:

{\em(i)}  $\mathcal{P}[s|i](\xi)$ is self-adjoint (or symmetric) on
the constraint
\begin{equation}\label{constraint}
    (\hat E_k)_{(+r)}[s|i]\subset J\hat {\it D}^{k+r}(\hat E[s|i]),
\end{equation}
for some $r\ge 0$.

{\em(ii)} The smallest eigenvalue
$\overline{\lambda}_1=\overline{\lambda}_1(t)$ of
$\mathcal{P}[s|i](\xi)$ is positive for any $t\in T$ and lower
bounded: $\overline{\lambda}_1\ge\lambda_1>0$.

Furthermore, average stability can be also translated into a variational problem constrained by $(\hat E_k)_{(+h)}[s]$,
for some $h\ge 0$, such that $k+h=2s$.
\end{theorem}

\begin{proof}
We shall use Theorem \ref{deformation} and the following lemma.

\begin{lemma}{\em(Gr\"onwall's lemma)\cite{GRON}}\label{Gronwall-lemma}
Suppose $f(t)$ is a real function whose derivative is bounded
according to the following inequality: $\frac{df}{dt}\le
g(t)f+h(t)$, for some real functions $g(t)$ and $h(t)$. Then, $f(t)$
is bounded pointwise in time according to $f(t)\le
f(0)e^{G(t)}+\int_{[0,t]}e^{G(t-s)}h(s) ds$, where
$G(t)=\int_{[0,t]}g(r)dr$.
\end{lemma}

Then a sufficient condition for the solution $V$ stability, with
respect to the quantum frame, is that inequality
(\ref{stability-inequality}) should be satisfied. In fact it is
enough to use Lemma \ref{Gronwall-lemma} with $g(t)=-c$ and
$h(t)=0$, to have $\mathfrak{p}[i](t)=\mathfrak{p}[i](0)e^{-ct}$.

Furthermore, condition (\ref{stability-inequality}) is satisfied iff
\begin{equation}\label{infimum-condition}
I[\xi|i]\equiv\left<-2\int_{B_t}<\frac{\delta i^*\xi}{\delta
t}+ci^*\xi,i^*\xi>\eta\right>\ge 0,
\end{equation}
for some constant $c>0$ and for any integrable solution $i^*\xi$ of
$\hat E_k[s|i]$. So the problem is converted to study the spectrum
of the differential operator, $\mathcal{P}[s|i](\xi)\equiv
\frac{\delta i^*\xi}{\delta t}$, on $\bar\pi:\hat E[s|i]\to N$,
constrained by $(\hat E_k)_{(+r)}[s]$, for some $r\ge 0$, since
$P[s|i](\xi)$ is of order $\ge k$. If this is self-adjoint, (or
symmetric), it follows that it has real spectrum and the stability
of the solution is related to the sign of the smallest
eigenvalue.\footnote{Really it should be enough to require that
$\mathcal{P}[s|i]$ is a symmetric operator in the Hilbert space
$\mathcal{H}_t$, canonically associated to $\hat E[s]|_{B_t}$. In
fact the point spectrum $Sp(A)_p$ of a symmetric linear operator $A$
on $\mathcal{H}_t$ is real $Sp(A)_p\subset \mathbb{R}$. (This is
true also for its continuous spectrum: $Sp(A)_c\subset \mathbb{R}$.)
In our case it is enough that $\mathcal{P}[s|i]$ should symmetric on
the space of $\hat E_k[s|i]$ solutions. However, it is well known in
functional analysis that every symmetric operator has a self-adjoint
extension, on a possibly larger space \cite{DU-SH}.} If such an
eigenvalue $\overline{\lambda}_1(t)$ is positive, $\forall t\in T$,
and $\lambda_1=\inf_{t\in T}>0$, then the ratio
$<-\mathop{\mathfrak{p}}\limits^{\bullet}[i](t)>/<\mathfrak{p}[i](t)>$
is higher than a positive constant, hence the solution $s$ is
average stable. In fact, we get

\begin{equation}
\left\{
\begin{array}{ll}
  -\mathop{\mathfrak{p}}\limits^{\bullet}[i](t)-\lambda_1\mathfrak{p}[i](t)& =
    \int_{B_t}[(\mathcal{P}[s|i](\xi).\xi)-\lambda_1i^*\xi^2]\eta\\
  & =\int_{B_t}(\overline{\lambda}_1(t)-\lambda_1)i^*\xi^2\eta=
  (\overline{\lambda}_1(t)-\lambda_1)\int_{B_t}i^*\xi^2\eta\ge
0,\\
\end{array}\right.
\end{equation}
for any $t\in T$. Thus we have also
\begin{equation}
\frac{-\mathop{\mathfrak{p}}\limits^{\bullet}[i](t)}{\mathfrak{p}[i](t)}\ge\lambda_1>
0,\quad \forall t\in T.
\end{equation}
So condition (\ref{stability-inequality}) is satisfied, hence the
solution $s$ is average stable. In order to complete the proof of
Theorem \ref{criterion-average-asymptotic-stability}, let us
emphasize that in general $\mathcal{P}[s|i](\xi)-ci^*\xi$ is not
identified with the Euler-Lagrange operator for some Lagrangian. In
fact, in general, the differential order of such an operator does
not necessitate to be even. By the way, since $\hat E_k$ is assumed
formally quantum integrable and completely quantum superintegrable,
we can identify any smooth solution $V\subset \hat E_k$, with its
$h$-prolongation $V^{(h)}\subset J\hat{\it D}^{k+h}(W)$, such that
$k+h=2s$. Thus the problem of average stability can be translated in
a variational problem, constrained by solutions of $(\hat
E_k)_{+(h)}[s|i]$.
\begin{equation}\label{constrained-Euler-Lagrange-solution-perturbation-operator}
\left\{-\frac{\delta i^*\xi}{\delta t}=2\lambda(t)i^*\xi, \quad
F_\alpha^I[s|i]=0,\quad 0\le|\alpha|\le h,\hskip 2pt
k+h=2s\right\}_{t=const}
\end{equation}
on the fiber bundle $\bar\pi:\hat E[s|i]\equiv i^*(s^*vTW)\to N$.
Here $F^I[s|i]=0$ are the equations encoding $\hat E_k[s|i]$. This
can be made not only locally but also globally. In fact one has the
following lemmas. (See for the terminology \cite{PRA18, PRA32} and
references quoted there.)

\begin{lemma}{\em\cite{PRA32}}\label{constrained-variational-problems}
Let $\pi:W\to M$, a fiber bundle in the category $\mathfrak{Q}_S$,
$\dim_A M=m|n $, $\dim_BW=(m|n,r|s)$. Let $L:\hat
J^k_{m|n}(W)\to\widehat{A}$ be a $k$-order quantum Lagrangian
function and $\theta\equiv L\eta\in\widehat{\Omega}^{m+n}(\hat
J_{m|n}^k(W))$, locally given by
$\theta=L\widehat{dx^1\triangle\cdots\triangle
dx^{m+n}}=l\hat\mu_*\circ dx^1\triangle\cdots\triangle dx^{m+n}$,
where $(x^\alpha,y^j)$ are fibered quantum coordinates on $W$, and
$\hat\mu_*:\dot T^{m+n}_0(A)\to A$ is the $Z$-homomorphism induced
by the product on $A$. Then, extremals for $\theta$, constrained by
$\hat E_k$, are solutions $f:X\to \hat E_k$, with $X$ a quantum
supermanifold of dimension $m|n$ with respect to $A$, such that the
following condition is satisfied:
\begin{equation}\label{variationa-action-integral}
\left<\sum_{1\le j\le r+s}\left[\nu^j\sum_{0\le |i|\le
k}(-1)^{|i|}\partial_i\left({{\partial L}\over{\partial
y^j_i}}\right)\right]\eta,X\right>=0,
\end{equation}

for any $\nu=\nu^j\partial y_j$, solution of the linearized equation
of $\hat E_k$ at the solution $s$. In particular, if $\hat E_k=\hat
J_{m|n}^k(W))$, then extremals are solutions of the following
equation ({\em Euler-Lagrange equation}):

\begin{equation}
\hat E[\theta]\subset \hat
J^{2k}_{m|n}(W):\quad\left\{\sum_{0\le|i|\le
k}(-1)^{|i|}\partial_i\left({{\partial L}\over{\partial
y^j_i}}\right)=0\right\}_{1\le j\le r+s}.
\end{equation}
\end{lemma}

This completes the proof. \end{proof}

\begin{example}{\em(Quantum super d'Alembert equation).}\label{quantum-super-d-Alembert-equation}
Let $ A=A_0\oplus A_1$ be a quantum superalgebra with $ Z=Z(A)$
Noetherian and $ \mathbb{K}=\mathbb{R}$. Let us consider the
following trivial fiber bundle $ \pi:W\equiv A^3\to A^2\equiv M$
with quantum coordinates $ (x,y,u)\mapsto(x,y)$. Then, the quantum
super d'Alembert equation, $ \widehat{(d'A)}\subset J\hat
D^2(W)\subset\hat J^2_2(W)$, is defined by means of the following $
\mathop{\widehat{A}}\limits^2$-valued $ Q^\infty_w$-function $
F\equiv uu_{xy}-u_xu_y:J\hat
D^2(W)\to\mathop{\widehat{A}}\limits^2\subset
\mathop{\widehat{A}}\limits^\infty$. By forgetting the
$\mathbb{Z}_2$-gradiation of $ A$, we get the same situation just
considered in \cite{PRA15}. So, $ \widehat{(d'A)}\subset J\hat
D^2(W)$ is just a formally quantum integrable quantum super PDE of
dimension $ (3,2,2)$ over the quantum algebra $ B\equiv
A\times\mathop{\widehat{A}}\limits^1\times\mathop{\widehat{A}}\limits^2$,
in the open quantum submanifold $ u\not=0$, and also completely
quantum integrable there. On the other hand, by considering that $
M$ has a natural structure of quantum supermanifold of dimension $
(2|2)$ over $ A$, it follows also that for any initial condition,
i.e., any point $ q\in\widehat{(d'A)}\setminus u^{-1}(0)$, passes a
quantum supermanifold of dimension $ (2|2)$ over $ A$, solution of $
\widehat{(d'A)}$. Therefore, $ \widehat{(d'A)}\subset\hat
J^2_{2|2}(W)$ is completely quantum superintegrable in the quantum
supermanifold $ u\not=0$ too. We can also state that
$\widehat{(d'A)}$ is completely quantum superintegrable, as it is
algebraic in the open set $\widehat{(d'A)}\setminus u^{-1}(0)$. Thus
$\widehat{(d'A)}$ is an extended crystal PDE. With respect to the
commutative exact diagram in
(\ref{Reinhart-bordism-groups-relation}) we get the following exact
commutative diagram

\begin{equation}\label{d-Alembert-Reinhart-bordism-groups-relation}
\xymatrix{0\ar[r]&K^{\widehat{(d'A)}}_{1|1;2}\ar[r]&\Omega_{1|1}^{\widehat{(d'A)}}\ar[r]&
\mathop{\Omega}\limits_c{}_{2}^{\widehat{(d'A)}}\ar[d]\ar[r]\ar[dr]&0&\\
&0\ar[r]&K^\uparrow_{2}\ar[r]&\Omega^\uparrow_{2}\ar[r]&
\mathbb{Z}_{2}\ar[r]& 0\\}
\end{equation}
Therefore, the crystal group of $\widehat{(d'A)}$ is
$G(2)=\mathbb{Z}^2\times\mathbb{Z}_2\equiv p2$, ($p2$ is its usual
crystallographic notation), and its crystal dimension is $2$.

Furthermore, according to Theorem \ref{main} we get that
$\widehat{(d'A)}\setminus\{u=0\}$ is functionally stable. From
Theorem \ref{finite-stable-extended-crystal-PDE} we get that
$\widehat{(d'A)}$ is a stabilizable quantum extended crystal PDE
with associated stable quantum extended crystal PDE
${}^{(S)}\widehat{(d'A)}=\widehat{(d'A)}_{+\infty}$. (The symbol of
$\widehat{(d'A)}$ is not zero, thus smooth global solutions are in
general unstable into finite times in $\widehat{(d'A)}$.) Moreover,
we get that the weak integral $ (1|1)$-bordism group of $
\widehat{(d'A)}$ is trivial. In fact, we have: $
\Omega^{\widehat{(d'A)}}_{1|1,w}\cong{}^A\underline{\Omega}_{1|1}(W)\cong(Z\otimes_{\mathbb{K}}H_1(W;\mathbb{K}))
\bigoplus(A_1\otimes_{\mathbb{K}}H_1(W;\mathbb{K}))=0$. Thus $
\widehat{(d'A)}$ is an extended $0$-crystal.

Under the full-admissibility hypothesis, i.e., by considering
admissible closed smooth integral quantum supermanifolds of
dimension $(m-1|n-1)$, on the which all integral characteristic
quantum supernumbers are zero, we can consider $ \widehat{(d'A)}$ a
$0$-crystal quantum super PDE, hence for such fully admissible
Cauchy data, the existence of global smooth solutions of $
\widehat{(d'A)}$ is assured. Finally, applying Theorem
\ref{criterion-average-asymptotic-stability} we get further
informations on the asymptotic average stability of $(d'A)$
solutions.
\end{example}

\begin{example}{\em(Quantum super Navier-Stokes equation).}\label{quantum-super-Navier-Stokes-equation}
In some previous works we have considered the Navier-Stokes equation
for quantum (super)fluids as a quantum (super)PDE. (See
Refs.\cite{PRA14, PRA15, PRA19}.) Now, we can extend such
considerations to stability of such equations. By using results in
\cite{PRA15, PRA22} we can prove that when $ A$ has Noetherian
centre $ Z=Z(A)$, $ (NS)$ contains a formally quantum integrable
quantum super PDE $\widehat{(NS)}$ that is completely quantum
superintegrable. So $(NS)$ is not functionally stable, but
$\widehat{(NS)}$ is so. This last equation is also an extended
crystal quantum super PDE with its infinity prolongation
$\widehat{(NS)}_{+\infty}$ as stable extended crystal quantum super
PDE. Furthermore, one can prove that the weak integral $
(3|3)$-bordism group $
\Omega_{3|3,w}^{\widehat{(NS)}}=0=\Omega_{3|3,s}^{\widehat{(NS)}}$.
This means that $\widehat{(NS)}$ is an extended $0$-crystal quantum
super PDE. However, it is not a $0$-crystal quantum super PDE. By
the way, whether we adopt the full admissibility hypothesis, then
$\widehat{(NS)}$ becomes a $0$-crystal quantum super PDE and this is
enough to state the existence of global smooth solutions of the
quantum super PDE $ (NS)$ for such admissible smooth boundary
condition contained int $ \widehat{(NS)}$.

With respect to the commutative exact diagram in
(\ref{Reinhart-bordism-groups-relation}) we get the following exact
commutative diagram

\begin{equation}\label{Navier-Stokes-Reinhart-bordism-groups-relation}
\xymatrix{0\ar[r]&K^{ \widehat{(NS)}}_{3|3;2}\ar[r]&\Omega_{3|3}^{
\widehat{(NS)}}\ar[r]&
\mathop{\Omega}\limits_c{}_{6}^{ \widehat{(NS)}}\ar[d]\ar[r]\ar[dr]&0&\\
&0\ar[r]&K^\uparrow_{6}\ar[r]&\Omega_6^\uparrow\ar[r]&
\mathbb{Z}_2\bigoplus\mathbb{Z}_2\bigoplus\mathbb{Z}_2\ar[r]& 0.\\}
\end{equation}
Therefore, the crystal group and the crystal dimension of $
\widehat{(NS)}$ are the same ones of $\widehat{(YM)}$.

Finally, applying Theorem
\ref{criterion-average-asymptotic-stability} we get further
informations on the asymptotic average stability of $\widehat{(NS)}$
solutions.
\end{example}

\section{\bf QUANTUM EXTENDED CRYSTAL SINGULAR PDE's}
\vskip 0.5cm

In this section we shall consider singular quantum super PDE's
extending our previous theory of singular PDE's,\footnote{See Refs.\cite{PRA13, PRA23}. See also \cite{AG-PRA3} where some interesting applications are considered.} i.e., by
considering singular quantum (super) PDE's as singular quantum
sub-(super)manifolds of jet-derivative spaces in the category
$\mathfrak{Q}$ or $\mathfrak{Q}_S$. In fact, our previous formal
theory of quantum (super) PDE's works well on quantum smooth or
quantum analytic submanifolds, since these regularity conditions are
necessary to develop such a theory. However, in many mathematical
problems and physical applications, it is necessary to work with
less regular structures, so it is useful to formulate a general
geometric theory for such more general quantum PDE's in the category
$\mathfrak{Q}_S$. Therefore, we shall assume that quantum singular
super PDE's are subsets of jet-derivative spaces where are presents
regular subsets, but also other ones where the conditions of
regularity are not satisfied. So the crucial point to investigate is
to obtain criteria that allow us to find existence theorems for
solutions crossing ''singular points'' and study their stability
properties.

The main result of this section is Theorem
\ref{main-quantum-singular1} that relates singular integral bordism
groups of singular qunatum PDE's to global solutions passing through
singular points. Some example are explicitly considered.

Let us, now, first begin with a generalization of algebraic
formulation of quantum super PDE's, starting with the following
definitions. (See also Refs.\cite{PRA13, PRA19, PRA20, PRA21}.)

\begin{definition}
The {\em general category of quantum superdifferential equations},
${\frak Q}_S^{\frak E}$, is defined by the following: {\em 1)}
$\hat{\frak B}\in Ob({\frak Q}_S^{\frak E})$ iff $\hat{\frak B}$ is
a filtered quantum superalgebra $\hat{\frak B}\equiv \{ \hat{\frak
B}_i\} , \hat{\frak B}_i\subset \hat{\frak B}_{i+1}$, such that in
the differential calculus in the category ${\frak
Q}^{FG}_S(\hat{\frak B})$ over $\hat{\frak B}$ is defined a natural
operation ${\it C}$ that satisfies ${\it C}\hat \Omega ^1 \wedge
\hat\Omega^\bullet={\it C}\hat\Omega^\bullet$ , where
$\hat\Omega^i\equiv \hat{\frak B}\wedge \cdots_i\cdots
\wedge\hat{\frak B}$ are the representative objects of the functor
$\hat D_i$ in the category ${\frak Q}^{FG}_S(\hat{\frak B})$ over
$\hat{\frak B}$, where $\hat D_i \equiv\hat D\cdots_i\cdots\hat D$,
being $\hat D(P)$ the $\hat{\frak B}$-module of all quantum
superdifferentiations of algebra $\hat{\frak B}$ with values in
module $P$. Furthermore, $\hat\Omega^\bullet\equiv\bigoplus_{i\ge
0}\hat\Omega^i$, $\hat\Omega^0\equiv A$. {\em 2)} $f\in Hom({\frak
Q}_S^{\frak E})$ iff $f$ is a homomorphism of filtered quantum
superalgebras preserving operation ${\it C}$.
\end{definition}

\begin{remark}
In practice we shall take $\hat{\frak B}\equiv \{ \hat{\frak
B}_i\equiv Q^\infty _w(M_i;A)\}$, where $M_i$ is a quantum
supermanifold and $A$ is a quantum superalgebra. Then, we have a
canonical inclusion: $j_i:M_i\rightarrow Sp(\hat{\frak B}_i),
x\mapsto j_i(x)\equiv e_x\equiv$ evaluation map at $x\in M_i$. To
the inclusion $\hat{\frak B}_i\subset \hat{\frak B}_{i+1}$
corresponds the quantum smooth map $M_{i+1}\rightarrow M_i$. So we
set $M_\infty =\mathop{lim}\limits_{\leftarrow }M_i$. One has
${\overline M}_\infty =Sp(\hat{\frak B}_\infty )$. However, as
$M_\infty $ contains all the ''nice'' points of $Sp(\hat{\frak
B}_\infty )$, we shall use the space $M_\infty $ to denote an object
of the {\em category of quantum superdifferential equations}.
\end{remark}

\begin{definition}
The {\em category of quantum superdifferential equations}
$\underline{\frak Q}^{\frak E}_S$ is defined by the Frobenius full
quantum superdistribution $\widehat C(X)\subset \widehat TX\equiv
Hom_{Z}(A;TX)$, which is locally the same as
 $\widehat{\mathbf{E}}_\infty$, i.e., the
Cartan quantum superdistribution of $\hat E_\infty $ for some
quantum super PDE $\hat E_k\subset J\hat D^k(W)$. We set: $s\dim
X\equiv \dim\widehat C(X)=(m+n|m+n)$, i.e., the {\it Cartan quantum
superdimension} of $X\in Ob(\underline{\frak Q}^{\frak E}_S)$. $f\in
Hom(\underline{\frak Q}^{\frak E}_S)$ iff it is a quantum
supersmooth map $f:X\rightarrow Y$, where $X,Y\in
Ob(\underline{\frak Q}^{\frak E}_S)$, such that conserves the
corresponding Frobenius full superdistributions: $\widehat
T(f):\widehat C(X)\rightarrow \widehat C(Y)$,
           $f\in Hom_{\underline{\frak Q}^{\frak E}_S}(X,Y)$,
           $sdimX=(m+n|m+n)$, $sdimY=(m'+n'|m'+n')$, $srankf=(r|s)=\dim(\hat T(f)_x
           (\widehat C(X)_x))$, $x\in X$.
Then the fibers $f^{-1}(y)$, $y\in im(f)\subset Y$, are
$(m+n-r|m+n-s)$-quantum superdimensional objects of
$\underline{\frak Q}^{\frak E}_S$. {\em Isomorphisms} of
$\underline{\frak Q}^{\frak E}_S$: quantum supermorphisms with
fibres consisting of separate points. {\em Covering maps} of
$\underline{\frak Q}^{\frak E}_S$: quantum supermorphims with
zero-quantum superdimensional fibres.
\end{definition}

\begin{example}{\em(Some quantum singular PDE's).}\label{examples-quantum-singular-PDEs}
$$\begin{tabular}{|l|l|}
\hline \multicolumn{2}{|c|}{\bsmall Tab.5 - Examples of quantum singular PDE's}\\
\multicolumn{2}{|c|}{\bsmall defined by differential polynomials}\\
\hline\hline {\rsmall Name}&{\rsmall Singular PDE} \\
\hline
{\rsmall PDE with node and triple point}&$\scriptstyle p_1\equiv(u^1_x)^4+(u^2_y)^4-(u^1_x)^2=0$\\
$\scriptstyle \hat R_1\subset \hat J{\it D}(E)$&$\scriptstyle
p_2\equiv(u^2_x)^6+(u^1_y)^6-u^2_xu^1_y=0$\\
\hline
{\rsmall PDE with cusp and tacnode}&$\scriptstyle q_1\equiv(u^1_x)^4+(u^2_y)^4-(u^1_x)^3+(u^2_y)^2=0$\\
$\scriptstyle \hat S_1\subset \hat J{\it D}(E)$&$\scriptstyle q_2\equiv(u^2_x)^4+(u^1_y)^4-(u^2_x)^2(u^1_y)-(u^2_x)(u^1_y)^2=0$\\
\hline {\rsmall PDE with conical double point,}&$\scriptstyle
r_1\equiv(u^1)^2-(u^1_x)(u^2_y)^2=0$\\
{\rsmall double line and pinch point} &$\scriptstyle
r_2\equiv(u^2)^2-(u^2_x)^2-(u^1_y)^2=0$\\
$\scriptstyle \hat T_1\subset J{\it D}(F)$&$\scriptstyle
r_3\equiv(u^3)^3+(u_y^3)^3+(u^2_x)(u^3_y)=0$\\
\hline \multicolumn{2}{l}{\rsmall$\scriptstyle\pi:E\equiv A^4\to
A^2,\quad(x,y,u^1,u^2)\mapsto(x,y)$.\quad
$\scriptstyle\bar\pi:F\equiv A^5\to A^2,\quad(x,y,u^1,u^2,u^3)\mapsto(x,y)$.}\\
\multicolumn{2}{l}{\rsmall$\scriptstyle\mathfrak{a}\equiv<p_1,p_2>\subset
\hat{\frak B}_1, \mathfrak{b}\equiv<q_1,q_2>\subset \hat{\frak B}_1,
\mathfrak{c}\equiv<r_1,r_2,r_3>\subset
\hat{\frak P}_1$.}\\
\end{tabular}$$

In Tab.5 we report some quantum singular PDE's having  some
algebraic singularities.  For the first two equations these are
quantum singular PDE's of first order defined on the quantum fiber
bundle $\pi:E\to M$, with $E\equiv A^4$, $M\equiv A^2$, where $A$ is
a quantum algebra. Then $\hat J{\it D}(E)\cong
B_1^{4,4}=A^4\times\widehat{A}^4$. Furthermore, for the third
equation one has the quantum fiber bundle $\bar\pi:F\to M$, with
$F\equiv A^5$, $M\equiv A^2$, and $\hat J{\it D}(F)\cong
B_1^{5,6}=A^5\times\widehat{A}^6$. We follow our usual notation
introduced in some previous works on the same subject. In particular
for a given quantum (super)algebra $A$, we put
\begin{equation}\label{dot-tensor-products}
\left\{\begin{array}{l}
\dot T^r_0(H)\equiv
\underbrace{H\otimes_Z\cdots\otimes_ZH}_r,\quad
r\ge0\\
\mathop{\widehat{A}}\limits^{r}\equiv Hom_Z(\dot T^r_0(A);A),\quad
r\ge0\\
\mathop{\widehat{A}}\limits^{0}\equiv Hom_Z(\dot T^0_0(A);A)\equiv
Hom_Z(A;A)\equiv \widehat{A}\\
\end{array}\right.
\end{equation}
with $Z$ the centre of $A$ and $H$ any $Z$-module.
Furthermore, we denote also by $\dot S^r_0(H)$ and $\dot
\Lambda^r_0(H)$ the corresponding symmetric and skewsymmetric
submodules of $\dot T^r_0(H)$. To the ideals $\mathfrak{a}\equiv
<p_1,p_2>\subset \hat{\frak B}_1$,
$\mathfrak{b}\equiv<q_1,q_2>\subset \hat{\frak B}_1$ and
$\mathfrak{c}\equiv<r_1,r_2,r_3>\subset \hat{\frak P}_1$, where
$\hat{\frak B}_1\equiv Q^\infty_w(\hat J{\it D}(E),B_2)$, with
$B_2\equiv A\times \widehat{A}\times\mathop{\widehat{A}}\limits^2$, and $\hat{\frak P}_1\equiv
Q^\infty_w(\hat J{\it D}(F),B_2)$, one associates the corresponding
algebraic sets $\hat R_1=\{q\in B_1^{4,4}|f(q)=0, \forall f\in
\mathfrak{a}\}\subset B_1^{4,4}$, $\hat S_1=\{q\in B_1^{4,4}|f(q)=0,
\forall f\in \mathfrak{b}\}\subset B_1^{4,4}$ and $\hat T_1=\{q\in
B_1^{5,6}|f(q)=0, \forall f\in \mathfrak{c}\}\subset B_1^{5,6}$.

Let us consider in some details, for example, the first equation in
Tab.5. There the node and the triple point refer to the singular
points in the planes $(u^1_x,u^2_y)$ and $(u^2_x,u^1_y)$
respectively, with respect to the $\mathbb{R}$-restriction. However, the equation $\hat R_1$ has a
set $\Sigma(\hat R_1)\subset \hat R_1$ of singular points that contains:
\begin{equation}\label{singular-points-eq-1-tab3}
    \Sigma(\hat R_1)_{0}\equiv\left\{
      q_0=(x,y,u^1,u^2,0,0,0,0)\right\}\cong A^4
    \subset\hat R_1.
\end{equation}
$\Sigma(\hat R_1)$ is, in general, larger than $\Sigma(\hat R_1)_{0}$. In fact the jacobian $(j(F)_{ij})$, $i=1,2$, $j=1,\cdots,8$, with
$(F_i)\equiv(p_1,p_2):\hat J{\it D}(E)\to B_2$, is given by the
following matrix with entries in the quantum algebra $B_2$:
\begin{equation}\label{jacobian}
(j(F)_{ij})=\left(
  \begin{array}{cccccccc}
    0 & 0 & 0 & 0 & 2u^1_x[2(u^1_x)^2-1] & 0 & 0 & 4(u^2_y)^3 \\
    0 & 0 & 0 & 0 & 0 & 6(u^1_y)^5-u_x^2 &6(u^2_x)^5-u_y^1  & 0 \\
  \end{array}
\right).
\end{equation}
Since, in general, $A$ can have a non-empty set of zero-divisors, in order $\hat Y_1\equiv \hat R_1\setminus \Sigma(\hat R_1)$ should represent $\hat R_1$ without singular points, i.e., in order to apply the implicit quantum function theorem (see Theorem 1.38 in \cite{PRA15}), it is enough to take the points $q\in\hat R_1$, where there are $2\times 2$ minors in {\em(\ref{jacobian})} with invertible determinant. This allows us to identify an open submanifold $\hat X_1$ in $\hat J{\it D}(E)$. Then, we get $\hat Y_1=\hat R_1\bigcap\hat X_1$. Thus we can call $\hat Y_1\subset\hat R_1$ the {\em regular component} of $\hat R_1$. This submanifold is not empty, since it contains the regular part of the $\mathbb{R}$-restriction of $\hat R_1$. Let us define the following subsets of $\hat R_1$:
\begin{equation}\label{subsets-eq-1-tab3}
    \left\{
    \begin{array}{l}
      \hat Y_1\equiv\hat R_1\setminus \Sigma(\hat R_1)\subset \hat R_1\\
      {}_2\hat R_1\equiv\left\{q\in\hat R_1|u^1_y(q)=0,u^2_x(q)=0\right\}\subset\hat R_1\\
{}_3\hat R_1\equiv\left\{q\in\hat R_1|u^1_x(q)=0,u^2_y(q)=0\right\}\subset\hat R_1.\\
\end{array}
    \right.
\end{equation}
One has ${}_2\hat R_1\bigcap{}_3\hat R_1\not=\varnothing$, $\hat Y_1\bigcap{}_2\hat R_1\not=\varnothing$,  $\hat Y_1\bigcap{}_3\hat R_1\not=\varnothing$. Furthermore the set of singular points $\Sigma({}_2\hat R_1)$ (resp. $\Sigma({}_3\hat R_1)$) of ${}_2\hat R_1$ (resp. ${}_3\hat R_1$) is contained in $\Sigma(\hat R_1)$ and contains $\Sigma({}_2\hat R_1)_0\equiv{}_2\hat R_1\bigcap\Sigma(\hat R_1)_0\cong A^4$ (resp. $\Sigma({}_3\hat R_1)_0\equiv{}_3\hat R_1\bigcap\Sigma(\hat R_1)_0\cong A^4$). We can write:
\begin{equation}\label{split-subsets-eq-1-tab3}
\hat R_1=\hat Y_1\bigcup\Sigma(\hat R_1)\cong{}_2\hat R_1\times{}_3\hat R_1\cong[\hat Y_2\bigcup\Sigma({}_2\hat R_1)]\times[\hat Y_3\bigcup\Sigma({}_3\hat R_1)]\subset\hat J{\it D}(E),
\end{equation}
where $\hat Y_2\equiv {}_2\hat R_1\setminus\Sigma({}_2\hat R_1)$ (resp. $\hat Y_3\equiv {}_3\hat R_1\setminus\Sigma({}_3\hat R_1$). $\hat Y_1$ is a formally quantum integrable and completely quantum
integrable quantum PDE of first order. (For the theory of formal
integrability of quantum PDE's, see Refs.\cite{PRA7, PRA11, PRA19,
PRA20, PRA21}.) In fact $\hat Y_1$ and its prolongations $(\hat
Y_1)_{+r}\subset \hat J{\it D}^{r+1}(E)$, are subbundles of $\hat
J{\it D}^{r+1}(E)\to\hat J{\it D}^{r}(E)$, $r\ge 0$. One can also
see that the canonical maps $\pi_{r+1,r}:(\hat Y_1)_{+r}\to(\hat
Y_1)_{+(r-1)}$, are surjective mappings. For example, for $r=1$, one
has the following isomorphisms:
\begin{equation}\label{isomorphisms-eq-1-tab3}
    \left\{
    \begin{array}{l}
    \dim_{B_1}\hat J{\it D}(E)=(4,4)\\
    \\
      \hat Y_1\cong A^4\times\widehat{A}{}^2\Rightarrow\dim_{B_1}\hat Y_1=(4,2)\\
\hat J{\it D}^2(E)\cong
A^4\times\widehat{A}{}^4\times(\mathop{\widehat{A}}\limits^{2})^8\Rightarrow\dim_{B_2}\hat J{\it D}^2(E)=(4,4,8)\\
      (\hat Y_1)_{+1}\cong
      A^4\times\widehat{A}{}^2\times(\mathop{\widehat{A}}\limits^{2})^4
      \Rightarrow\dim_{B_2}(\hat Y_1)_{+1}=(4,2,4)\\
      Hom_Z(\dot S^2_0(T_pM);vT_{\bar
      q}E)\cong(\mathop{\widehat{A}}\limits^{2})^8\Rightarrow\dim_{B_2}Hom_Z(\dot S^2_0(T_pM);vT_{\bar
      q}E)=(0,0,8)\\
      ((\hat g_1)_{+1})_{q\in\hat Y_1}\cong (\mathop{\widehat{A}}\limits^{2})^4 \Rightarrow\dim_{B_2}
((\hat g_1)_{+1})_{q\in\hat Y_1}=(0,0,4)\\
\\
\left[\dim_{B_2}(\hat Y_1)_{+1}\right]=\left[\dim_{B_2}\hat
Y_1\right]+\left[\dim_{B_2}((\hat
g_1)_{+1})_{q\in\hat Y_1}\right].\\
\end{array}
    \right.
\end{equation}
Therefore, $(\hat Y_1)_{+1}\to(\hat Y_1)$, is surjective, and by
iterating this process, we get that also the mappings $(\hat
Y_1)_{+r}\to(\hat Y)_{+(r-1)}$, $r\ge 0$, are surjective. We put
$(\hat Y_1)_{+(-1)}\equiv E$. Thus $\hat Y_1$ is a quantum regular
quantum PDE, and under the hypothesis that $A$ has a Noetherian
centre, it follows that $\hat Y_1$ is quantum $\delta-$regular too.
Then, from Theorem 3.4 in \cite{PRA11}, it follows that $\hat Y_1$ is
formally quantum integrable. Since it is quantum analytic, it is
completely quantum integrable too.

\end{example}

\begin{definition}\label{quantum-extended-crystal-singular-super-PDE}
We define {\em quantum extended crystal singular super PDE}, a
singular quantum super PDE $\hat E_k\subset \hat J^k_{m|n}(W)$ that
splits in irreducible components $\hat A_i$, i.e., $\hat
E_k=\bigcup_i \hat A_i$, where each $\hat A_i$ is a quantum extended
crystal super PDE. Similarly we define {\em quantum extended
$0$-crystal singular PDE}, (resp. {\em quantum $0$-crystal singular
PDE}), a quantum extended crystal singular PDE where each component
$\hat A_i$ is a quantum extended $0$-crystal PDE, (resp. quantum
$0$-crystal PDE).
\end{definition}

\begin{definition}{\em(Algebraic singular solutions of quantun singular super PDE's)}.\label{singular-algebraic-solution}
Let $\hat E_k\subset \hat J^k_{m|n}(W)$ be a quantum singular super
PDE, that splits in irreducible components $\hat A_i$, i.e., $\hat
E_k=\bigcup_i\hat A_i$. Then, we say that $\hat E_k$ admits an {\em
algebraic singular solution} $V\subset \hat E_k$, if $V\bigcap \hat
A_r\equiv V_r$ is a solution  (in the usual sense) in $\hat A_r$ for
at least two different components $\hat A_r$, say $\hat A_i$, $\hat
A_j$, $i\not=j$, and such that one of following conditions are
satisfied: {\em(a)} ${}_{(ij)}\hat E_k\equiv \hat A_i\bigcap \hat
A_j\not=\varnothing$; {\em(b)} ${}^{(ij)}\hat E_k\equiv \hat
A_i\bigcup \hat A_j$ is a connected set, and ${}_{(ij)}\hat
E_k=\varnothing$. Then we say that the algebraic singular solution $V$
is in the case {\em(a)}, {\em weak}, {\em singular} or {\em smooth},
if it is so with respect to the equation ${}_{(ij)}\hat E_k$. In the
case {\em(b)}, we can distinguish the following situations:
{\em(weak solution):} There is a discontinuity in $V$, passing from
$V_i$ to $V_j$; {\em(singular solution):} there is not discontinuity
in $V$, but the corresponding tangent spaces $TV_i$ and $TV_j$ do
not belong to a same $n$-dimensional Cartan sub-distribution of
$\hat J^k_{m|n}(W)$, or alternatively $TV_i$ and $TV_j$ belong to a
same $(m|n)$-dimensional Cartan sub-distribution of $\hat
J^k_{m|n}(W)$, but the kernel of the canonical projection
$(\pi_{k,0})_*:T\hat J^k_{m|n}(W)\to TW$, restricted to $V$ is
larger than zero; {\em(smooth solution):} there is not discontinuity
in $V$ and the tangent spaces $TV_i$ and $TV_j$ belong to a same
$(m|n)$-dimensional Cartan sub-distribution of $\hat J^k_{m|n}(W)$
that projects diffeomorphically on $W$ via the canonical projection
$(\pi_{k,0})_*:T\hat J^k_{m|n}(W)\to TW$. Then we say that a
solution passing through a critical zone {\em
bifurcate}.\footnote{Note that the bifurcation does
 not necessarily imply that the tangent planes in the points of $V_{ij}\subset
 V$ to the components $V_i$ and $V_j$, should be different.}
\end{definition}

\begin{definition}{\em(Integral bordism for quantum singular super PDE's)}.\label{integral-bordism-quantum-singular-super-PDE}
Let $\hat E_k\subset \hat J^k_{m|n}(W)$ be a quantum super PDE on
the fiber bundle $\pi:W\to M$, $\dim_B W=(m|n,r|s)$, $\dim_A M=m|n$,
$B=A\times E$, $E$ a quantum superalgebra that is also a $Z$-module,
with $Z=Z(A)$ the centre of $A$. Let $N_1, N_2\subset \hat
E_k\subset \hat J^k_{m|n}(W)$ be two
 $(m-|n-1)$-dimensional, (with respect to $A$), admissible closed integral
quantum supermanifolds. We say that $N_1$ {\em algebraic integral
bords} with $N_2$, if $N_1$ and $N_2$ belong to two different
irreducible components, say $N_1\subset \hat A_i$, $N_2\subset \hat
A_j$, $i\not=j$, such that there exists an algebraic singular
solution $V\subset \hat E_k$ with $\partial V=N_1\sqcup N_2$.

In the integral bordism group $\Omega_{m-1|n-1}^{\hat E_k}$ (resp.
$\Omega_{m-1|n-1,s}^{\hat E_k}$, resp. $\Omega_{m-1|n-1,w}^{\hat
E_k}$) of a quantum singular super PDE $\hat E_k\subset \hat
J^k_{m|n}(W)$, we call {\em algebraic class} a class
$[N]\in\Omega_{m-1|n-1}^{\hat E_k}$, (resp. $[N]\in\Omega_{m-1|n-1,
s}^{\hat E_k}$, resp. $[N]\in\Omega_{m-1|n-1}^{\hat E_k}$,), with
$N\subset A_j$, such that there exists a closed
$(m-1|n-1)$-dimensional, (with respect to $A$), admissible integral
quantum supermanifolds $X\subset \hat A_i\subset \hat E_k$,
algebraic integral bording with $N$, i.e., there exists a smooth
(resp. singular, resp. weak) algebraic singular solution $V\subset
\hat E_k$, with $\partial V=N\sqcup X$.
\end{definition}

\begin{theorem}{\em(Singular integral bordism group of quantum singular
super PDE)}.\label{main-quantum-singular1} Let $\hat
E_k\equiv\bigcup_i \hat A_i\subset \hat J^k_{m|n}(W)$ be a quantum
singular super PDE. Then under suitable conditions, {\em algebraic
singular solutions integrability conditions}, we can find (smooth)
algebraic singular solutions bording assigned admissible closed
smooth $(m-1|n-1)$-dimensional, (with respect to $A$), integral
quantum supermanifolds $N_0$ and $N_1$ contained in some component
$\hat A_i$ and $\hat A_j$, $i\not= j$.
\end{theorem}

\begin{proof}
In fact, we have the following lemmas.
\begin{lemma}\label{lemma-main-quantum-singular1}
Let $\hat E_k\equiv\bigcup_i \hat A_i\subset \hat J^k_{m|n}(W)$ be a
quantum singular super PDE with ${}_{(ij)}\hat E_k\equiv \hat
A_i\bigcap \hat A_j\not=\varnothing$. Let us assume that $\hat
A_i\subset \hat J^k_{m|n}(W)$, $\hat A_j\subset \hat J^k_{m|n}(W)$
and ${}_{(ij)}\hat E_k\subset \hat J^k_{m|n}(W)$ be formally
integrable and completely integrable quantum super PDE's with
nontrivial symbols. Then, one has the following isomorphisms:
\begin{equation}\label{singular-bordism-groups-quantum-singular-PDE}
    \Omega_{m-1|n-1,w}^{\hat A_i}\cong\Omega_{m-1|n-1,w}^{\hat A_j}\cong\Omega_{m-1|n-1,w}^{{}_{(ij)}\hat E_k}\cong
    \Omega_{m-1|n-1,s}^{\hat A_i}\cong\Omega_{m-1|n-1,s}^{\hat A_j}\cong\Omega_{m-1|n-1,s}^{{}_{(ij)}\hat E_k}.
\end{equation}

So we can find a weak or singular algebraic singular solution
$V\subset \hat E_k$ such that $\partial V=N_0\sqcup N_1$, $N_0\subset
\hat A_i$,  $N_1\subset \hat A_j$, iff $N_1\in[N_0]$.
\end{lemma}

\begin{proof}
In fact, under the previous hypotheses one has that we can apply
Theorem 2.1 in \cite{PRA10} to each component $\hat A_i$, $\hat A_j$
and ${}_{(ij)}\hat E_k$ to state that all their weak and singular
integral bordism groups of dimension $(m-1|n-1)$ are isomorphic to
$H_{m-1|n-1}(W;A)$.
\end{proof}

\begin{lemma}\label{lemma-main-quantum-singular1}
Let $\hat E_k=\bigcup_i\hat A_i$ be a quantum $0$-crystal singular
PDE. Let ${}^{(ij)}\hat E_k\equiv \hat A_i\bigcup \hat A_j$ be
connected, and ${}_{(ij)}\hat E_k\equiv \hat A_i\bigcap \hat
A_j\not=\varnothing$. Then $\Omega_{m-1|n-1,s}^{{}^{(ij)}\hat
E_k}=0$.\footnote{But, in general, it is
$\Omega_{m-1|n-1}^{{}^{(ij)}\hat E_k}\not=0$.}
\end{lemma}

\begin{proof}
In fact, let $Y\subset{}_{(ij)}\hat E_k$ be an admissible closed
$(m-1|n-1)$-dimensional closed integral quantum supermanifold, then
there exists a smooth solution $V_i\subset \hat A_i$ such that
$\partial V_i=N_0\sqcup Y$ and a solution $V_j\subset \hat A_j$ such
that $\partial V_j=Y\sqcup N_1$. Then, $V=V_i\bigcup_Y V_j$ is an
algebraic singular solution of $\hat E_k$. This solution is singular
in general.
\end{proof}

After above lemmas the proof of the theorem can be considered done
besides the algebraic singular solutions integrability conditions.
\end{proof}

\end{document}